\documentclass[11pt]{article}

\usepackage{amsthm}
\usepackage{amssymb}
\usepackage{graphicx}
\usepackage{amsmath}
\usepackage{verbatim}
\usepackage{setspace}
\usepackage{ulem}
\usepackage{textpos}
\usepackage{changepage}
\usepackage{url}
\usepackage{subcaption}
\usepackage{multirow}

\usepackage{hyperref}

\usepackage[round]{natbib}

\newtheorem{assumption}{Assumption}
\newtheorem{proposition}{Proposition}

\newtheorem{lemma}{Lemma}

\newtheorem{example}{Example}

\def\bs{\boldsymbol}
\def\tbf{\textbf}
\def\t{^{\top}}
\def\bSig{\bs{\Sigma}}

\begin{document}
\bibliographystyle{asa}

{
\singlespacing
\title{\textbf{Testing the Order of Multivariate Normal Mixture Models}
\author{Hiroyuki Kasahara\thanks{Address for correspondence: Hiroyuki Kasahara, Vancouver School of Economics, University of British Columbia, 997-1873 East Mall, Vancouver, BC V6T 1Z1, Canada.   The authors thank Pengfei Li and participants at the conference on Advances in Finite Mixture and other Non-regular Models, Guilin, China in 2018  for helpful comments and the Institute of Statistical Mathematics for the facilities and the use of SGI ICE X. This research is support by the Natural Science and Engineering Research Council of Canada and JSPS Grant-in-Aid for Scientific Research (C) No. 26380267.  }\\
Vancouver School of Economics\\
University of British Columbia\\
hkasahar@mail.ubc.ca \and Katsumi Shimotsu\\
Faculty of Economics \\
University of Tokyo\\
shimotsu@e.u-tokyo.ac.jp
}}
\maketitle
}
\begin{abstract} 
Finite mixtures of multivariate normal distributions have been widely used in empirical applications in diverse fields such as statistical genetics and statistical finance. Testing the number of components in multivariate normal mixture models is a long-standing challenge even in the most important case of testing homogeneity. This paper develops likelihood-based tests of the null hypothesis of $M_0$ components against the alternative hypothesis of $M_0 + 1$ components for a general $M_0 \geq 1$. For heteroscedastic normal mixtures, we propose an EM test and derive the asymptotic distribution of the EM test statistic. For homoscedastic normal mixtures, we derive the asymptotic distribution of the likelihood ratio test statistic. We also derive  the asymptotic distribution of the likelihood ratio test statistic and EM test statistic under local alternatives and show the validity of parametric bootstrap. The simulations show that the proposed test has good finite sample size and power properties.
\end{abstract}


\section{Introduction}

Finite mixtures of multivariate normal distributions have been widely used in empirical applications in diverse fields such as statistical genetics and statistical finance. Comprehensive surveys on theoretical properties and applications can be found, for example, \citet{lindsay95book}, \citet{mclachlanpeel00book}, and \citet{fruehwirth06book}.

In many applications of finite mixture models, the number of components is of substantial interest. In multivariate normal mixture models, however, testing for the number of components has been an unsolved problem even in the most important case of testing homogeneity. For general finite mixture models, the asymptotic distribution of the likelihood ratio test statistic (LRTS) has been derived as a functional of the Gaussian process \citep{dacunha99as, liushao03as, zhuzhang04jrssb, azais09esaim}. These results are not applicable to normal mixtures because normal mixtures have an undesirable mathematical property that invalidates key assumptions in these works \citep{chenlifu12jasa}. In particular, the normal density with mean $\mu$ and variance $\sigma^2$, $f(y;\mu,\sigma^2)$, has the property $\frac{\partial^2}{\partial \mu \partial \mu} f(y;\mu,\sigma^2) = 2\frac{\partial}{\partial \sigma^2} f(y;\mu,\sigma^2)$. This leads to the loss of ``strong identifiability'' condition introduced by \citet{chen95as}. As a result, neither Assumption (P1) of \citet{dacunha99as} nor Assumption 7 of \citet{azais09esaim} holds, and Assumption 3 of \citet{zhuzhang04jrssb} is violated, while Corollary 4.1 of \citet{liushao03as} does not hold in normal mixtures. Heteroscedastic normal mixture models have an additional problem called the infinite Fisher information problem \citep{lcm09bm} that the score of the LRTS has infinite variance if the range of the variance is unrestricted.
 
This paper develops likelihood-based tests of the null hypothesis of $M_0$ components against the alternative hypothesis of $M_{0} + 1$ components for a general $M_0 \geq 1$ in multivariate normal mixtures. We consider both heteroscedastic and homoscedastic mixtures. For heteroscedastic normal mixtures, we propose an EM test by building on the EM approach pioneered by \citet{lcm09bm} and \citet{lichen10jasa}. The asymptotic null distribution of the proposed EM test statistic is shown to be the maximum of $M_0$ random variables, each of which is a projection of a Gaussian random variable on a cone. For homoscedastic normal mixtures, we derive the asymptotic distribution of the LRTS because homoscedastic normal mixtures do not suffer from the infinite Fisher information problem. 

In univariate heteroscedastic normal mixtures, \citet{chenli09as} develop an EM test for $M_0 = 1$ against $M_0=2$, and \citet{chenlifu12jasa} develop an EM test for testing $H_0:M = M_0$ against $H_A:M > M_0$. Our result may be viewed as generalization of \citet{chenli09as} to the multivariate case. \citet{kasaharashimotsu15jasa} develop an EM test for testing $H_0:M = M_0$ against $H_A:M = M_0+1$ for general $M_0 \geq 1$ in finite normal mixture regression models. In univariate homoscedastic normal mixtures, \citet{chenchen03sinica} derive the asymptotic distribution of the LRTS. Our results generalize the results in \citet{chenchen03sinica} to multivariate homoscedastic normal mixtures. For some specific models such as binomial mixtures, the asymptotic distribution of the LRTS has been derived by, for example,  \citet{ghoshsen85book, chernofflander95jspi, lemdanipons97spl, chenchen01cjstat, chenchen03sinica, cck04jrssb, garel01jspi, garel05jspi}.

The remainder of this paper is organized as follows. Section 2 introduces the likelihood ratio test for heteroscedastic multivariate normal mixture models as a precursor of the EM test and derives the asymptotic distribution of the LRTS. Section 3 introduces the EM test and derives the asymptotic distribution of the EM test statistic. Section 4 derives the asymptotic distribution of the LRTS and EM test statistics under local alternatives and Section 5 shows the validity of parametric bootstrap. Section 6 analyzes  homoscedastic multivariate normal mixture models. Section 7 reports the simulation results and provides empirical applications. Appendix A contain proofs, and Appendices B--D collect auxiliary results.  

We collect notation. Let $:=$ denote ``equals by definition.'' Boldface letters denote vectors or matrices. For a matrix $\bs{B}$, denote its $(i,j)$ element by $B_{ij}$, and let $\lambda_{\min}(\bs{B}) $ and $\lambda_{\max}(\bs{B})$ be the smallest and the largest eigenvalue of $\bs{B}$, respectively. For a $k$-dimensional vector $\bs{x} = (x_1,\ldots,x_k)\t$ and a matrix $\bs{B}$, define $|\bs{x}| := (\bs{x}\t\bs{x})^{1/2}$ and $|\bs{B}| := (\lambda_{\max}(\bs{B}\t\bs{B}))^{1/2}$. Let $\bs{x}^{\otimes k} := \bs{x} \otimes \bs{x} \otimes \cdots \otimes \bs{x}$ ($k$ times). Let $\mathbb{I}\{A\}$ denote an indicator function that takes value 1 when $A$ is true and 0 otherwise. $\mathcal{C}$ denotes a generic nonnegative finite constant whose value may change from one expression to another. Given a sequence $\{f(\bs{Y}_i)\}_{i=1}^n$, let $\nu_n(f(\bs{y})) := n^{-1/2} \sum_{i=1}^n [f(\bs{Y}_i) - Ef(\bs{Y}_i)]$ and $P_n(f(\bs{y})) := n^{-1} \sum_{i=1}^n f(\bs{Y}_i)$. All the limits are taken as $n \to \infty$ unless stated otherwise. 

\section{Heteroscedastic multivariate  finite normal mixture models}\label{sec:hetero}

Denote the density of a $d$-variate normal distribution with mean $\bs{\mu}+ \bs{\gamma}^\top\bs{z}$ and variance $\bs{\Sigma}$ by 
\begin{equation} \label{normal_density}
f(\bs{x}|\bs{z};\bs{\gamma},\bs{\mu},\bs{\Sigma}):= (2\pi)^{-\frac{d}{2}}(\det \bs{\Sigma})^{-\frac{1}{2}} \exp \left(-\frac{(\bs{x}-\bs{\mu}-\bs{\gamma}^\top\bs{z})\t\bs{\Sigma}^{-1}(\bs{x}-\bs{\mu}-\bs{\gamma}^\top\bs{z})}{2}\right),
\end{equation}
where $\bs{x}$ and $\bs{\mu}$ are $d \times 1$, $\bs{\gamma}$ is $d\times p$, and  $\bs{z}$ is $p \times 1$. Let $\Theta_{\bs{\gamma}} \subset \mathbb{R}^{dp}$, $\Theta_{{\bs{\mu}}} \subset \mathbb{R}^{d}$, and $\Theta_{\bs\Sigma} \subset \mathbb{S}^d_+$ denote the space of $\bs{\gamma}$, ${\bs{\mu}}$, and $\bs{\Sigma}$, respectively, where $\mathbb{S}^d_+$ denotes the space of $d\times d$ positive definite matrices. For $M \geq 2$, denote the density of $M$-component finite normal mixture distribution as:
\begin{equation}
f_M(\bs{x}|\bs{z};{\bs{\vartheta}}_M) : = \sum_{j = 1}^M \alpha_j f(\bs{x}|\bs{z}; \bs{\gamma},{\bs{\mu}}_j,\bs{\Sigma}_j),\label{general_model}
\end{equation}
where $\bs{\vartheta}_M := ({\bs{\alpha}},\bs{\gamma},{\bs{\mu}}_1,\ldots,{\bs{\mu}}_{M},\bs{\Sigma}_1,\ldots,\bs{\Sigma}_M)$ with ${\bs{\alpha}} : = (\alpha_1,\ldots,\alpha_{M-1})^{\top}$, and $\alpha_{M}$ being determined by $\alpha_{M}: = 1-\sum_{j = 1}^{M - 1} \alpha_j$. $\bs{\mu}_j$ and $\bs{\Sigma}_j$ are mixing parameters that characterize the $j$-th component, and $\alpha_j$s are mixing probabilities. $\bs{\gamma}$ is the coefficient of the covariate $\bs{z}$, and $\bs{\gamma}$ is assumed to be common to all the components. Define the set of admissible values of ${\bs{\alpha}}$ by $\Theta_{{\bs{\alpha}}} :=\{{\bs{\alpha}}: \alpha_j \geq 0, \sum_{j = 1}^{M - 1} \alpha_j \in [0,1] \}$, and let the space of ${\bs{\vartheta}}_M$ be $\Theta_{{\bs{\vartheta}}_M} := \Theta_{{\bs{\alpha}}}\times \Theta_{\bs{\gamma}}\times\Theta_{{\bs{\mu}}}^M\times\Theta_{\Sigma}^M$.  

The number of components $M$ is the smallest number such that the data density admits the representation (\ref{general_model}). Our objective is to test
\[
H_0:\ M = M_0\quad \text{against}\quad H_A: M = M_0 + 1.
\]

\subsection{Likelihood ratio test of $H_0: M = 1$ against $H_A: M = 2$} \label{section:homogeneity}

As a precursor of the EM test developed in Section \ref{section:emtest}, this section establishes the asymptotic distribution of the LRTS for testing the null hypothesis $H_0: M = 1$ against $H_A: M = 2$ when the data are from $H_0$. 


We consider a random sample of $n$ independent observations $\{\bs{X}_i,\bs{Z}_i\}_{i = 1}^n$ from the true one-component density $f(\bs{x}|\bs{z}; \bs{\gamma}^*, \bs{\mu}^*,\bs{\Sigma}^{*})$. Here, the superscript $*$ signifies the true parameter value. Let a two-component mixture density with ${\bs{\vartheta}}_2 = (\alpha,\bs{\gamma}, \bs{\mu}_1, \bs{\mu}_2,\bs{\Sigma}_1,\bs{\Sigma}_2) \in \Theta_{{\bs{\vartheta}}_2}$ be
\begin{equation}
f_2(\bs{x}|\bs{z};\bs{\vartheta}_2) := \alpha f(\bs{x}|\bs{z}; \bs{\gamma} ,\bs{\mu}_1,\bs{\Sigma}_1) + (1-\alpha) f(\bs{x}|\bs{z}; \bs{\gamma},\bs{\mu}_2,\bs{\Sigma}_2). \label{two-component}
\end{equation}
We partition the null hypothesis $H_0: m = 1$ into two as follows:
\[
H_{01}: (\bs{\mu}_{1},\bs{\Sigma}_1) = (\bs{\mu}_{2},\bs{\Sigma}_2)\ \text{ and }\ H_{02}: \alpha(1-\alpha) = 0.
\]
In the following, we focus on testing $H_{01}:(\bs{\mu}_{1},\bs{\Sigma}_1) = (\bs{\mu}_{2},\bs{\Sigma}_2)$ because, as discussed in \citet{chenli09as}, the Fisher information for testing $H_{02}$ is not finite unless the range of $\det(\bs{\Sigma}_1)/\det(\bs{\Sigma}_2)$ is restricted.
 
The log-likelihood function for testing $H_{01}:(\bs{\mu}_{1},\bs{\Sigma}_1) = (\bs{\mu}_{2},\bs{\Sigma}_2)$ is unbounded if $\det(\bs{\Sigma}_1)$ and $\det(\bs{\Sigma}_1)$ are not bounded away from 0 \citep{hartigan85book}. Therefore, we consider a maximum penalized likelihood estimator (PMLE) introduced by \citet{chentan09jmva}. Similar to \citet{chentan09jmva}, we use the following penalty function with $M=2$: 
\begin{equation} \label{pen_pmle}
p_n(\bs{\vartheta}_M) = \sum_{m=1}^M p_{n}(\bs{\Sigma}_m;\widehat{\bs{\Omega}}) = \sum_{m=1}^M -a_n\left\{ \text{tr}(\widehat{\bs{\Omega}}\bs{\Sigma}_m^{-1}) - \log ( \det(\widehat{\bs{\Omega}}\bs{\Sigma}_m^{-1})) -d \right\},
\end{equation}
where $\widehat{\bs{\Omega}}$ is the maximum likelihood estimator (MLE) of $\bs{\Sigma}$ from the one-component model, and $a_n$ is a non-random sequence such that $a_n \geq 1/n$ and $a_n =o(n)$.  Let $\widehat{\bs{\vartheta}}_2$ denote the PMLE that maximizes $PL_n({\bs{\vartheta}}_2):= \sum_{i = 1}^n f_2(\bs{X}_i|\bs{Z}_i;{\bs{\vartheta}}_2) + p_n(\bs{\vartheta}_2)$. 

\begin{assumption} \label{assn_consis} $\bs{Z}$ has finite second moment, and $\Pr(\bs{\gamma}^\top \bs{Z}_i \neq \bs{\gamma}^{*\top} \bs{Z}_i )>0$ for any $\bs{\gamma} \neq \bs{\gamma}^*$.
\end{assumption}
Model (\ref{two-component}) yields the true density $f(\bs{x}|\bs{z};\bs{\gamma}^*,\bs{\mu}^*,\bs{\Sigma}^{*})$ if ${\bs{\vartheta}}_2$ lies in the set $\Theta_2^* := \{{\bs{\vartheta}}_2 \in \Theta_{{\bs{\vartheta}}_2}: \{ (\bs{\mu}_1,\bs{\Sigma}_1) = (\bs{\mu}_2,\bs{\Sigma}_2) = (\bs{\mu}^*,\bs{\Sigma}^{*}), \bs{\gamma}=\bs{\gamma}^*\}\ \text{or}\ \{\alpha =1, (\bs{\mu}_1,\bs{\Sigma}_1) = (\bs{\mu}^*,\bs{\Sigma}^{*}), \bs{\gamma}=\bs{\gamma}^*\}\ \text{or}\ \{\alpha=0, (\bs{\mu}_2,\bs{\Sigma}_2) = (\bs{\mu}^*,\bs{\Sigma}^{*}), \bs{\gamma}=\bs{\gamma}^*\}\}$. The following proposition shows the consistency of $\widehat{\bs{\vartheta}}_2$.
\begin{proposition} \label{P-consis}
Suppose that Assumption \ref{assn_consis} holds. Then, under the null hypothesis $H_0: M=1$, $\inf_{{\bs{\vartheta}}_2 \in \Theta_2^*} |\widehat {\bs{\vartheta}}_2 - {\bs{\vartheta}}_2| \rightarrow_p 0$.
\end{proposition}
In testing $H_{01}$, the standard asymptotic analysis of the LRTS breaks down because the Fisher information matrix is degenerate. This is due to the fact that, for any $\bar {\bs{\vartheta}}_2$ such that $(\bs{\mu}_1,\bs{\Sigma}_1)=(\bs{\mu}_2,\bs{\Sigma}_2)$, the derivatives of the density of different orders  are linearly dependent as
\begin{align*}
&\nabla_{\bs{\mu}_1} f_2(\bs{x}|\bs{z};\bar{\bs{\vartheta}}_2) =\frac{\alpha}{1-\alpha} \nabla_{\bs{\mu}_2} f_2(\bs{x}|\bs{z};\bar{\bs{\vartheta}}_2),\ \nabla_{\bs{\Sigma}_1}f_2(\bs{x}|\bs{z};\bar{\bs{\vartheta}}_2) = \frac{\alpha}{1-\alpha} \nabla_{\bs{\Sigma}_2}l(y|\bs{x,z};\bar {\bs{\vartheta}}_2), \nonumber \\ 
&\nabla_{\mu_{1i}\mu_{1j}} f_2(\bs{x}|\bs{z};\bar{\bs{\vartheta}}_2) = 2\nabla_{\Sigma_{1,ij}} f_2(\bs{x}|\bs{z};\bar{\bs{\vartheta}}_2), \quad \nabla_{\mu_{2i}\mu_{2j}} f_2(\bs{x}|\bs{z};\bar{\bs{\vartheta}}_2) = 2\nabla_{\Sigma_{2,ij}} f_2(\bs{x}|\bs{z};\bar{\bs{\vartheta}}_2).
\end{align*}
This dependence leads to the loss of strong identifiability and causes substantial difficulties in existing literature.

We analyze the LRTS for testing $H_{01}: (\bs{\mu}_{1},\bs{\Sigma}_1)=(\bs{\mu}_{2},\bs{\Sigma}_2)$ by developing a higher-order approximation of the log-likelihood function  through an ingenious reparameterization that  extends the result of \citet{rotnitzky00bernoulli} and \citet{kasaharashimotsu15jasa}. Collect the unique elements in $\bs{\Sigma}$ into a $d(d+1)/2$-vector
\begin{equation*}
\begin{aligned}
\bs{v} &= (v_{11},v_{12},\ldots,v_{1d},v_{22},v_{23},\ldots,v_{2d},\ldots,v_{d-1,d-1},v_{d-1,d},v_{dd}) \t \\
& := (\Sigma_{11},2\Sigma_{12},\ldots,2\Sigma_{1d},\Sigma_{22},2\Sigma_{23},\ldots,2\Sigma_{2d},\ldots,\Sigma_{d-1,d-1},2\Sigma_{d-1,d},\Sigma_{dd}) \t.
\end{aligned}
\end{equation*}
Define the density of $N(\bs{\mu},\bs{\Sigma})$ parameterized in terms of $\bs{\mu}$ and $\bs{v}$ as 
\begin{equation} \label{fv_defn}
f_v(\bs{\mu},\bs{v}) : = f(\bs{\mu},\bs{S}(\bs{v})), \quad \text{ where }\ S_{ij}(\bs{v}):= \begin{cases}
v_{ii} & \text{if } i=j , \\
v_{ij}/2 & \text{if } i \neq j .
\end{cases}
\end{equation}
For a $d\times d$ symmetric matrix $\bs{A}$, define a function $\bs{w}(\bs{A}) \in \mathbb{R}^{d(d+1)/2}$ that collects the unique elements of $\bs{A}$ as
\begin{align*}
\bs{w}(\bs{A}) & := (A_{11},2A_{12},\ldots,2A_{1d},A_{22},2A_{23},\ldots,2A_{2d},\ldots,A_{d-1,d-1},2A_{d-1,d},A_{dd}) \t.
\end{align*}
Then $f_v(\bs{\mu},\bs{v})$ and $f(\bs{\mu},\bs{\Sigma})$ are related as
\[
f(\bs{\mu},\bs{\Sigma}) = f_v(\bs{\mu},\bs{w}(\bs{\Sigma})).
\]

We introduce the following one-to-one mapping between $(\bs{\mu}_1,\bs{\mu}_2,\bs{v}_1,\bs{v}_2)$ and the reparameterized parameter $(\bs{\lambda}_{\bs{\mu}},\bs{\nu}_{\bs\mu},\bs{\lambda}_{\bs{v}},\bs{\nu}_{\bs{v}})$:
\begin{equation} \label{repara2}
\begin{pmatrix}
{\bs{\mu}}_1\\
{\bs{\mu}}_2\\
\bs{v}_1\\
\bs{v}_2
\end{pmatrix}
 =
\begin{pmatrix}
\bs{\nu}_{\bs\mu} + (1-\alpha) \bs{\lambda}_{\bs\mu} \\
\bs{\nu}_{\bs\mu} -\alpha \bs{\lambda}_{\bs\mu}\\
\bs{\nu}_{\bs{v}} + (1- \alpha)(2\bs{\lambda}_{\bs{v}}+ C_1 \bs{w}(\bs{\lambda}_{\bs\mu}\bs{\lambda}_{\bs\mu}\t) )\\
\bs{\nu}_{\bs{v}} - \alpha(2\bs{\lambda}_{\bs{v}}+ C_2 \bs{w}(\bs{\lambda}_{\bs\mu}\bs{\lambda}_{\bs\mu}\t)
\end{pmatrix}, 
\end{equation}
where $C_1 := -(1/3)(1 + \alpha)$ and $C_2 := (1/3)(2 - \alpha)$. Collect the reparameterized parameters, except for $\alpha$, into one vector $\bs{\psi}$ defined as
\begin{equation*} 
\bs{\psi}:=(\bs{\gamma},\bs{\nu}_{\bs\mu},\bs{\nu}_{\bs{v}},\bs{\lambda}_{\bs\mu},\bs{\lambda}_{\bs{v}}) \in \Theta_{\bs{\psi}}.
\end{equation*}
In the reparameterized model, the null hypothesis of $H_{01}:(\bs{\mu}_{1},\bs{v}_1) = (\bs{\mu}_{2},\bs{v}_2)$ is written as $H_{01}:(\bs{\lambda}_{\bs\mu},\bs{\lambda}_{\bs v})= \bs{0}$, and the density is given by
\begin{equation} \label{loglike}
\begin{aligned}
g(\bs{x}|\bs{z};\bs{\psi},\alpha) & = \alpha f_v\left(\bs{x}\middle|\bs{z};\bs{\gamma},\bs{\nu}_{\bs\mu}+(1-\alpha)\bs{\lambda}_{\bs\mu}, \bs{\nu}_{\bs{v}} + (1 - \alpha)(2\bs{\lambda}_{\bs{v}} + C_1 \bs{w}(\bs{\lambda}_{\bs\mu}\bs{\lambda}_{\bs\mu}\t) ) \right) \\
& \quad + (1 - \alpha) f_v \left(\bs{x} \middle|\bs{z};\bs{\gamma},\bs{\nu}_{\bs\mu} -\alpha\bs{\lambda}_{\bs\mu},\bs{\nu}_{\bs{v}} - \alpha(2\bs{\lambda}_{\bs v}+ C_2 \bs{w}( \bs{\lambda}_{\bs\mu}\bs{\lambda}_{\bs\mu}\t) ) \right).
\end{aligned}
\end{equation}

Partition $\bs{\psi}$ as $\bs{\psi} = (\bs{\eta}\t,\bs{\lambda}\t)\t$, where $\bs{\eta}: = (\bs{\gamma}\t,\bs{\nu}_{\bs\mu}\t,\bs{\nu}_{\bs{v}}\t)\t \in \Theta_{\bs{\eta}}$ and $\bs{\lambda}: = (\bs{\lambda}_{\bs\mu}\t,\bs{\lambda}_{\bs v}\t)\t \in\Theta_{\bs{\lambda}}$. Denote the true values of $\bs{\eta}$, $\bs{\lambda}$, and $\bs{\psi}$ by $\bs{\eta}^*: = ((\bs{\gamma}^*)\t,({\bs{\mu}}^*)\t,(\bs{v}^{*})\t)\t$, $\bs{\lambda}^*: = \bs{0}$, and $\bs{\psi}^* = ((\bs{\eta}^*)\t, \bs{0}\t)\t$, respectively.  Under this reparameterization, the first derivative of (\ref{loglike}) with respect to (w.r.t., hereafter) $\bs{\eta}$ under $\bs{\psi} = \bs{\psi}^*$ is identical to the first derivative of the density of the one-component model:
\begin{equation}\label{dvareta}
\nabla_{\bs{\eta}}g(\bs{x}|\bs{z};\bs{\psi}^*,\alpha) =\nabla_{(\bs{\gamma}\t,\bs{\mu}\t,\bs{v}\t)\t} f_v(\bs{x}|\bs{z};\bs{\gamma}^*,\bs{\mu}^*,\bs{v}^{*}).
\end{equation}
On the other hand,   the first, second, and third derivatives of $g(\bs{x}|\bs{z};\bs{\psi},\alpha)$ w.r.t.\ $\bs{\lambda}_{\bs\mu}$ and the first derivative w.r.t.\ $\bs{\lambda}_{\bs v}$ become zero when evaluated at $\bs{\psi}=\bs{\psi}^*$. Consequently, the information on $\bs{\lambda}_{\bs\mu}$ and $\bs{\lambda}_{\bs v}$ is provided by the fourth derivative w.r.t.\ $\bs{\lambda}_{\bs\mu}$, the cross-derivative w.r.t.\ $\bs{\lambda}_{\bs{\mu}}$ and $\bs{\lambda}_{\bs v}$, and the second derivative w.r.t.\ $\bs{\lambda}_{\bs v}$. 

We derive the asymptotic distribution of the LRTS.  Let $f^*_v$ and $\nabla f^*_v$ denote $f_v(\bs{x}|\bs{z};\bs{\gamma}^*,\bs{\mu}^*,\bs{v}^*)$ and $\nabla f_v(\bs{x}|\bs{z};\bs{\gamma}^*,\bs{\mu}^*,\bs{v}^*)$, and let $d_\eta:=(p+d+d(d+1)/2)$, $d_{\mu v}:=d(d+1)(d+2)/6$, and $d_{\mu^4}:=d(d+1)(d+2)(d+3)/24$. Define the score vector $\bs{s}(\bs{x},\bs{z})$ as
\begin{equation} \label{score_defn}
\begin{aligned}
\bs{s}(\bs{x},\bs{z}) & := 
\begin{pmatrix}
\bs{s}_{\bs{\eta}}\\
\bs{s}_{\bs{\lambda}}
\end{pmatrix} := 
\begin{pmatrix}
\bs{s}_{\bs{\eta}}\\
\bs{s}_{\bs{\mu v}}\\
\bs{s}_{\bs{\mu}^4}
\end{pmatrix}
\quad \text{with }\
\underset{(d_\mu \times 1)}{\bs{s}_{\bs{\eta}}} := \frac{\nabla_{(\bs{\gamma}\t,\bs{\mu}\t,\bs{v}\t)\t}f^*_v }{ f^*_v}, \\
\underset{(d_{\mu v} \times 1)}{\bs{s}_{\bs{\mu v}}} &:= \left\{\frac{\nabla_{\mu_i \mu_j \mu_k} f^*_v}{ 3! f^*_v} \right\}_{1 \leq i \leq j \leq k \leq d}, \quad
\underset{(d_{\mu^4} \times 1)}{\bs{s}_{\bs{\mu}^4}} := \left\{ \frac{\nabla_{\mu_i \mu_j \mu_k \mu_\ell} f^*_v }{ 4! f^*_v} \right\}_{1 \leq i \leq j \leq k \leq \ell \leq d} ,
\end{aligned}
\end{equation}
where we suppress the dependence of $(\bs{s}_{\bs{\eta}}, \bs{s}_{\bs{\mu v}}, \bs{s}_{\bs{\mu}^4})$ on $(\bs{x},\bs{z})$. Collect the relevant reparameterized parameters as
\begin{equation} \label{tpsi_defn}
\bs{t}(\bs{\psi},\alpha) :=
\begin{pmatrix}
\bs{\eta}-\bs{\eta}^*\\
\bs{t}_{\bs{\lambda}}(\bs{\lambda},\alpha)
\end{pmatrix}
:=
\begin{pmatrix}
\bs{\eta}-\bs{\eta}^*\\
\alpha(1-\alpha)12 \bs{\lambda}_{\bs{\mu v}}\\
\alpha(1-\alpha)[12\bs{\lambda}_{\bs{v}^2}+ b(\alpha) \bs{\lambda}_{\bs{\mu}^4}]
\end{pmatrix},
\end{equation}
with $b(\alpha): = -(2/3) (\alpha^2 - \alpha + 1)<0$ and 
\begin{equation} \label{lambda_muv_defn}
\begin{aligned}
\underset{(d_{\mu v} \times 1)}{\bs{\lambda}_{\bs{\mu v}}} &:= \{(\bs{\lambda}_{\bs{\mu v}})_{ijk}\}_{1 \leq i \leq j \leq k \leq d}, \text{ where } (\bs{\lambda}_{\bs{\mu v}})_{ijk} :=\sum_{(t_1,t_2,t_3) \in p_{12}(i,j,k)} \lambda_{\mu_{t_1}}\lambda_{v_{t_2t_3}}, \\
\underset{(d_{\mu^4} \times 1)}{\bs{\lambda}_{\bs{v}^2}} &:= \{(\bs{\lambda}_{\bs{v}^2})_{ijk\ell}\}_{1 \leq i \leq j \leq k \leq \ell \leq d}, \text{ where } (\bs{\lambda}_{\bs{v}^2})_{ijk\ell}:= \sum_{(t_1,t_2,t_3,t_4) \in p_{22}(i,j,k,\ell)} \lambda_{v_{t_1t_2}}\lambda_{v_{t_3t_4}} ,\\
\underset{(d_{\mu^4} \times 1)}{\bs{\lambda}_{\bs{\mu}^4}} &:= \{(\bs{\lambda}_{\bs{\mu}^4})_{ijk\ell}\}_{1 \leq i \leq j \leq k \leq \ell \leq d}, \text{ where } (\bs{\lambda}_{\bs{\mu}^4})_{ijk\ell}:= \sum_{(t_1,t_2,t_3,t_4) \in p(i,j,k,\ell)} \lambda_{\mu_{t_1}}\lambda_{\mu_{t_2}}\lambda_{\mu_{t_3}}\lambda_{\mu_{t_4}} ,
\end{aligned}
\end{equation}
where $\sum_{(t_1,t_2,t_3) \in p_{12}(i,j,k)}$ denotes the sum over all distinct permutations of $(i,j,k)$ to $(t_1,t_2,t_3)$ with $t_2 \leq t_3$, $\sum_{(t_1,t_2,t_3,t_4) \in p_{22}(i,j,k,\ell)}$ denotes the sum over all distinct permutations of $(i,j,k,\ell)$ to $(t_1,t_2,t_3,t_4)$ with $t_1 \leq t_2$ and $t_3 \leq t_4$, and $\sum_{(t_1,t_2,t_3,t_4) \in p(i,j,k,\ell)}$ denotes the sum over all distinct permutations of $(i,j,k,\ell)$ to $(t_1,t_2,t_3,t_4)$.  In (\ref{lambda_muv_defn}), $\bs{\lambda}_{\bs{\mu v}}$ is a function of $\bs{\lambda}_{\bs{\mu}} \otimes \bs{\lambda}_{\bs{v}}$ and corresponds to the score vector $\bs{s}_{\bs{\mu v}}$. $\bs{\lambda}_{\bs{v}^2}$ is a function of $\bs{\lambda}_{\bs{v}}^{\otimes 2}$, and $\bs{\lambda}_{\bs{\mu}^4}$ depends on $\bs{\lambda}_{\bs{\mu}}^{\otimes 4}$. Here, $\alpha(1-\alpha)12\bs{\lambda}_{\bs{\mu v}}\t \bs{s}_{\bs{\mu v}}$ collects the unique elements that correspond to the cross-derivative with respect to $\bs{\lambda}_{\bs{\mu}}$ and $\bs{\lambda}_{\bs{v}}$ in the expansion of the log-likelihood function, and $\alpha(1-\alpha)[12\bs{\lambda}_{\bs{v}^2}+ b(\alpha) \bs{\lambda}_{\bs{\mu}^4}]\t \bs{s}_{\bs{\mu}^4}$ collects the unique elements of the second-order terms with respect to $\bs{\lambda}_{\bs{v}}$ and the fourth-order terms with respect to $\bs{\lambda}_{\bs{\mu}}$.

Let $L_n(\bs{\psi},\alpha): = \sum_{i = 1}^n \log g(\bs{X}_i|\bs{Z}_i;\bs{\psi},\alpha)$ denote the reparameterized log-likelihood function. Let $\widehat{\bs{\psi}} : = \arg\max_{\bs{\psi} \in \Theta_{\bs{\psi}}} PL_n(\bs{\psi},\alpha)$ denote the PMLE of $\bs{\psi}$, where $\Theta_{\bs{\psi}}$ is defined so that the value of ${\bs{\vartheta}}_2$ implied by $\bs{\psi}$ is in $\Theta_{{\bs{\vartheta}}_2}$. Let $(\widehat{\bs{\gamma}}_0,\widehat{\bs{\mu}}_0,\widehat{\bs{\Sigma}}_0)$ denote the one-component MLE that maximizes the one-component log-likelihood function $L_{0,n}(\bs{\gamma},{\bs{\mu}},\bs{\Sigma}) := \sum_{i=1}^n \log f (\bs{X}_i|\bs{Z}_i;\bs{\gamma},{\bs{\mu}},\bs{\Sigma})$. Define the LRTS for testing $H_{01}$ as, with $\epsilon_1 \in (0,1/2)$,
\begin{equation} \label{LRT1}
LR_{n}(\epsilon_1) := \max_{\alpha \in [\epsilon_1,1-\epsilon_1]} 2\{L_n(\widehat{\bs{\psi}},\alpha) - L_{0,n}(\widehat{\bs{\gamma}}_0,\widehat{\bs{\mu}}_0,\widehat{\bs{\Sigma}}_0)\}.
\end{equation}
We could use the penalized LRTS defined by $PLR_{n}(\epsilon_1):= \max_{\alpha \in [\epsilon_1,1-\epsilon_1]} 2\{PL_n(\widehat{\bs{\psi}},\alpha) - L_{0,n}(\widehat{\bs{\gamma}}_0,\widehat{\bs{\mu}}_0,\widehat{\bs{\Sigma}}_0)\}$ instead of $LR_n(\epsilon_1)$. Because the effect of the penalty term is negligible under our assumptions, $PLR_n(\epsilon_1)$ has the same asymptotic distribution as $LR_n(\epsilon_1)$.

With $(\bs{s}_{\bs{\eta}}, \bs{s}_{\bs{\lambda}})$ defined in (\ref{score_defn}), define 
\begin{equation} \label{I_lambda}
\begin{aligned}
&\bs{\mathcal{I}}_{\bs{\eta}} := E[\bs{s}_{\bs{\eta}}\bs{s}_{\bs{\eta}}\t], \quad \bs{\mathcal{I}}_{\bs{\lambda}} := E[\bs{s}_{\bs{\lambda}}\bs{s}_{\bs{\lambda}}\t], \quad \bs{\mathcal{I}}_{\bs{\lambda\eta} } := E[\bs{s}_{\bs{\lambda}}\bs{s}_{\bs{\eta}}\t],\\
&\bs{\mathcal{I}}_{\bs{\eta \lambda}} := \bs{\mathcal{I}}_{\bs{\lambda \eta}}\t, \quad \bs{\mathcal{I}}_{\bs{\lambda}.\bs{\eta}}:=\bs{\mathcal{I}}_{\bs{\lambda}}-\bs{\mathcal{I}}_{\bs{\lambda\eta}}\bs{\mathcal{I}}_{\bs{\eta}}^{-1}\bs{\mathcal{I}}_{\bs{\eta\lambda}}, \quad \bs{Z}_{\bs{\lambda}}:=(\bs{\mathcal{I}}_{\bs{\lambda}.\bs{\eta}})^{-1} \bs{G}_{\bs{\lambda}.\bs{\eta}},
\end{aligned}
\end{equation} 
where $\bs{G}_{\bs{\lambda}.\bs{\eta}} \sim N(0,\bs{\mathcal{I}}_{\bs{\lambda}.\bs{\eta}})$. The following sets characterize the limit of possible values of $\sqrt{n}\bs{t}_{\bs{\lambda}}(\bs{\lambda},\alpha)$ defined in (\ref{tpsi_defn}) as $n\rightarrow\infty$. Define 
\begin{equation} \label{Lambda-e} 
\begin{aligned}
\Lambda_{\bs{\lambda} }^{1} & := \left\{ \left((\bs{t}_{\bs{\mu v}} )\t, ( \bs{t}_{\bs{\mu}^4} ) \t \right)\t \in \mathbb{R}^{d_{\mu v}+d_{\mu^4}}: \bs{t}_{\bs{\mu v}} = \bs{\lambda}_{\bs{\mu v}},\ \bs{t}_{\bs{\mu}^4}= \bs{\lambda}_{\bs{v}^2} \text{ for some $\bs{\lambda} \in \mathbb{R}^{d+d(d+1)/2}$} \right\},\\
\Lambda_{\bs{\lambda} }^{2} & := \left\{ \left((\bs{t}_{\bs{\mu v}} )\t, ( \bs{t}_{\bs{\mu}^4} ) \t \right)\t \in \mathbb{R}^{d_{\mu v}+d_{\mu^4}}: \bs{t}_{\bs{\mu v}} = \bs{\lambda}_{\bs{\mu v}},\ \bs{t}_{\bs{\mu}^4}= -\bs{\lambda}_{\bs{\mu}^4} \text{ for some $\bs{\lambda} \in \mathbb{R}^{d+d(d+1)/2}$} \right\}.
\end{aligned}
\end{equation} 
For $j=1,2$, define $\widehat{\bs{t}}_{\bs{\lambda}}^{j}$ by
\begin{equation} \label{t-lambda}
r(\widehat{\bs{t}}_{\bs{\lambda}}^{j}) = \inf_{\bs{t}_{\bs{\lambda}} \in \Lambda_{\bs{\lambda} }^{j}}r(\bs{t}_{\bs{\lambda}}), \quad r(\bs{t}_{\bs{\lambda}}) := (\bs{t}_{\bs{\lambda}} -\bs{Z}_{\bs{\lambda}})\t \bs{\mathcal{I}}_{\bs{\lambda}.\bs{\eta}} (\bs{t}_{\bs{\lambda}} -\bs{Z}_{\bs{\lambda}}),
\end{equation}
where $\bs{\mathcal{I}}_{\bs{\lambda}.\bs{\eta}}$, $\bs{Z}_{\bs{\lambda}}$, and $\Lambda_{\bs{\lambda} }^{j}$ for $j=1,2$ are defined in (\ref{I_lambda})--(\ref{Lambda-e}).

The following proposition establishes the asymptotic null distribution of the LRTS. 
\begin{assumption} \label{A-taylor1}
$\bs{Z}$ has finite tenth moment. 
\end{assumption}
\begin{proposition} \label{P-LR-N1}
Suppose that Assumptions \ref{assn_consis} and \ref{A-taylor1} hold, $a_n$ in (\ref{pen_pmle}) satisfies $a_n = O(1)$, and $\bs{\mathcal{I}}:=E[\bs{s}(\bs{X},\bs{Z})\bs{s}(\bs{X},\bs{Z})\t]$ is finite and nonsingular. Then, under the null hypothesis of $M=1$,  $LR_{n}(\epsilon_1) \rightarrow_d \max\left\{ (\widehat{\bs{t}}_{\bs{\lambda}}^{1})\t \bs{\mathcal{I}}_{\bs{\lambda}.\bs{\eta}} \widehat{\bs{t}}_{\bs{\lambda}}^{1}, (\widehat{\bs{t}}_{\bs{\lambda}}^{2})\t \bs{\mathcal{I}}_{\bs{\lambda}.\bs{\eta}} \widehat{\bs{t}}_{\bs{\lambda}}^{2} \right\}$, where $LR_{n}(\epsilon_1)$ and $\widehat{\bs{t}}_{\bs{\lambda}}^{j}$ are defined in (\ref{LRT1}) and (\ref{t-lambda}), respectively.
\end{proposition}
For each $j=1,2$, the random variable $(\widehat{\bs{t}}_{\bs{\lambda}}^{j})\t \bs{\mathcal{I}}_{\bs{\lambda}.\bs{\eta}} \widehat{\bs{t}}_{\bs{\lambda}}^{j}$ is a projection of a Gaussian random variable on a cone $\Lambda_{\bs{\lambda} }^{j}$. 

\begin{example}
When $d=1$ with $\bs{\lambda}=(\lambda_\mu,\lambda_v)\t$, we have $\Lambda_{\bs{\lambda} }^{1} \cup \Lambda_{\bs{\lambda} }^{2} = \mathbb{R}^2$ and  $LR_{n}(\epsilon_1) \rightarrow_d \bs{Z}_{\bs{\lambda}}\t \bs{\mathcal{I}}_{\bs{\lambda}.\bs{\eta}} \bs{Z}_{\bs{\lambda}} \sim \chi^2(2)$.
When $d=2$, we have $\bs{\lambda}=(\bs{\lambda}_{\bs{\mu}}\t,\bs{\lambda}_{\bs{v}}\t)\t=(\lambda_{\mu_1},\lambda_{\mu_2},\lambda_{v_{11}},\lambda_{v_{12}}, \lambda_{v_{22}})\t$, 
\[
\begin{aligned}
\bs{s}_{\bs{\mu v}} &= (\nabla_{\mu_1^3} f^*_v,\nabla_{\mu_1^2\mu_2} f^*_v,\nabla_{\mu_1\mu_2^2} f^*_v,\nabla_{\mu_2^3} f^*_v)\t/ 3! f^*_v,  \\
\bs{s}_{\bs{\mu}^4} & = (\nabla_{\mu_1^4} f^*_v,\nabla_{\mu_1^3\mu_2} f^*_v,\nabla_{\mu_1^2\mu_2^2} f^*_v,\nabla_{\mu_1\mu_2^3} f^*_v,\nabla_{\mu_2^4} f^*_v)\t/ 4! f^*_v,
 \end{aligned}
 \]
 and  $\Lambda_{\bs{\lambda} }^{j}$ is given by (\ref{Lambda-e}) with
\[
\begin{aligned} 
\bs{t}_{\bs{\mu v}} & =(\lambda_{\mu_1}\lambda_{v_{11}},\lambda_{\mu_1}\lambda_{v_{12}}+\lambda_{\mu_2}\lambda_{v_{11}}, \lambda_{\mu_1}\lambda_{v_{22}}+\lambda_{\mu_2}\lambda_{v_{12}},\lambda_{\mu_2}\lambda_{v_{22}})\t,\\ 
 \bs{t}_{\bs{\mu }^4} &= 
 \begin{cases} 
 ( \lambda_{v_{11}}^2, 2\lambda_{v_{11}}\lambda_{v_{12}}, 2\lambda_{v_{11}}\lambda_{v_{22}}+\lambda_{v_{12}}^2, 2\lambda_{v_{12}}\lambda_{v_{22}} , \lambda_{v_{22}}^2)\t& \text{if } j=1,\\ 
 -(\lambda_{\mu_1}^4,4\lambda_{\mu_1}^3\lambda_{\mu_2},6\lambda_{\mu_1}^2\lambda_{\mu_2}^2,4\lambda_{\mu_1}\lambda_{\mu_2}^3,\lambda_{\mu_2}^4)\t& \text{if } j=2.
 \end{cases}
 \end{aligned}
 \]
\end{example}

\subsection{Likelihood ratio test of $H_0: M = M_0$ against $H_A: M = M_0 + 1$ for $M_0\geq 2$} \label{sec-general}

This section establishes the asymptotic distribution of the LRTS for testing the null hypothesis of $M_0$ components against the alternative of $M_0+1$ components for general $M_0 \geq 1$. 


We consider a random sample of $n$ independent observations $\{\bs{X}_i,\bs{Z}_i\}_{i = 1}^n$ from the $M_0$-component $d$-variate finite normal mixture distribution, whose density with the true parameter value ${\bs{\vartheta}}_{M_0}^*=(\alpha_{1}^*,\ldots,\alpha_{M_0 - 1}^*,{\bs{\gamma}}^*,\bs{\mu}_1^*,\ldots,\bs{\mu}_{M}^*,\bs{\Sigma}_{1}^{*},\ldots,\bs{\Sigma}_{M}^{*})$ is
\begin{equation}
f_{M_0}(\bs{x}|\bs{z};{\bs{\vartheta}}_{M_0}^*):=\sum_{j=1}^{M_0} \alpha_{j}^{*} f(\bs{x}|\bs{z};{\bs{\gamma}}^*, \bs{\mu}_{j}^{*},\bs{\Sigma}_{j}^{*}), \label{true_model}
\end{equation}
where $\alpha_j^*>0$. We assume $(\bs{\mu}_{1}^*,\bs{\Sigma}^{*}_{1}) <\ldots< (\bs{\mu}_{M_0}^*,\bs{\Sigma}^{*}_{M_0})$ for identification. Let the density of an $(M_0+1)$-component mixture model be
\begin{equation}
f_{M_0+1}(\bs{x}|\bs{z};{\bs{\vartheta}}_{M_0+1}):=\sum_{j=1}^{M_0+1}\alpha_j f(\bs{x}|\bs{z};\bs{\gamma},\bs{\mu}_j,\bs{\Sigma}_j),\label{fitted_model}
\end{equation}
where ${\bs{\vartheta}}_{M_0+1} = (\alpha_1,\ldots,\alpha_{M_0},\bs{\gamma},\bs{\mu}_1,\ldots.,\bs{\mu}_{M_0+1},\bs{\Sigma}_1,\ldots,\bs{\Sigma}_{M_0+1})$. As in the case of the test of homogeneity, we partition the null hypothesis into two as $H_0 = H_{01} \cup H_{02}$, where $H_{01}: = \cup_{m=1}^{M_0} H_{0,1m}$ and $H_{02} := \cup_{m=1}^{M_0+1} H_{0,2m}$ with
\[
H_{0,1m} :(\bs{\mu}_1,\bs{\Sigma}_1) < \cdots < (\bs{\mu}_{m},\bs{\Sigma}_{m}) = (\bs{\mu}_{m + 1},\bs{\Sigma}_{m + 1}) < \cdots < (\bs{\mu}_{M_0 + 1},\bs{\Sigma}_{M_0 + 1}) \ \text{and}\ H_{0,2m}: \alpha_m = 0.
\]
The inequality constraints are imposed on $(\bs{\mu}_j,\bs{\Sigma}_j)$ for identification.

We focus on testing $H_{01}$ because the LRTS for testing $H_{02}$ has infinite Fisher information unless a stringent restriction is imposed on the admissible values of $\bs{\Sigma}_j$ \citep{kasaharashimotsu15jasa}. Define the set of values of ${\bs{\vartheta}}_{M_0+1}$ that yields the true density (\ref{true_model}) as 
\begin{equation*}
\Upsilon^*:=\{{\bs{\vartheta}}_{M_0 + 1}: f_{M_0 + 1}(\bs{X}|\bs{Z};{\bs{\vartheta}}_{M_0 + 1}) = f_{M_0}(\bs{X}|\bs{Z};{\bs{\vartheta}}_{M_0}^*)\text{ with probability one}\}.
\end{equation*} 
Under $H_{0,1m}$, the $(M_0 + 1)$-component model (\ref{fitted_model}) generates the true $M_0$-component density (\ref{true_model}) when $(\bs{\mu}_m,\bs{\Sigma}_m) = (\bs{\mu}_{m + 1},\bs{\Sigma}_{m + 1})=(\bs{\mu}_{m}^{*},\bs{\Sigma}^{*}_{m})$. Define the subset of $\Upsilon^*$ corresponding to $H_{0,1m}$ as
\begin{align}
\Upsilon_{1m}^* &:= \left\{{\bs{\vartheta}}_{M_0 + 1} \in \Theta_{{\bs{\vartheta}}_{M_0 + 1}}:\  \alpha_j = \alpha_{j}^{*}\ \text{and}\
(\bs{\mu}_j,\bs{\Sigma}_j)=(\bs{\mu}_{j}^*,\bs{\Sigma}_{j}^{*})\ \text{for $j < m$}; \right. \nonumber \\
& \qquad \left. 
\alpha_m + \alpha_{m + 1} = \alpha_{m}^{*}\ \text{and}\ (\bs{\mu}_m,\bs{\Sigma}_m)=(\bs{\mu}_{m + 1},\bs{\Sigma}_{m + 1})=(\bs{\mu}_{m}^{*},\bs{\Sigma}_{m}^{*});\   \right.\nonumber \\
& \qquad \left. \alpha_{j}=\alpha_{j - 1}^*\ \text{and}\ (\bs{\mu}_{j},\bs{\Sigma}_j)=(\bs{\mu}_{j - 1}^*,\bs{\Sigma}_{j - 1}^{*})\ \text{for $j> m + 1$};\ {\bs{\gamma}}={\bs{\gamma}}^* \right\}, \nonumber 
\end{align}
and define $\Upsilon_1^*:= \Upsilon_{11}^* \cup \cdots \cup \Upsilon_{1M_0}^*$.
 
Let $\Theta_{{\bs{\vartheta}}_{M_0 + 1}}(\epsilon_1)$ be a subset of $\Theta_{{\bs{\vartheta}}_{M_0 + 1}}$ such that $\alpha_j\in[\epsilon_1,1-\epsilon_1]$ for $j = 1,\ldots,M_0 + 1$, and define the PMLE by
\begin{equation}\label{PMLEs}
\begin{aligned}
\widehat{\bs{\vartheta}}_{M_0+1}(\epsilon_1) & := \mathop{\arg \max}_{{\bs{\vartheta}}_{M_0+1}\in\Theta_{{\bs{\vartheta}}_{M_0+1}}(\epsilon_1)}PL_n({\bs{\vartheta}}_{M_0 + 1}),\\
\widehat{\bs{\vartheta}}_{M_0} & := \mathop{\arg \max}_{{\bs{\vartheta}}_{M_0}\in\Theta_{{\bs{\vartheta}}_{M_0}}}PL_{0,n}({\bs{\vartheta}}_{M_0}),
\end{aligned}
\end{equation}
where $PL_n({\bs{\vartheta}}_{M_0 + 1}):=L_n({\bs{\vartheta}}_{M_0 + 1})+p_n(\bs{\vartheta}_{M_0+1})$ and $PL_{0,n}({\bs{\vartheta}}_{M_0}):=L_{0,n}({\bs{\vartheta}}_{M_0})+p_n(\bs{\vartheta}_{M_0})$ with $L_n({\bs{\vartheta}}_{M_0 + 1}):=\sum_{i = 1}^n \log f_{M_0 + 1}(\bs{X}_i|\bs{Z}_i;{\bs{\vartheta}}_{M_0 + 1})$ and $L_{0,n}({\bs{\vartheta}}_{M_0}):=\sum_{i = 1}^n \log f_{M_0}(\bs{X}_i|\bs{Z}_i;{\bs{\vartheta}}_{M_0})$ for the density (\ref{true_model})--(\ref{fitted_model}) and the penalty function in (\ref{pen_pmle}). 
We consider the LRTS for testing $H_{01}$ given by
\begin{equation}\label{LR-M_0}
LR_{n}^{M_0}(\epsilon_1):= 2\{L_n(\widehat{{\bs{\vartheta}}}_{M_0+1}(\epsilon_1))- L_{0,n}(\widehat{{\bs{\vartheta}}}_{M_0})\}.
\end{equation} 

Collect the score vector for testing $H_{0,11},\ldots,H_{0,1M_0}$ into one vector as
\begin{equation}
\widetilde{\bs{s}}(\bs{x},\bs{z}) :=
\begin{pmatrix}
\widetilde{\bs{s}}_{\bs{\eta}} \\ \widetilde{\bs{s}}_{\bs{\lambda}}
\end{pmatrix},\ \text{ where }\
\widetilde{\bs{s}}_{\bs{\eta}}:=
\begin{pmatrix}
\bs{s}_{\bs{\alpha}} \\ \bs{s}_{(\bs{\gamma},\bs{\mu}, \bs{v})} 
\end{pmatrix}\ \text{and}\ \widetilde{\bs{s}}_{\bs{\lambda}}: =
\begin{pmatrix}
\bs{s}_{\bs{\mu}\bs{v}}^1 \\
\bs{s}_{\bs{\mu}^4}^1 \\
 \vdots \\
\bs{s}_{\bs{\mu}\bs{v}}^{M_0} \\
\bs{s}_{\bs{\mu}^4}^{M_0} 
\end{pmatrix}, \label{stilde}
\end{equation}
where, with $f_0^*:=f_{M_0}(\bs{x}|\bs{z};{\bs{\vartheta}}_{M_0}^*)$ and for $m=1,\ldots,M_0$,
\begin{equation} \label{sh}
\begin{aligned}
\bs{s}_{\bs{\alpha}} & :=
\begin{pmatrix}
f_v(\bs{x}|\bs{z};\bs{\gamma}^*,\bs{\mu}_1^{*},\bs{v}_1^*)-f_v(\bs{x}|\bs{z};\bs{\gamma}^*,\bs{\mu}_{M_0}^{*},\bs{v}_{M_0}^*) \\
\vdots \\
f_v(\bs{x}|\bs{z};\bs{\gamma}^*,\bs{\mu}_{M_0-1}^{*},\bs{v}_{M_0-1}^*)-f_v(\bs{x}|\bs{z};\bs{\gamma}^*,\bs{\mu}^{*}_{M_0},\bs{v}_{M_0}^*) \\
\end{pmatrix} \Bigl/ f_0^*, \bigr.
\\
\bs{s}_{(\bs{\gamma},\bs{\mu}, \bs{v})} & :=\sum_{m=1}^{M_0}\alpha_{m}^* \nabla_{(\bs{\gamma}\t,\bs{\mu}\t, \bs{v}\t)\t} f_v(\bs{x}|\bs{z};\bs{\gamma}^*,\bs{\mu}_m^{*},\bs{v}_m^*)/ f_0^*, \\ 
\bs{s}_{\bs{\mu v}}^m &:= \left\{\alpha_{m}^* \nabla_{\mu_i \mu_j \mu_k} f^*_v(\bs{x}|\bs{z};\bs{\gamma}^*,\bs{\mu}_m^{*},\bs{v}_m^*) / (3! f_0^*) \right\}_{1 \leq i \leq j \leq k \leq d}, \\
\bs{s}_{\bs{\mu}^4}^m &:= \left\{ \alpha_{m}^* \nabla_{\mu_i \mu_j \mu_k \mu_\ell} f^*_v(\bs{x}|\bs{z};\bs{\gamma}^*,\bs{\mu}_m^{*},\bs{v}_m^*) / (4! f_0^*) \right\}_{1 \leq i \leq j \leq k \leq \ell \leq d} .
\end{aligned} 
\end{equation}

Define 
\begin{equation}\label{Itilde}
\begin{aligned}
&\widetilde{\bs{\mathcal{I}}}:= E[ \widetilde{\bs{s}}(\bs{X},\bs{Z})\widetilde{\bs{s}}(\bs{X},\bs{Z})^{\top}],\ \widetilde{\bs{\mathcal{I}}}_{\bs{\eta}}:= E[\widetilde {\bs{s}}_{\bs{\eta}} \widetilde {\bs{s}}_{\bs{\eta}}^{\top}], \ \widetilde{\bs{\mathcal{I}}}_{\bs{\lambda}\bs{\eta}}:= E[\widetilde {\bs{s}}_{\bs{\lambda}} \widetilde {\bs{s}}_{\bs{\eta}}^{\top}], \ \\
&\widetilde{\bs{\mathcal{I}}}_{\bs{\eta\lambda}}:= \widetilde{\bs{\mathcal{I}}}_{\bs{\lambda}\bs{\eta}}^{\top},\quad \widetilde{\bs{\mathcal{I}}}_{\bs{\lambda}}:= E[\widetilde {\bs{s}}_{\bs{\lambda}} \widetilde {\bs{s}}_{\bs{\lambda}}^{\top}],\ \widetilde{\bs{\mathcal{I}}}_{\bs{\lambda}.\bs{\eta}} :=\widetilde{\bs{\mathcal{I}}}_{\bs{\lambda}} - \widetilde{\bs{\mathcal{I}}}_{\bs{\lambda}\bs{\eta}}\widetilde{\bs{\mathcal{I}}}_{\bs{\eta}}^{-1} \widetilde{\bs{\mathcal{I}}}_{\bs{\eta\lambda}}.
\end{aligned}
\end{equation} 
Let $\widetilde{\bs{G}}_{\bs{\lambda}.\bs{\eta}}=((\bs{G}_{\bs{\lambda}.\bs{\eta}}^{1})^{\top},\ldots,(\bs{G}_{\bs{\lambda}.\bs{\eta}}^{M_0})^\top)^\top \sim N(0,\widetilde{\bs{\mathcal{I}}}_{\bs{\lambda}.\bs{\eta}})$ be an $\mathbb{R}^{M_0 (d_{\mu v} + d_{\mu^4})}$--valued random vector, and define ${\bs{\mathcal{I}}}_{\bs{\lambda}.\bs{\eta}}^m:=E[\bs{G}_{\bs{\lambda}.\bs{\eta}}^m (\bs{G}_{\bs{\lambda}.\bs{\eta}}^m)^\top]$ and $\bs{Z}_{\bs{\lambda}}^m:=({\bs{\mathcal{I}}}_{\bs{\lambda}.\bs{\eta}}^m)^{-1}\bs{G}_{\bs{\lambda}.\bs{\eta}}^m$. For $j=1,2$, similar to $\widehat{\bs{t}}_{\bs{\lambda}}^{j}$ in the test of homogeneity, define $\widehat{\bs{t}}_{\bs{\lambda},m}^{j}$ by
\begin{equation*}
r^m(\widehat{\bs{t}}_{\bs{\lambda},m}^{j}) = \inf_{\bs{t}_{\bs{\lambda}} \in \Lambda_{\bs{\lambda} }^{j}}r^m(\bs{t}_{\bs{\lambda}}), \quad r^m(\bs{t}_{\bs{\lambda}}) := (\bs{t}_{\bs{\lambda}} -\bs{Z}_{\bs{\lambda}}^m)\t \bs{\mathcal{I}}_{\bs{\lambda}.\bs{\eta}}^m (\bs{t}_{\bs{\lambda}} -\bs{Z}_{\bs{\lambda}}^m),
\end{equation*}
where $\Lambda_{\bs{\lambda} }^{j}$ is defined in (\ref{Lambda-e}).
The following proposition gives the asymptotic null distribution of the LRTS for testing $H_{01}$. In the neighborhood of $\Upsilon_{1h}^*$, the log-likelihood function permits a quadratic approximation in terms of polynomials of the parameters similar to testing $H_0:M=1$ against $H_A:M=2$. Consequently, the LRTS is asymptotically distributed as the maximum of $M_0$ random variables.

\begin{assumption}\label{A-vec-2} (a) $\alpha_{j}^*\in [\epsilon_1,1-\epsilon_1]$ for $j = 1,\ldots,M_0$. (b) $\widetilde{\bs{\mathcal{I}}}$ defined in (\ref{Itilde}) is nonsingular.
\end{assumption}

\begin{proposition} \label{local_lr-2}
Suppose that Assumptions \ref{assn_consis}, \ref{A-taylor1}, and \ref{A-vec-2} hold  and $a_n$ in (\ref{pen_pmle}) satisfies $a_n = o(1)$. Then, under the null hypothesis $H_0: M=M_0$, $LR_{n}^{M_0}(\epsilon_1) \rightarrow_d \max\{v_1,\ldots, v_{M_0}\}$, where $v_m := \max\left\{ (\widehat{\bs{t}}_{\bs{\lambda},m}^{1})\t \bs{\mathcal{I}}_{\bs{\lambda}.\bs{\eta}}^m \widehat{\bs{t}}_{\bs{\lambda},m}^1, (\widehat{\bs{t}}_{\bs{\lambda},m}^{2})\t \bs{\mathcal{I}}_{\bs{\lambda}.\bs{\eta}}^m \widehat{\bs{t}}_{\bs{\lambda},m}^2\right\}$.
\end{proposition}

\section{EM test} \label{section:emtest}

Implementing the likelihood ratio test in Section \ref{sec:hetero} requires the researcher to choose a lower bound $\epsilon_1$ on $\alpha_j$ and assume $\alpha_j^* >\epsilon_1$.  In this section, we develop an EM test of $H_0:M= M_0$ against $H_A:M = M_0 + 1$  that does not require such a lower bound on $\alpha_j$.  For brevity, we  suppress covariate $\bs{Z}$ in this section. First, we develop an EM test statistic for testing $H_{0,1m}: (\bs{\mu}_m,\bs{\Sigma}_m)=(\bs{\mu}_{m + 1},\bs{\Sigma}_{m + 1})$. We construct $M_0$ sets $\{D_1^*,\cdots,D_{M_0}^*\}$ of admissible values of $(\bs{\mu},\bs{\Sigma})$, such that $D_m$ contains $(\bs{\mu}_m^*,\bs{\Sigma}_m^{*})$ but no other $(\bs{\mu}_j^*,\bs{\Sigma}_j^{*})$'s for $j \neq m$. For example, as in our simulation, we may assume that the first element of $\bs{\mu}$ are distinct, let $\overline{\mu}_j^*:=(\mu_{j1}^* +\mu_{j + 1,1}^*)/2$ with $\mu_{j1}$ denoting the first element of $\bs{\mu}_j$, and set $D_1^* = (-\infty , \overline{\mu}_{1}^*] \times\Theta_{\tilde{\bs{\mu}}}\times\Theta_{\bs{\Sigma}}$, $D_j^* =[\overline{\mu}_{j-1}^*, \overline{\mu}_{j}^*]\times\Theta_{\tilde{\bs{\mu}}}\times\Theta_{\bs{\Sigma}}$ for $j = 2,\ldots,M_0 - 1$, and $D_{M_0}^* = [\overline{\mu}_{M_0 - 1}^*, \infty ) \times\Theta_{\tilde{\bs{\mu}}} \times \Theta_{\bs{\Sigma}}$, where $\Theta_{\tilde{\bs{\mu}}}$ denotes the space of $\tilde{\bs{\mu}}:=(\mu_2,\ldots,\mu_d)\top$.

Collect the mixing parameters of the $(M_0 + 1)$-component model into one vector as $\bs{\varsigma}:=(\bs{\mu}_1,\ldots,\bs{\mu}_{M_0 + 1},\bs{\Sigma}_1,\ldots,\bs{\Sigma}_{M_0 + 1}) \in \Theta_{\bs{\varsigma}}:=\Theta_{\bs{\mu}}^{M_0 + 1}\times \Theta_{\bs{\Sigma}}^{M_0 + 1}$. For $m = 1,\ldots,M_0$, define a restricted parameter space of $\bs{\varsigma}$ by $\Xi_m^*: = \{\bs{\varsigma} \in \Theta_{\bs{\varsigma}}: (\bs{\mu}_j,\bs{\Sigma}_j) \in D_j^* \ \text{for}\ j = 1,\ldots, m - 1;\ (\bs{\mu}_m,\bs{\Sigma}_m),\ (\bs{\mu}_{m + 1},\bs{\Sigma}_{m + 1})\in D_m^*;\ (\bs{\mu}_j,\bs{\Sigma}_j) \in D_{j - 1}^* \ \text{for}\ j = m + 2,\ldots,M_0 + 1 \}$. Let $\widehat{\Xi}_m$ and $\widehat{D}_m$ be consistent estimates of $\Xi_m^*$ and $D_m^*$, which can be constructed from  the PMLE  of the $M_0$-component model. We test $H_{0,1m}: (\bs{\mu}_m,\bs{\Sigma}_m)=(\bs{\mu}_{m + 1},\bs{\Sigma}_{m + 1})$ by estimating the $(M_0 + 1)$-component model (\ref{fitted_model}) under the restriction $\bs{\varsigma} \in \widehat{\Xi}_m$. For example, when we test $H_{0,11}:(\bs{\mu}_1,\bs{\Sigma}_1) = (\bs{\mu}_2,\bs{\Sigma}_2)$ in a three-component model, the restriction can be given as $(\bs{\mu}_1,\bs{\Sigma}_1), (\bs{\mu}_2,\bs{\Sigma}_2) \in \widehat{D}_1$ and $(\bs{\mu}_3,\bs{\Sigma}_3) \in \widehat{D}_2$.

Define the penalty term $p_n^m(\bs{\vartheta}_{M_0+1})$ on $\bs{\Sigma}_j$'s as
\begin{equation}\label{penalty-em}
p_n^m(\bs{\vartheta}_{M_0+1}) := \sum_{j=1}^{M_0+1}p^m_{n}(\bs{\Sigma}_j;\bs{\Omega}_j), \quad p^m_{n}(\bs{\Sigma}_j;\bs{\Omega}_j):=- a_n \left\{ \text{tr}(\bs{\Omega}_j\bs{\Sigma}_j^{-1}) - \log ( \det({\bs{\Omega}}_j\bs{\Sigma}_j^{-1})) -d \right\},
\end{equation}
with $\bs{\Omega}_j = \widehat{\bs{\Sigma}}_j$ for $j=1,\ldots,m-1$, $\bs{\Omega}_j =\widehat{\bs{\Sigma}}_m$ for $j=m,m+1$, and $\bs{\Omega}_j =\widehat{\bs{\Sigma}}_{j-1}$ for $j=m+2,\ldots,M_0+1$,
where $\widehat{\bs{\Sigma}}_j$ is a consistent estimator of $\bs{\Sigma}_j$ from the $M_0$-component PMLE. This penalty term is a multivariate version of the one in \citet{chenlifu12jasa} and satisfies $p^m_{n}(\bs{\Omega}_j;\bs{\Omega}_j)=0$. Let $\mathcal{T}$ be a finite set of numbers from $(0,0.5]$, and let $p(\tau) \leq 0$ be a penalty term that is continuous in $\tau$, $p(0.5)=0$, and $p(\tau) \to -\infty$ as $\tau$ goes to $0$. 
 
For each $\tau_0 \in \mathcal{T}$, define the restricted penalized MLE as $\bs{\vartheta}_{M_0+1}^{m(1)}(\tau_0) := \\\mathop{\arg \max}_{\bs{\vartheta}_{M_0+1} \in \Theta^m(\tau_0)} (PL_n^m({\bs{\vartheta}}_{M_0 + 1}) + p(\tau_0))$, where $\Theta^m(\tau) := \{ \bs{\vartheta}_{M_0+1} \in \Theta_{\bs{\vartheta}_{M_0+1}}: \alpha_{m}/(\alpha_{m} + \alpha_{m + 1})=\tau \text{ and } \bs{\varsigma}\in \hat \Xi_m\}$ and $PL_{n}^m({\bs{\vartheta}}_{M_0+1}):= \sum_{i=1}^n f_{M_0+1}(\bs{X}_i ;\bs{\vartheta}_{M_0+1}) + p_n^m(\bs{\vartheta}_{M_0+1})$. Starting from $\bs{\vartheta}_{M_0+1}^{m(1)}(\tau_0)$, we update $\bs{\vartheta}_{M_0+1}$ by the following generalized EM algorithm. Henceforth, we suppress $(\tau_0)$ from $\bs{\vartheta}_{M_0+1}^{m(k)}(\tau_0)$. Suppose we have already calculated $\bs{\vartheta}_{M_0+1}^{m(k)}$. For $i = 1,\ldots,n$ and $j = 1,\ldots,M_0+1$, define the weights for an E-step as 
\begin{align*}
w_{ij}^{(k)} &:=
\begin{cases}
\alpha_j^{(k)} f(\bs{X}_i;{\bs{\mu}}_j^{(k)},\bs{\Sigma}_j^{(k)}) / f_{M_0+1}(\bs{X}_i;{\bs{\vartheta}}_{M_0+1}^{m(k)}) & \mbox{for } j=1,\ldots,m-1,\\
\alpha_{j}^{(k)} f(\bs{X}_i;{\bs{\mu}}_j^{(k)},\bs{\Sigma}_j^{(k)}) / f_{M_0+1}(\bs{X}_i;{\bs{\vartheta}}_{M_0+1}^{m(k)}) & \mbox{for } j=m+2,\ldots,M_0+1,\\
\end{cases} \\
w_{im}^{(k)} &:= \frac{\tau^{(k)}(\alpha_m^{(k)}+\alpha_{m+1}^{(k)}) f(\bs{X}_i;{\bs{\mu}}_m^{(k)},\bs{\Sigma}_m^{(k)})}{f_{M_0+1}(\bs{X}_i;{\bs{\vartheta}}_{M_0+1}^{m(k)})}, \ w_{i,m+1}^{(k)} := \frac{(1-\tau^{(k)})(\alpha_m^{(k)}+\alpha_{m+1}^{(k)})f(\bs{X}_i;{\bs{\mu}}_{m+1}^{(k)},\bs{\Sigma}_{m+1}^{(k)}) }{f_{M_0+1}(\bs{X}_i;{\bs{\vartheta}}_{M_0+1}^{m(k)})}.
\end{align*}
In an M-step, update $\tau$ and $\bs{\alpha}$ by
\begin{align*}
\tau^{(k+1)} &:= \mathop{\arg \max}_{\tau} \left\{\sum_{i=1}^n w_{im}^{(k)} \log(\tau) + \sum_{i=1}^n w_{i,m+1}^{(k)} \log(1-\tau) + p(\tau) \right\},\\
\alpha_j^{(k+1)} &:= 
n^{-1} \sum_{i=1}^n w_{ij}^{(k)} \quad \mbox{for } j=1, \ldots,M_0+1, 
\end{align*}
and update $\bs{\mu}_j$ and $\bs{\Sigma}_j$ for $j=1,\ldots,M_0+1$ by
\begin{equation*} 
\begin{aligned}
\bs{\mu}_j^{(k+1)} & : = \frac{\sum_{i=1}^n w_{ij}^{(k)} \bs{X}_i}{\sum_{i=1}^n w_{ij}^{(k)} },\quad
\bs{\Sigma}_j^{(k+1)} : = \frac{2a_n \bs{\Omega}_j + \bs{S}_j^{(k+1)}}{2a_n + \sum_{i=1}^n w_{ij}^{(k)}},
 \\
\end{aligned}
\end{equation*}
where $\bs{S}_j^{(k+1)}:= \sum_{i=1}^nw_{ij}^{(k)} \left(\bs{X}_i-\bs{\mu}_j^{(k+1)}\right) \left(\bs{X}_i-\bs{\mu}_j^{(k+1)}\right)\t$.
The penalized likelihood value never decreases after each generalized EM step \citep[Theorem 1]{dempster77jrssb}. Note that $\bs{\vartheta}_{M_0+1}^{m(k)}$ for $k \geq 2$ does not use the restriction $\hat \Xi_m$. For each $\tau_0 \in \mathcal{T}$ and $k$, define
\begin{equation*}
\text{M}_{n}^{m(k)}(\tau_0) : = 2\left\{PL_{n}^m(\bs{\vartheta}_{M_0+1}^{m(k)}(\tau_0) ) +p(\tau^{(k)}) - L_{0,n}(\widehat{\bs{\vartheta}}_{M_0}) \right\}, 
\end{equation*} 
where $\widehat{\bs{\vartheta}}_{M_0}$ and $L_{0,n}({\bs{\vartheta}}_{M_0})$ are defined in (\ref{PMLEs}). 

Finally, with a pre-specified number $K$, define the \textit{local EM test statistic} for testing $H_{0,1m}$ by taking the maximum of $\text{M}_n^{m(K)}(\tau_0)$ over $\tau_0\in\mathcal{T}$ as $\text{EM}_{n}^{m(K)} : = \max \{\text{M}_{n}^{m(K)}(\tau_0): \tau_0 \in \mathcal{T} \}$. The \textit{EM test statistic} is defined as the maximum of $M_0$ local EM test statistics:
\begin{equation}\label{EM-test}
\text{EM}_{n}^{(K)}: = \max\left\{\text{EM}_{n}^{1(K)},\text{EM}_{n}^{2(K)},\ldots,\text{EM}_{n}^{M_0(K)}\right\}.
\end{equation}
The following proposition shows that for any finite $K$, the EM test statistic is asymptotically equivalent to the penalized LRTS for testing $H_{01}$.

\begin{proposition} \label{EM_stat-1}
Suppose that  Assumptions \ref{assn_consis} and \ref{A-vec-2} hold, $a_n$ in (\ref{penalty-em}) satisfies $a_n = O(1)$, and $\{0.5\} \in \mathcal{T}$. Then, under the null hypothesis $H_0: M=M_0$, for any fixed finite $K$, $\text{EM}_{n}^{{(K)}} \rightarrow_d \max\{v_1,\ldots, v_{M_0}\}$ as $n \rightarrow \infty$, where the $v_m$'s are given in Proposition \ref{local_lr-2}.
\end{proposition}

\section{Asymptotic distribution under local alternatives} 
In this section, we derive the asymptotic distribution of our LRTS  and EM test statistic  under local alternatives. For brevity, we focus on the case of testing $H_0: M=1$ against $H_A: M=2$.  

Given a local parameter $\bs{h} =(\bs{h}_{\bs \eta}\t,
\bs{h}_{\bs \lambda }\t)\t$ and $\alpha  \in (\epsilon_1,1-\epsilon_1)$, we consider the sequence of contiguous local alternatives $\bs{\vartheta}_{n} = (\bs{\psi}_n\t,\alpha_n)\t = (\bs{\eta}_n\t,\bs\lambda_n\t,\alpha_n)\t\in\Theta_{\bs{\eta}}\times\Theta_{\bs\lambda}\times\Theta_{\alpha}$ such that, with $\bs{t}_{\bs\lambda}(\bs\lambda,\alpha)$ given by (\ref{tpsi_defn}), 
\begin{equation}\label{local-alternative}
\bs{h}_{\bs \eta}= \sqrt{n}(\bs{\eta}_n-\bs\eta^*),\quad \bs{h}_{\bs\lambda} =\sqrt{n} \bs{t}_{\bs\lambda}(\bs\lambda_{n},\alpha_n)+o(1),\quad \text{and }\ \alpha_n = \alpha+o(1).
\end{equation} 

Let $\mathbb{P}_{\bs{\vartheta}}^n$ be the probability measure on $\{\bs{X}_i\}_{i=1}^n$ conditional on $\{\bs{Z}_i\}_{i=1}^n$ under $\bs{\vartheta}$. Then, for the density (\ref{loglike}), the log-likelihood ratio is given by
\[
\log \frac{d\mathbb{P}_{\bs{\vartheta}_n}^n}{d \mathbb{P}_{\bs{\vartheta}^*}^n} = L_n(\bs{\psi}_n,\alpha_n)-L_n(\psi^*,\alpha)= \sum_{i=1}^n \log \left( \frac{g(X_i|Z_i;\bs{\eta}_n,\bs{\lambda}_n,\alpha_n)}{g(X_i|Z_i;\bs{\eta}^*,\bs{0},\alpha)} \right).
\]  

The following proposition provides the asymptotic distribution of the LRTS under contiguous local alternatives. 
\begin{proposition} \label{P-LAN2} Suppose that the assumptions of Proposition \ref{P-LR-N1} hold. Consider a sequence of contiguous local alternatives $\bs{\vartheta}_{n} = ((\bs{\eta}^*)\t,(\bs{\lambda}_n) \t,\alpha_n)$, where $\bs{\lambda}_n$ and $\alpha_n$  are given by (\ref{local-alternative}). Then, under $H_{1n}: \bs{\vartheta} = \bs{\vartheta}_{n}$, we have 
$LR_{n}(\epsilon_1) \rightarrow_d \max\left\{ (\widehat{\bs{t}}_{\bs{\lambda}}^{1})\t \bs{\mathcal{I}}_{\bs{\lambda}.\bs{\eta}} \widehat{\bs{t}}_{\bs{\lambda}}^{1}, (\widehat{\bs{t}}_{\bs{\lambda}}^{2})\t \bs{\mathcal{I}}_{\bs{\lambda}.\bs{\eta}} \widehat{\bs{t}}_{\bs{\lambda}}^{2} \right\}$, where $\widehat{\bs{t}}_{\bs{\lambda}}^1$ and $\widehat{\bs{t}}_{\bs{\lambda}}^2$ are defined as in (\ref{t-lambda}) but replacing $\bs{Z_{\lambda}}$ with $\left(\bs{I_{\lambda,\eta}}\right)^{-1} \bs{G_{\lambda.\eta}}+\bs{h}_{\bs\lambda}$. 
\end{proposition} 

In this proposition, the local alternatives are implicitly defined through the condition that $\bs{h}_{\bs\lambda} =\sqrt{n} \bs{t}_{\bs\lambda} (\bs\lambda_{n} ,\alpha_n)+o(1)$. We now give an example for $d=1$, where we explicitly construct local alternatives with different orders, including the ones with order $n^{1/8}$.
\begin{example} When $d=1$, for $\overline\alpha\in(\epsilon_1,1-\epsilon_1)$ and $\overline{\bs{\lambda}}:=(\overline{{\lambda}}_{\mu},\overline{{\lambda}}_{v})\t$, define  
\begin{align*}
&H_{1n}^a : \bs{\vartheta}_n^a=((\bs{\eta}_n^a)\t,(\bs{\lambda}_{n}^a)\t,\alpha_n^a)\t : = ((\bs\eta^*)\t,(\overline{\bs{\lambda}}/n^{1/4})\t,\overline\alpha + o(1))\t,\\
&H_{1n}^b : \bs{\vartheta}_n^b=((\bs{\eta}_n^b)\t,(\bs{\lambda}_{n}^b)\t,\alpha_n^b)\t : = ((\bs\eta^*)\t,\overline{{\lambda}}_{\mu}/n^{1/8},\overline{{\lambda}}_{v}/n^{3/8},\overline\alpha + o(1))\t.
\end{align*} 
Then, for $j\in\{a,b\}$, $\bs{h}_{\bs\lambda}^j =\sqrt{n} \bs{t}_{\bs\lambda} (\bs\lambda_{n}^j ,\alpha_n^j)+o(1)$ holds with $\bs{h}_{\bs\lambda}^a := 12\overline\alpha(1-\overline\alpha)\times ( \overline{{\lambda}}_{\mu} \overline{{\lambda}}_{v}, \overline{{\lambda}}_{v}^2)$ and $\bs{h}_{\bs\lambda}^b := \overline\alpha(1-\overline\alpha)\times ( 12\overline{{\lambda}}_{\mu} \overline{{\lambda}}_{v}, b(\overline\alpha)\overline{{\lambda}}_{\mu}^4 )\t$. Therefore, Proposition \ref{P-LAN2} gives the asymptotic distribution of $LR_{n}(\epsilon_1)$. 
\end{example}

\section{Parametric bootstrap} 

Given that it may not be easy to simulate the asymptotic distributions of the LRTS and the EM test statistic for testing $H_0: M=M_0$ against $H_A: M=M_0+1$, we provide the validity of parametric bootstrap. We consider the following parametric bootstrap to obtain the bootstrap critical value $c_{\alpha,B}$ and bootstrap $p$-value.
\begin{enumerate}
\item Using the observed data, compute $\widehat{\bs{\vartheta}}_{M_0}$ and compute $LR_{n}^{M_0}(\epsilon_1)$ in (\ref{LR-M_0}) and $\text{EM}_{n}^{{(K)}}$ in (\ref{EM-test}).
\item Given $\widehat{\bs{\vartheta}}_{M_0}$, generate $B$ independent samples $\{\bs X_1^b,\ldots,\bs X_n^b\}_{b=1}^B$ under $H_0$ with $\bs\vartheta_{M_0}=\widehat{\bs{\vartheta}}_{M_0}$ conditional on the observed value of $\{\bs Z_1,\ldots,\bs Z_n\}$.
\item For each simulated sample $\{\bs X_1^b,\ldots,\bs X_n^b\}$ with $\{\bs Z_1,\ldots,\bs Z_n\}$, compute $LR_{n}^{M_0,b}(\epsilon_1)$ and $\text{EM}_{n}^{{(K),b}}$ as in Step 1 for $b=1,\ldots,B$.
\item Let $c_{\alpha,B}$ be the $(1-\alpha)$ quantile of $\{LR_{n}^{M_0,b}\}_{b=1}^B$ or $\{EM_{n}^{(K),b}\}_{b=1}^B$, and define the bootstrap $p$-value as $B^{-1}\sum_{b=1}^B \mathbb{I}\{ LR_{n}^{M_0,b} > LR_{ n}^{M_0}\}$ or $B^{-1}\sum_{b=1}^B \mathbb{I}\{ EM_{n}^{(K),b} > EM_{ n}^{(K)}\}$.
\end{enumerate} 
The following proposition shows the consistency of the bootstrap critical values $c_{\alpha,B}$ for testing $H_0: M=1$. The case of testing $H_0:M=M_0$ for $M_0\geq 2$ can be proven similarly. 
\begin{proposition} \label{P-bootstrap} 
Suppose that the assumptions of Proposition \ref{P-LR-N1} holds. Then, the bootstrap critical values $c_{\alpha,B}$ converge to the asymptotic critical values in probability as $n$ and $B$ go to infinity under $H_0$ and under the local alternatives described in Propositions \ref{P-LAN2}.
\end{proposition}

\section{  Homoscedastic multivariate finite normal mixture models} 

In this section, we consider testing the order of homoscedastic multivariate normal mixtures. We consider the likelihood ratio test but do not consider the EM test because, unlike heteroscedastic normal mixtures, homoscedastic normal mixture models do not suffer from infinite Fisher information and unbounded likelihood.

\subsection{Likelihood ratio test of $H_0: M = 1$ against $H_A: M = 2$}

Consider a two-component normal mixture density function with common variance:
\begin{equation*}
f_2(\bs{x}|\bs{z};\bs{\vartheta}_2) := \alpha f(\bs{x}|\bs{z}; \bs{\gamma} ,\bs{\mu}_1,\bs{\Sigma}) + (1-\alpha) f(\bs{x}|\bs{z}; \bs{\gamma},\bs{\mu}_2,\bs{\Sigma}), 
\end{equation*}
with ${\bs{\vartheta}}_2 = (\alpha,\bs{\gamma}, \bs{\mu}_1, \bs{\mu}_2,\bs{\Sigma}) \in \Theta_{{\bs{\vartheta}}_2}$. We assume $\Theta_{{\bs{\vartheta}}_2}$ is compact.
\begin{assumption} \label{assn_compact}
The parameter space $\Theta_{{\bs{\vartheta}}_2}$ is compact.
\end{assumption}
Assume $\alpha\in [0,3/4]$ without loss of generality. Then, the null hypothesis $H_0:M=1$ is written as
\[
H_0: \alpha(\bs{\mu}_1-\bs{\mu}_2)=0.
\]
For an arbitrary small $\zeta>0$, we partition the parameter space as $\Theta_{{\bs{\vartheta}}_2} = \Theta_{{\bs{\vartheta}}_2,\zeta}^1 \cup \Theta_{{\bs{\vartheta}}_2,\zeta}^2 $, where  
\[
\Theta_{{\bs{\vartheta}}_2,\zeta}^1 =\{{\bs{\vartheta}}_2\in \Theta_{{\bs{\vartheta}}_2} : |\bs{\mu}_{1}- \bs{\mu}_{2}|\leq \zeta \}\ \ \text{and}\ \ \Theta_{{\bs{\vartheta}}_2,\zeta}^2 =\{{\bs{\vartheta}}_2\in \Theta_{{\bs{\vartheta}}_2}: |\bs{\mu}_{1}- \bs{\mu}_{2}|  \geq \zeta \}.
\] 
Let $L_n({\bs{\vartheta}}_2): = \sum_{i = 1}^n \log f_2( \bs{X}_i|\bs{Z}_i;{\bs{\vartheta}}_2)$ denote the log-likelihood function, and define the two-component MLE by $\widehat{\bs{\vartheta}}_2 :=\arg\max_{ {\bs{\vartheta}}_2 \in \Theta_{{\bs{\vartheta}}_2}} L_n({\bs{\vartheta}}_2)$. Define the restricted two-component MLE by $\widehat{\bs{\vartheta}}_{2,\zeta}^j :=\arg\max_{ {\bs{\vartheta}}_2^j \in \Theta_{{\bs{\vartheta}}_2,\zeta}^j } L_n({\bs{\vartheta}}_2)$ for $j=1,2$ so that $L_n(\widehat {\bs{\vartheta}}_2) = \max\{ L_n(\widehat {\bs{\vartheta}}_{2,\zeta}^1 ),L_n(\widehat {\bs{\vartheta}}_{2,\zeta}^2 ) \}$. Let $(\widehat{\bs{\gamma}}_0,\widehat{{\mu}}_0,\widehat{{\sigma}}^2_0)$ denote the one-component MLE that maximizes the one-component log-likelihood function $L_{0,n}(\bs{\gamma},{ {\mu}},\bs{\sigma}) := \sum_{i=1}^n \log f ( \bs{X}_i|\bs{Z}_i;\bs{\gamma},{{\mu}}, {\sigma^2})$. Define the LRTS for testing $H_0$ as $LR_{n}:= 2\{L_n(\widehat {\bs{\vartheta}}_2) - L_{0,n}(\widehat{\bs{\gamma}}_0,\widehat{{\mu}}_0,\widehat{{\sigma}}^2_0)\}=\max \{LR_{n,\zeta}^1 ,LR_{n,\zeta}^2 \}$, where $LR_{n,\zeta}^j := 2\{L_n(\widehat{\bs{\vartheta}}_{2,\zeta}^j ) - L_{0,n}(\widehat{\bs{\gamma}}_0,\widehat{{\mu}}_0,\widehat{{\sigma}}^2_0)\}$. 

In the following, we derive the asymptotic null distribution of $LR_{n,\zeta}^1$, $LR_{n,\zeta}^2$, and $LR_{n}$ by using a similar approach to the heteroscedastic case. We introduce a reparameterization that extracts the direction in which the Fisher information matrix is singular and approximate the log-likelihood in terms of the polynomials of the reparameterized parameters.

\subsubsection{Asymptotic distribution of $LR_{n,\zeta}^1$}
In this section, we derive the asymptotic distribution of $LR_{n,\zeta}^1$. Let $\bs{v}=\bs{w}(\bs{\Sigma})$ and consider the following one-to-one mapping between $(\bs{\mu}_1,\bs{\mu}_2,\bs{v})$ and the reparameterized parameter $(\bs{\lambda} ,\bs{\nu}_{\bs\mu},\bs{\nu}_{\bs{v}})$:
\begin{equation} \label{repara2-homo}
\begin{pmatrix}
{\bs{\mu}}_1\\
{\bs{\mu}}_2\\
\bs{v}
\end{pmatrix}
 =
\begin{pmatrix}
\bs{\nu}_{\bs\mu} + (1-\alpha) \bs{\lambda} \\
\bs{\nu}_{\bs\mu} -\alpha \bs{\lambda} \\
\bs{\nu}_{\bs{v}} - \alpha(1-\alpha) \bs{w}(\bs{\lambda} \bs{\lambda} \t)
\end{pmatrix}.
\end{equation}
In the reparameterized model, the density is given by
\begin{equation} \label{loglike-homo}
\begin{aligned}
g(\bs{x}|\bs{z};\bs{\psi},\alpha) & = \alpha f_v\left(\bs{x}\middle|\bs{z};\bs{\gamma},\bs{\nu}_{\bs\mu}+(1-\alpha)\bs{\lambda}, \bs{\nu}_{\bs{v}} - \alpha(1-\alpha) \bs{w}(\bs{\lambda}\bs{\lambda}\t)\right) \\
& \quad + (1 - \alpha) f_v \left(\bs{x} \middle|\bs{z};\bs{\gamma},\bs{\nu}_{\bs\mu} -\alpha\bs{\lambda},\bs{\nu}_{\bs{v}} - \alpha(1-\alpha) \bs{w}(\bs{\lambda}\bs{\lambda}\t) \right).
\end{aligned}
\end{equation}

We partition $\bs{\psi}$ as $\bs{\psi} = (\bs{\eta}\t,\bs{\lambda}\t)\t$, where $\bs{\eta}: = (\bs{\gamma}\t,\bs{\nu}_{\bs\mu}\t,\bs{\nu}_{\bs{v}}\t)\t \in \Theta_{\bs{\eta}}$  and $\bs{\lambda} \in \Theta_{\bs{\lambda}}$. Denote the true values of $\bs{\eta}$, $\bs{\lambda}$, and $\bs{\psi}$ under $H_{0}$ by $\bs{\eta}^*: = ((\bs{\gamma}^*)\t,({\bs{\mu}}^*)\t,(\bs{v}^{*})\t)\t$, $\bs{\lambda}^*: = \bs{0}$, and $\bs{\psi}^* = ((\bs{\eta}^*)\t, \bs{0}\t)\t$, respectively. The first derivative of (\ref{loglike-homo}) w.r.t.\ $\bs{\eta}$ under $\bs{\psi} = \bs{\psi}^*$ is given by (\ref{dvareta}) and the first and second derivatives of $g(\bs{x}|\bs{z};\bs{\psi},\alpha)$ w.r.t.\ $\bs{\lambda}$ become zero when evaluated at $\bs{\psi}=\bs{\psi}^*$. Consequently, the information on $\bs{\lambda}$ is provided by the third and fourth derivatives w.r.t.\ $\bs{\lambda}$. 

Define the score vector $\bs{s}(\bs{x},\bs{z})$ as
\begin{equation} \label{score_defn-homo}
\begin{aligned}
\bs{s}(\bs{x},\bs{z}) & := 
\begin{pmatrix}
\bs{s}_{\bs{\eta}}\\
\bs{s}_{\bs{\lambda}}
\end{pmatrix} := 
\begin{pmatrix}
\bs{s}_{\bs{\eta}}\\
\bs{s}_{\bs{\mu^3}}\\
\bs{s}_{\bs{\mu}^4}
\end{pmatrix}
\quad \text{with }\
\underset{(d \times 1)}{\bs{s}_{\bs{\eta}}} := \frac{\nabla_{(\bs{\gamma}\t,\bs{\mu}\t,\bs{v}\t)\t}f^*_v }{ f^*_v}, \\
\underset{(d_{\mu^3} \times 1)}{\bs{s}_{\bs{\mu}^3}} &:= \left\{\frac{\nabla_{\mu_i \mu_j \mu_k} f^*_v}{3! f^*_v} \right\}_{1 \leq i \leq j \leq k \leq d}, \quad
\underset{(d_{\mu^4} \times 1)}{\bs{s}_{\bs{\mu}^4}} := \left\{ \frac{\nabla_{\mu_i \mu_j \mu_k \mu_\ell} f^*_v }{ 4!f^*_v} \right\}_{1 \leq i \leq j \leq k \leq \ell \leq d} ,
\end{aligned}
\end{equation}
where we suppress the dependence of $(\bs{s}_{\bs{\eta}}, \bs{s}_{\bs{\mu}^3}, \bs{s}_{\bs{\mu}^4})$ on $(\bs{x},\bs{z})$. Collect the relevant reparameterized parameters as
\begin{equation} \label{tpsi_defn-homo}
\bs{t}(\bs{\psi},\alpha) :=
\begin{pmatrix}
\bs{\eta}-\bs{\eta}^*\\
\bs{t}_{\bs{\lambda}}(\bs{\lambda},\alpha)
\end{pmatrix}
:=
\begin{pmatrix}
\bs{\eta}-\bs{\eta}^*\\
 \alpha(1-\alpha)(1-2\alpha)\bs{\lambda}_{\bs{\mu}^3}\\
\alpha(1-\alpha)(1-6\alpha+6\alpha^2)\bs{\lambda}_{\bs{\mu}^4}
\end{pmatrix},
\end{equation}
where $\underset{(d_{\mu^3} \times 1)}{\bs{\lambda}_{\bs{\mu}^3}} = \{(\bs{\lambda}_{\bs{\mu}^3})_{ijk}\}_{1 \leq i \leq j \leq k \leq d}$ with $(\bs{\lambda}_{\bs{\mu }^3})_{ijk} := \sum_{(t_1,t_2,t_3) \in p(i,j,k)} \lambda_{t_1} \lambda_{t_2} \lambda_{t_3}$, where $\sum_{(t_1,t_2,t_3) \in p(i,j,k)}$ denotes the sum over all distinct permutations of $(i,j,k)$ to $(t_1,t_2,t_3)$, and $\bs{\lambda}_{\bs{\mu}^4}$ is given in (\ref{lambda_muv_defn}). 
 
The third and fourth order derivatives of the density ratio w.r.t. $\bs{\lambda}$ are given by $\alpha(1-\alpha)(1-2\alpha)\bs{\lambda}_{\bs{\mu}^3}\t \bs{s}_{\bs{\mu^3}}$ and $\alpha(1-\alpha)(1-6\alpha+6\alpha^2)\bs{\lambda}_{\bs{\mu}^4}\t\bs{s}_{\bs{\mu}^4}$, respectively. When $\alpha$ is bounded away from $1/2$, the third order derivative identifies $\bs{\lambda}$ because the third order derivative dominates the fourth order derivative as $\bs{\lambda}\to 0$. 
When $\alpha=1/2$, the third order derivative is identically equal to zero and the fourth order derivative identifies $\bs{\lambda}$. When $\alpha$ is in the neighborhood of $1/2$ such that $1-2\alpha \propto \bs{\lambda}$, the third and fourth order derivatives jointly identify $\bs{\lambda}$. 

Accordingly, we characterize  the limit of possible values of $\sqrt{n}\bs{t}_{\bs{\lambda}}(\bs{\lambda},\alpha)$ defined in (\ref{tpsi_defn-homo}) as $n\rightarrow\infty$ by the following two sets: 
\begin{equation} \label{Lambda-e-homo}
\begin{aligned}
\Lambda_{\bs{\lambda} }^1 &:= \left\{\left( \bs{t}_{\bs{\mu}^3}\t, \bs{t}_{\bs{\mu}^4}\t \right)\t\in \mathbb{R}^{d_{\mu^3}+d_{\mu^4}}:\ \bs{t}_{\bs{\mu}^3}= \bs{\lambda_{\mu^3}},\ \bs{t}_{\bs{\mu}^4}= \bs{0} \text{ for some $ \bs{\lambda} \in \mathbb{R}^{d}$ }\right\}, \\
\Lambda_{\bs{\lambda} }^2 &:= \left\{\left( \bs{t}_{\bs{\mu}^3}\t, \bs{t}_{\bs{\mu}^4}\t \right)\t\in \mathbb{R}^{d_{\mu^3}+d_{\mu^4}}:\ \bs{t}_{\bs{\mu}^3}=c \bs{\lambda_{\mu^3}},\ \bs{t}_{\bs{\mu}^4}=- \bs{\lambda_{\mu^4}}\text{ for some $(\bs{\lambda}\t,c)\t \in \mathbb{R}^{d+1}$ }\right\},
\end{aligned}
\end{equation} 
where $\Lambda_{\bs{\lambda} }^1$ represents the case when $\alpha$ is bounded away from $1/2$ while, by choosing different values of $c$,  $\Lambda_{\bs{\lambda} }^2$ represents both  cases when $\alpha=1/2$  and when   $\alpha$ is in the neighborhood of $1/2$.

Define $\widehat{\bs{t}}_{\bs{\lambda}}^j $ by
\begin{equation} \label{t-lambda-homo}
r(\widehat{\bs{t}}_{\bs{\lambda}}^j ) = \inf_{\bs{t}_{\bs{\lambda}} \in \Lambda_{\bs{\lambda}}^j } r(\bs{t}_{\bs{\lambda}}), \quad r(\bs{t}_{\bs{\lambda}}) := (\bs{t}_{\bs{\lambda}} -\bs{Z}_{\bs{\lambda}})\t \bs{\mathcal{I}}_{\bs{\lambda}.\bs{\eta}} (\bs{t}_{\bs{\lambda}} -\bs{Z}_{\bs{\lambda}}),
\end{equation}
where $\Lambda_{\bs{\lambda} }^j$ for $j=1,2$ is defined in (\ref{Lambda-e-homo}) while $\bs{\mathcal{I}}_{\bs{\lambda}.\bs{\eta}}$ and $\bs{Z}_{\bs{\lambda}}$ are defined by 
\begin{align*}
& \bs{\mathcal{I}}_{\bs{\lambda}.\bs{\eta}}:=\bs{\mathcal{I}}_{\bs{\lambda}}-\bs{\mathcal{I}}_{\bs{\lambda\eta}}\bs{\mathcal{I}}_{\bs{\eta}}^{-1}\bs{\mathcal{I}}_{\bs{\eta\lambda}} \quad\text{and}\quad \bs{Z}_{\bs{\lambda}}:=(\bs{\mathcal{I}}_{\bs{\lambda}.\bs{\eta}})^{-1} \bs{G}_{\bs{\lambda}.\bs{\eta}}   \quad\text{with}\\ 
&\bs{\mathcal{I}}_{\bs{\eta}} := E[\bs{s}_{\bs{\eta}}\bs{s}_{\bs{\eta}}\t], \quad \bs{\mathcal{I}}_{\bs{\lambda}} := E[\bs{s}_{\bs{\lambda}}\bs{s}_{\bs{\lambda}}\t], \quad \bs{\mathcal{I}}_{\bs{\lambda\eta} } := E[\bs{s}_{\bs{\lambda}}\bs{s}_{\bs{\eta}}\t],\quad \bs{\mathcal{I}}_{\bs{\eta \lambda}} := \bs{\mathcal{I}}_{\bs{\lambda \eta}}\t,  
\end{align*}
given $(\bs{s}_{\bs{\eta}}, \bs{s}_{\bs{\lambda}})$ in (\ref{score_defn-homo}), where $\bs{G}_{\bs{\lambda}.\bs{\eta}} \sim N(0,\bs{\mathcal{I}}_{\bs{\lambda}.\bs{\eta}})$. 

The following proposition establishes the asymptotic null distribution of $LR_{n,\zeta}^1$. 
\begin{proposition} \label{P-LR-N1-homo-1}
Suppose that Assumptions \ref{assn_consis} and \ref{A-taylor1} hold and $\bs{\mathcal{I}} := E[\bs{s}(\bs{X},\bs{Z})\bs{s}(\bs{X},\bs{Z})\t]$ is finite and nonsingular. Then, under the null hypothesis of $H_{01}: \mu_1=\mu_2$, for any $\zeta>0$, $LR_{n,\zeta}^1 \rightarrow_d \max\left\{ (\widehat{\bs{t}}^1_{\bs{\lambda}})\t \bs{\mathcal{I}}_{\bs{\lambda}.\bs{\eta}} \widehat{\bs{t}}^1_{\bs{\lambda}}, (\widehat{\bs{t}}^2_{\bs{\lambda}})\t \bs{\mathcal{I}}_{\bs{\lambda}.\bs{\eta}} \widehat{\bs{t}}^2_{\bs{\lambda}} \right\}$.
\end{proposition} 
\begin{example}
When $d=1$ with $\bs{\lambda}=\lambda$, we have $\bs{s_{\mu^3}} = {\nabla_{\mu^3 } f^*_v}/{3! f^*_v}$, $\bs{s_{\mu^4}} = {\nabla_{\mu^4} f^*_v}/{4! f^*_v}$, and the possible values of $\sqrt{n}\bs{t}_{\bs{\lambda}}(\bs{\lambda},\alpha)=\sqrt{n}\alpha(1-\alpha) \left( (1-2\alpha)\lambda^3, (1-6\alpha+6\alpha^2)\lambda^4\right)\t$ as $n\to \infty$ are given by $\Lambda_{\lambda}^1\cup \Lambda_{\lambda}^2=\mathbb{R}\times \mathbb{R}_-$. In this case, $LR_{n,\zeta}^1 \rightarrow_d (\widehat{\bs{t}}_{\bs{\lambda}})\t \bs{\mathcal{I}}_{\bs{\lambda}.\bs{\eta}} \widehat{\bs{t}}_{\bs{\lambda}} $ with $\widehat{\bs{t}}_{\bs{\lambda}} $ defined by
$r(\widehat{\bs{t}}_{\bs{\lambda}} ) = \inf_{\bs{t}_{\bs{\lambda}} \in \mathbb{R}\times \mathbb{R}_-} r(\bs{t}_{\bs{\lambda}})$.
 When $d=2$ with $\bs{\lambda}=(\lambda_1,\lambda_2)\t$, we have   $\Lambda_{\bs{\lambda} }^1= \left\{(\bs{\lambda_{\mu^3}}\t,\bs{0}\t)\t: (\lambda_1,\lambda_2)\t \in \mathbb{R}^2\right\}$ and $
\Lambda_{\bs{\lambda} }^2 := \left\{(c\bs{\lambda_{\mu^3}}\t,-\bs{\lambda_{\mu^4}}\t)\t: (\lambda_1,\lambda_2,c)\t \in \mathbb{R}^3\right\}$ with $\bs{\lambda_{\mu^3}} = (\lambda_1^3,3\lambda_1^2\lambda_2,3\lambda_1\lambda_2^2,\lambda_2^3)\t$ and $\bs{\lambda_{\mu^4}} = (\lambda_1^4,4\lambda_1^3\lambda_2,6\lambda_1^2\lambda_2^2,4\lambda_1\lambda_2^3,\lambda_2^4)\t$. 
\end{example}

\subsubsection{Asymptotic distribution of $LR_{n,\zeta}^2$}

This section derives the asymptotic distribution of $LR_{n,\zeta}^2$. We use the reparameterization (\ref{repara2-homo}) but collect the reparameterized parameters into $\bs{\phi}:=(\bs{\eta}\t,\alpha)\t$ and $\bs{\lambda}$, where $\bs{\eta}: = (\bs{\gamma}\t,\bs{\nu}_{\bs\mu}\t,\bs{\nu}_{\bs{v}}\t)\t$. Let the resulting density be
\begin{equation} \label{loglike-homo-2}
\begin{aligned}
h(\bs{x}|\bs{z};\bs{\phi},\bs{\lambda}) & := \alpha f_v\left(\bs{x}\middle|\bs{z};\bs{\gamma},\bs{\nu}_{\bs\mu}+(1-\alpha)\bs{\lambda}, \bs{\nu}_{\bs{v}} - \alpha(1-\alpha) \bs{w}(\bs{\lambda}\bs{\lambda}\t)\right) \\
& \quad + (1 - \alpha) f_v \left(\bs{x} \middle|\bs{z};\bs{\gamma},\bs{\nu}_{\bs\mu} -\alpha\bs{\lambda},\bs{\nu}_{\bs{v}} - \alpha(1-\alpha) \bs{w}(\bs{\lambda}\bs{\lambda}\t) \right).
\end{aligned}
\end{equation} 
The right hand side of (\ref{loglike-homo-2}) is the same as that of (\ref{loglike-homo}). When we restrict the parameter space to $\Theta_{{\bs{\vartheta}}_2,\zeta}^2$, the reparameterized density $h(\bs{x}|\bs{z};\bs{\phi},\bs{\lambda})$ in (\ref{loglike-homo-2}) becomes the one-component density if and only if $\alpha =0$. Furthermore, $\bs{\lambda}$ is not identified when $\alpha=0$. Denote the true value of $\bs{\phi}$ under $H_{0}$ by $\bs{\phi}^* = ((\bs{\eta}^*)\t, 0)\t$, where $\bs{\eta}^* := ((\bs{\gamma}^*)\t,({\bs{\mu}}^*)\t,(\bs{v}^{*})\t)\t$. The MLE of $\bs{\phi}$ under the restriction $\bs{\vartheta}_2\in\Theta_{{\bs{\vartheta}}_2,\zeta}^2$ converges to $\bs{\phi}^*$ in probability. 

Define $f_v^*(\bs{x}|\bs{z};\bs{\lambda}):=f_v\left(\bs{x}|\bs{z};\bs{\gamma}^*, \bs\mu^*+ \bs{\lambda}, \bs{v}^*\right)$ so that $f^*_v(\bs{x}|\bs{z};\bs{0}) = f^*_v$ and $\nabla f^*_v(\bs{x}|\bs{z};\bs{0}) = \nabla f^*_v$. Define the score vectors $\bs{s}(\bs{x},\bs{z};\bs{\lambda})$ indexed by $\bs{\lambda}$ as 
\begin{equation} \label{score_defn-homo-2}
\bs{s}(\bs{x},\bs{z};\bs{\lambda}) := 
\begin{pmatrix}
\bs{s}_{\bs{\eta}}\\ 
{s}_{\alpha}(\bs{\lambda})
\end{pmatrix},
\end{equation}
where $\bs{s}_{\bs{\eta}} =\nabla_{(\bs{\gamma}\t,\bs{\mu}\t,\bs{v}\t)\t}f_v^*/f_v^*$ as defined in (\ref{score_defn-homo}) and 
\begin{equation} \label{score_defn-homo-3}
{s}_{\alpha}(\bs{\lambda}) := 
\frac{f_v^*(\bs{x}|\bs{z};\bs{\lambda}) -f_v^* - \nabla_{\bs{\mu}\t} f_v^* \bs{\lambda} - \nabla_{\bs{v}\t} f_v^* \bs{\lambda}_{\bs{\mu}^2} }{ |\bs{\lambda}|^3 f_v^*},
\end{equation} 
where $\underset{(d_{\mu^2} \times 1)}{\bs{\lambda}_{\bs{\mu}^2}} := \{(\bs{\lambda}_{\bs{\mu}^2})_{ij}\}_{1 \leq i \leq j \leq d}$ with $(\bs{\lambda}_{\bs{\mu }^2})_{ij}:=\lambda_{ii}^2$ if $i=j$ and $2\lambda_{ij}$ if $i \neq j$. The division by $|\bs{\lambda}|^3$ is necessary to define $s_\alpha(\bs{\lambda})$ here because, if we were to define $s_\alpha(\bs{\lambda})$ as $(f_v^*(\bs{x}|\bs{z};\bs{\lambda}) -f_v^* - \nabla_{\bs{\mu}\t} f_v^* \bs{\lambda} - \nabla_{\bs{v}\t} f_v^* \bs{\lambda}_{\bs{\mu}^2})/ f_v^*$, then we have $s_\alpha(\bs{\lambda})\to 0$ as $\bs{\lambda}\to 0$, invalidating the approximation when $\bs{\lambda}$ is close to zero.

Collect the relevant reparameterized parameters as 
\begin{equation} \label{tpsi_defn-homo-2} 
\bs{t} (\bs{\phi},\bs\lambda) 
:=
\begin{pmatrix}
\bs{t}_{\bs{\eta}}(\bs\lambda)\\
t_\alpha(\bs{\lambda})
\end{pmatrix}\quad\text{with}\quad
\bs{t}_{\bs{\eta}}(\bs\lambda):=
\begin{pmatrix}
\bs{\gamma}-\bs{\gamma}^*\\
\bs{\nu_\mu}-\bs{\mu}^* \\
\bs{\nu_v}-\bs{v}^* 
\end{pmatrix}\quad \text{and} \quad
t_\alpha(\bs{\lambda}):=\alpha |\bs{\lambda}|^3.
\end{equation} 
 
In (\ref{score_defn-homo-3}), $s_\alpha(\bs{\lambda})$ is non-degenerate and not perfectly correlated with $\bs{s}_{\bs{\eta}}$ even when $\bs{\lambda} \to \bs{0}$. With $(\bs{s}_{\bs{\eta}}, {s}_{\alpha}(\bs{\lambda}))$ defined in (\ref{score_defn-homo-2}), define 
\begin{equation} \label{I_lambda-homo-2}
\begin{aligned}
&\bs{\mathcal{I}}_{\bs{\eta}} := E[\bs{s}_{\bs{\eta}}(\bs{s}_{\bs{\eta}})\t], \quad \bs{\mathcal{I}}_{\alpha\bs{\eta} }(\bs{\lambda}) := E[{s}_{\alpha}(\bs{\lambda}) \bs{s}_{\bs{\eta}} \t], \quad \bs{\mathcal{I}}_{\bs{\eta }\alpha}(\bs{\lambda}) := (\bs{\mathcal{I}}_{\alpha \bs{\eta}}(\bs{\lambda}))\t,\\
& {\mathcal{I}}_{\alpha}(\bs{\lambda}_1,\bs{\lambda}_2) := E[{s}_{\alpha}(\bs{\lambda}_1) {s}_{\alpha}(\bs{\lambda}_2) ],\quad {\mathcal{I}}_{\alpha.\bs{\eta}}(\bs{\lambda}_1,\bs{\lambda}_2):= {\mathcal{I}}_{\alpha}(\bs{\lambda}_1,\bs{\lambda}_2)-\bs{\mathcal{I}}_{\alpha\bs{\eta}}(\bs{\lambda}_1)(\bs{\mathcal{I}}_{\bs{\eta}})^{-1}\bs{\mathcal{I}}_{\bs{\eta}\alpha}(\bs{\lambda}_2), \\
& {Z}_{\alpha}(\bs{\lambda}):=( {\mathcal{I}}_{\alpha.\bs{\eta}}(\bs{\lambda}, \bs{\lambda}))^{-1} {G}_{\alpha.\bs{\eta}}(\bs{\lambda}),
\end{aligned}
\end{equation} 
where ${G}_{\alpha.\bs{\eta}}(\bs{\lambda})$ is a mean zero Gaussian process indexed by $\bs{\lambda}$ with $\text{Cov}({G}_{\alpha.\bs{\eta}}(\bs{\lambda}_1),{G}_{\alpha.\bs{\eta}}(\bs{\lambda}_2)) = {\mathcal{I}}_{\alpha.\bs{\eta}}(\bs{\lambda}_1,\bs{\lambda}_2)$. 
Define $\widehat{t}_\alpha(\bs{\lambda})$ by
\begin{align} \label{t-lambda-homo-2}
r(\widehat{t}_\alpha(\bs{\lambda}))= \inf_{ t_\alpha \geq 0 }r( t_\alpha),\quad 
 r( t_\alpha) := (t_\alpha - {Z}_{{\alpha}}(\bs{\lambda}))^2 {\mathcal{I}} _{{\alpha}.\bs{\eta}}(\bs{\lambda}, \bs{\lambda}).
 \end{align} 
The following proposition establishes the asymptotic null distribution of $LR_{n,\zeta}^{2}$.  Define $\bs{\mathcal{I}}(\bs{\lambda}) := E[\bs{s}(\bs{X},\bs{Z};\bs{\lambda}) \bs{s}(\bs{X},\bs{Z};\bs{\lambda})\t ]$.
\begin{proposition} \label{P-LR-N1-homo-2}
Suppose that Assumptions \ref{assn_consis}, \ref{A-taylor1}, and \ref{assn_compact} hold and $0 < \inf_{\Theta_{\bs{\lambda}} \setminus \{\bs{0}\}} \lambda_{\min}(\bs{\mathcal{I}}(\bs{\lambda})) \leq \sup_{\Theta_{\bs{\lambda}} \setminus \{\bs{0}\}} \lambda_{\max}(\bs{\mathcal{I}}(\bs{\lambda})) < \infty$. Then, under the null hypothesis of $H_{0}: M=1$, for any $\zeta>0$, $LR_{n,\zeta}^{2} \rightarrow_d \sup_{\Theta_{\bs{\lambda}} \cap \{|\bs{\lambda}|\geq \zeta\} }\ (\widehat{t}_\alpha(\bs{\lambda}) )^2 {\mathcal{I}}_{\alpha.\bs{\eta}}(\bs{\lambda}, \bs{\lambda})$.
\end{proposition}
 
\subsubsection{Testing $H_{0}: M=1$}

The following proposition derives the asymptotic distribution of $LR_{n}$.
The proof is omitted because it is a straightforward consequence of Propositions \ref{P-LR-N1-homo-1} and \ref{P-LR-N1-homo-2} in view of $LR_{n}=\lim_{\zeta\to 0}\max \{LR_{n,\zeta}^1 ,LR_{n,\zeta}^2 \}$.

\begin{proposition}\label{P-LR-N1-homo}
Suppose that Assumptions of Propositions \ref{P-LR-N1-homo-1} and \ref{P-LR-N1-homo-2} hold. Then, under the null hypothesis of $H_{0}: M=1$, 
\[ 
LR_{n} \rightarrow_d \max\left\{ (\widehat{\bs{t}}^1_{\bs{\lambda}})\t \bs{\mathcal{I}}_{\bs{\lambda}.\bs{\eta}} \widehat{\bs{t}}^1_{\bs{\lambda}}, (\widehat{\bs{t}}^2_{\bs{\lambda}})\t \bs{\mathcal{I}}_{\bs{\lambda}.\bs{\eta}} \widehat{\bs{t}}^2_{\bs{\lambda}}, \ \sup_{\Theta_{\bs{\lambda}} \setminus \{\bs{0}\} }\ (\widehat{t}_\alpha(\bs{\lambda}) )^2 {\mathcal{I}}_{\alpha.\bs{\eta}}(\bs{\lambda}, \bs{\lambda}) \right\},
\] 
where $\widehat{\bs{t}}_{\bs{\lambda}}^j $ for $j=1,2$ is defined in (\ref{t-lambda-homo}) and $\widehat{t}_\alpha(\bs{\lambda})$ is defined in (\ref{t-lambda-homo-2}).
\end{proposition}
This result generalizes Theorem 2 of \citet{chenchen03sinica}, who derive the asymptotic distribution of the LRTS in the univariate case.

\subsection{Likelihood ratio test of $H_0: M = M_0$ against $H_A: M = M_0 + 1$ for $M_0\geq 2$} 

We consider a random sample $\{\bs{X}_i,\bs{Z}_i\}_{i = 1}^n$ generated from the following $M_0$-component $d$-variate normal mixture density model with common variance:
\begin{equation}
f_{M_0}(\bs{x}|\bs{z};{\bs{\vartheta}}_{M_0}^*):=\sum_{j=1}^{M_0} \alpha_{j}^{*} f(\bs{x}|\bs{z};{\bs{\gamma}}^*, \bs{\mu}_{j}^{*},\bs{\Sigma}^{*}), \label{true_model-homo}
\end{equation}
where ${\bs{\vartheta}}_{M_0}^*=(\alpha_{1}^*,\ldots,\alpha_{M_0 - 1}^*,{\bs{\gamma}}^*,\bs{\mu}_1^*,\ldots,\bs{\mu}_{M}^*,\bs{\Sigma}^{*} )$ and $\alpha_j^*>0$. We assume $\bs{\mu}_{1}^* <\ldots< \bs{\mu}_{M_0}^*$ for identification. The corresponding density of an $(M_0+1)$-component mixture model is given by
\begin{equation}
f_{M_0+1}(\bs{x}|\bs{z};{\bs{\vartheta}}_{M_0+1}):=\sum_{j=1}^{M_0+1}\alpha_j f(\bs{x}|\bs{z};\bs{\gamma},\bs{\mu}_j,\bs{\Sigma} ),\label{fitted_model-homo}
\end{equation}
where ${\bs{\vartheta}}_{M_0+1} = (\alpha_1,\ldots,\alpha_{M_0},\bs{\gamma},\bs{\mu}_1,\ldots.,\bs{\mu}_{M_0+1},\bs{\Sigma})$. Partition the null hypothesis as $H_0 = \cup_{m=1}^{M_0} H_{0,m}$ with $H_{0,m}: \alpha_m(\bs{\mu}_{m} - \bs{\mu}_{m + 1}) = 0$.   

Define the LRTS for testing $H_{01}$ as
\[
 LR_{n}^{M_0} := \max_{{\bs{\vartheta}}_{M_0+1}\in \Theta_{{\bs{\vartheta}}_{M_0+1}}}2\{ L_n({\bs{\vartheta}}_{M_0+1})- L_{0,n}(\widehat{{\bs{\vartheta}}}_{M_0})\},
\]
where $L_n({\bs{\vartheta}}_{M_0 + 1}):=\sum_{i = 1}^n \log f_{M_0 + 1}(\bs{X}_i|\bs{Z}_i;{\bs{\vartheta}}_{M_0 + 1})$, $L_{0,n}({\bs{\vartheta}}_{M_0})=\sum_{i = 1}^n \log f_{M_0}(\bs{X}_i|\bs{Z}_i;{\bs{\vartheta}}_{M_0})$, and $\widehat{\bs{\vartheta}}_{M_0}=\mathop{\arg \max}_{{\bs{\vartheta}}_{M_0}\in\Theta_{{\bs{\vartheta}}_{M_0}}}L_{0,n}({\bs{\vartheta}}_{M_0})$ for the densities (\ref{true_model-homo})--(\ref{fitted_model-homo}). Define $\widetilde{\bs{\lambda}}:=((\bs{\lambda}_1)\t,\ldots,(\bs{\lambda}_{M_0})\t)\t \in \Theta_{\widetilde{\bs{\lambda}}}$ with $\bs{\lambda}_m \in \Theta_{\bs{\lambda}_m}$.  Collect the score vector for testing $H_{0,1},\ldots,H_{0,M_0}$   as
\begin{equation}
\begin{aligned}
&\widetilde{\bs{s}}(\bs{x},\bs{z}) :=
\begin{pmatrix}
\widetilde{\bs{s}}_{\bs{\eta}} \\ \widetilde{\bs{s}}_{\bs{\lambda}}
\end{pmatrix}\quad \text{and}\quad 
\bar{\bs{s}}(\bs{x},\bs{z};\widetilde{\bs{\lambda}}):=
\begin{pmatrix}
\widetilde{\bs{s}}_{\bs{\eta}} \\ \bar{\bs{s}}_{\bs\alpha}(\tilde{\bs{\lambda}})
\end{pmatrix},
\ \text{ where }\\
&\widetilde{\bs{s}}_{\bs{\eta}}:= \left(\begin{array}{c} {\bs{s}}_{\bs{\alpha}} \\ {\bs{s}}_{{ (\bs{\gamma},\bs{\mu}, \bs{v}) }} 
 \end{array}\right),\quad \widetilde{\bs{s}}_{\bs{\lambda}}: =
\begin{pmatrix}
\bs{s}_{\bs{\mu}^3}^1 \\
\bs{s}_{\bs{\mu}^4}^1 \\
 \vdots \\
\bs{s}_{\bs{\mu}^3}^{M_0} \\
\bs{s}_{\bs{\mu}^4}^{M_0} 
\end{pmatrix}, \quad \text{and}\quad
\bar{\bs{s}}_{\bs\alpha}(\widetilde{\bs{\lambda}}): =
\begin{pmatrix}
{s}_{\bs{\alpha}}^1(\bs{\lambda}_1)\\ 
 \vdots \\ 
{s}_{\bs{\alpha}}^{M_0}(\bs{\lambda}_{M_0})
\end{pmatrix},
\end{aligned}
\label{stilde-homo}
\end{equation}
where $\bs{s}_{\bs{\mu^3}}^m := \left\{\alpha_{m}^* \nabla_{\mu_i \mu_j \mu_k} f^*_v(\bs{x}|\bs{z};\bs{\gamma}^*,\bs{\mu}_m^{*},\bs{v}^*) / (3!f_0^*) \right\}_{1 \leq i \leq j \leq k \leq d}$; $\bs{s}_{\bs{\alpha}}$, $\bs{s}_{(\bs{\gamma},\bs{\mu}, \bs{v})}$, and $\bs{s}_{\bs{\mu}^4}^m$ are defined similarly to those in (\ref{sh}) but using the density (\ref{true_model-homo}) in place of (\ref{true_model}) with the common value of $v^*$ across components; ${s}_{\alpha}^m(\bs{\lambda}_m)$ is defined as
\begin{equation*} 
{s}_{\alpha}^m(\bs{\lambda}_m) := \alpha_m^*
\frac{f_v^{m*}(\bs{\lambda}_m) -f_v^{m*}-\nabla_{\bs{\mu}} f^{m*}_v \bs{\lambda}_m - \nabla_{\bs{v}} f^{m*}_v \bs{\lambda}_{\bs{\mu}^2,m} }{ |\bs{\lambda}_m|^3f^{m*}_v},
\end{equation*} 
where $ f_v^{m*}(\bs{\lambda}_m):=f_v\left(\bs{x}|\bs{z};\bs{\gamma}^*, \bs\mu_m^*+ \bs{\lambda}_m, \bs{v}^*\right)$ and $f_v^{m*} :=f_v\left(\bs{x}|\bs{z};\bs{\gamma}^*, \bs\mu_m^*, \bs{v}^*\right)$, and $\bs{\lambda}_{\bs{\mu}^2,m} $ is defined similarly to $\bs{\lambda}_{\bs{\mu}^2}$ but with $\bs{\lambda}_m$ in place of $\bs{\lambda}$.

Define $\widetilde{\bs{\mathcal{I}}}$, $\widetilde{\bs{\mathcal{I}}}_{\bs{\eta}}$, $\widetilde{\bs{\mathcal{I}}}_{\bs{\lambda}\bs{\eta}}$, $\widetilde{\bs{\mathcal{I}}}_{\bs{\eta\lambda}}$, $\widetilde{\bs{\mathcal{I}}}_{\bs{\lambda}}$, $\widetilde{\bs{\mathcal{I}}}_{\bs{\lambda}.\bs{\eta}}$ similarly to those in (\ref{Itilde}) but using $\widetilde{\bs{s}}(\bs{x},\bs{z})$ defined in (\ref{stilde-homo}) in place of (\ref{stilde}). Let $\widetilde{\bs{G}}_{\bs{\lambda}.\bs{\eta}}=((\bs{G}_{\bs{\lambda}.\bs{\eta}}^{1})^{\top},\ldots,(\bs{G}_{\bs{\lambda}.\bs{\eta}}^{M_0})^\top)^\top \sim N(0,\widetilde{\bs{\mathcal{I}}}_{\bs{\lambda}.\bs{\eta}})$, and define ${\bs{\mathcal{I}}}_{\bs{\lambda}.\bs{\eta}}^m:=E[\bs{G}_{\bs{\lambda}.\bs{\eta}}^m (\bs{G}_{\bs{\lambda}.\bs{\eta}}^m)^\top]$ and $\bs{Z}_{\bs{\lambda}}^m:=({\bs{\mathcal{I}}}_{\bs{\lambda}.\bs{\eta}}^m)^{-1}\bs{G}_{\bs{\lambda}.\bs{\eta}}^m$. For $j=1,2$, define $\widehat{\bs{t}}_{\bs{\lambda},m}^j $ by
\begin{equation*}
r^m(\widehat{\bs{t}}_{\bs{\lambda},m}^j ) = \inf_{\bs{t}_{\bs{\lambda}} \in \Lambda_{\bs{\lambda} }^j}r^m(\bs{t}_{\bs{\lambda}}), \quad r^m(\bs{t}_{\bs{\lambda}}) := (\bs{t}_{\bs{\lambda}} -\bs{Z}_{\bs{\lambda}}^m)\t \bs{\mathcal{I}}_{\bs{\lambda}.\bs{\eta}}^m (\bs{t}_{\bs{\lambda}} -\bs{Z}_{\bs{\lambda}}^m),
\end{equation*} 
where $\Lambda_{\bs{\lambda}}^j$ is given by (\ref{Lambda-e-homo}).

Define $\widehat{t}_{\alpha,m}(\bs{\lambda}_m)$ by
\begin{align*} 
r^m(\widehat{t}_{\alpha,m} (\bs{\lambda}_m))= \inf_{ t_\alpha \geq 0 }r^m( t_\alpha),\quad 
 r^m( t_\alpha) := (t_\alpha - {Z}_{\bs{\alpha}}^m(\bs{\lambda}_m))^2 {\mathcal{I}} _{\bs{\alpha}.\bs{\eta}}^m(\bs{\lambda}_m, \bs{\lambda}_m),
\end{align*}
where ${\mathcal{I}}_{\bs{\alpha}.\bs{\eta}}^m(\bs{\lambda}_m, \bs{\lambda}_m)$ and ${Z}_{\bs{\alpha}}^m(\bs{\lambda}_m)$ are defined similarly to $ {\mathcal{I}}_{\alpha.\bs{\eta}}(\bs{\lambda}, \bs{\lambda})$ and $Z_{{\alpha}}(\bs{\lambda})$ in (\ref{I_lambda-homo-2}), respectively, but using $\widetilde{\bs{s}}_{\bs{\eta}}$ and $s_{\bs{\alpha}}^m(\bs{\lambda}_m)$ in place of ${\bs{s}}_{\bs{\eta}}$ and $s_{\alpha}(\bs{\lambda})$.

\begin{assumption}\label{A-vec-2-homo}  (a) The parameter spaces $\Theta_{\bs{\vartheta}_{M_0}}$ and $\Theta_{\bs{\vartheta}_{M_0+1}}$ are compact. (b)  $\widetilde{\bs{\mathcal{I}}}=E[\widetilde{\bs{s}}(\bs{X},\bs{Z})\widetilde{\bs{s}}(\bs{X},\bs{Z})\t]$ is finite and nonsingular and $0 < \inf_{\Theta_{\widetilde{\bs{\lambda}}}\setminus \{\bs{0}\}  } \lambda_{\min}(\bar{\bs{\mathcal{I}}}(\widetilde{\bs{\lambda}})) \leq \sup_{\Theta_{\widetilde{\bs{\lambda}}}\setminus \{\bs{0}\}} \lambda_{\max}(\bar{\bs{\mathcal{I}}}(\widetilde{\bs{\lambda}})) < \infty$, where $\bar{\bs{\mathcal{I}}}(\widetilde{\bs{\lambda}}):= E[\bar{\bs{s}}(\bs{X},\bs{Z};\widetilde{\bs{\lambda}}) (\bar{\bs{s}}(\bs{X},\bs{Z};\widetilde{\bs{\lambda}}))\t]$ and $\widetilde{\bs{s}}(\bs{X},\bs{Z})$ and $\bar{\bs{s}}(\bs{X},\bs{Z};\widetilde{\bs{\lambda}})$ are defined in (\ref{stilde-homo}). 
\end{assumption}

\begin{proposition} \label{local_lr-2-homo}
Suppose that Assumptions \ref{assn_consis}, \ref{A-taylor1}, and \ref{A-vec-2-homo} hold. Then, under the null hypothesis $H_0: m=M_0$, $LR_{n}^{M_0} \rightarrow_d \max\{v_1,\ldots, v_{M_0}\}$, where 
\[
v_m := \max\left\{ (\widehat{\bs{t}}_{\bs{\lambda},m}^1 )\t \bs{\mathcal{I}}_{\bs{\lambda}.\bs{\eta}}^m \widehat{\bs{t}}_{\bs{\lambda},m}^1 , \ (\widehat{\bs{t}}_{\bs{\lambda},m}^2 )\t \bs{\mathcal{I}}_{\bs{\lambda}.\bs{\eta}}^m \widehat{\bs{t}}_{\bs{\lambda},m}^2 , \ \sup_{\Theta_{\bs{\lambda}_m} \setminus \{\bs{0}\}}\ (\widehat{t}_{\alpha,m}(\bs{\lambda}_m) )^2 {\mathcal{I}}_{\bs{\alpha}.\bs{\eta}}^m(\bs{\lambda}_m, \bs{\lambda}_m) \right\}.
\]  
\end{proposition}

\section{Simulation} \label{section:simulation }
  
\subsection{Choice of penalty function} \label{section:penalty}

To apply our EM test, we need to specify the set $\mathcal{T}$, number of iterations $K$, and penalty functions for  $p(\tau)$. Based on our experience, we recommend $\mathcal{T} = \{0.1, 0.3, 0.5\}$ and $K = \{1,2,3\}$.    
We set $p(\tau)= \log(2\min\{\tau,1-\tau\})$ as suggested by \citet{chenli09as}. When estimating the model under the null hypothesis and computing $L_{0,n}(\widehat{\bs{\vartheta}}_{M_0})$, we use the penalty function (\ref{pen_pmle}) and set $a_n = n^{-1/2}$ as recommended by \citet{chentan09jmva}. For the alternative model,  we consider
$a_n = n^{-1/2}$  and $1$ to examine the sensitivity of the rejection frequencies to the choice of $a_n$.

\subsection{Simulation results} \label{section:results}

We examine the type I error rates and powers of the EM test by small simulations using mixtures of bivariate normal distributions. Computation was done using R \citep{R}. The critical values are computed by bootstrap with $399$ and $199$ bootstrap replications when testing $H_0:M=1$ and $H_0:M=2$, respectively. We use $2000$ replications, and the sample sizes are set to $200$ and $400$. 

Table \ref{table1} reports the type I error rates of the EM test of $H_0:M = 1$ against the alternative $H_1:M = 2$  under the null hypothesis using two models given at the bottom of Table \ref{table1}. In both models, the EM test statistics give accurate type I errors for $n=200$ and $400$ across two values of $a_n=n^{-1/2}$ and $1$. Table \ref{table3} reports the powers of the EM test  when  $a_n=1$ under three alternative models given in Table \ref{table2}. Comparing the rejection frequency of Model 1 with that of Model 2 or Model 3, the EM test shows higher power as the distance between two component distributions in the alternative model increases in terms of means (Model 2) or variance (Model 3).

Table \ref{table4} reports the type I error rates of the EM test of $H_0:M = 2$ against the alternative $H_1:M = 3$  under the  two null models given at the bottom of   Table \ref{table4}. The EM test gives accurate type I errors across two models, sample sizes, and the values of $a_n$. Table \ref{table6} reports the powers of the EM test of  $H_0:M = 2$ against the alternative $H_1:M = 3$ under the alternative model given in Table \ref{table5}. Overall, the EM test shows good power under finite sample size.

\section{Empirical applications}
The sequential hypothesis testing based on our  EM test provides a useful alternative to the AIC or the BIC in determining the number of components in empirical applications.\footnote{For penalty function in our empirical applications, we set $a_n=n^{-1/2}$ for the null model and $a_n=1$ for the alternative model. } 

\subsection{The flea beetles}

The flea beetles data available in R package \textbf{tourr} contains a sample of 74 flea beetles from three species, ``Concinna," ``Heikertingeri," and ``Heptapotamica'' with  21, 31, and 22 observations, respectively.\footnote{The data is originally from \citet{lubischew62bio}.}  Figure \ref{figure1} provides a scatter plot of two physical measurements ``tars1'' and ``aede1,'' which measure the
width of the first joint of the first tarsus in microns (the sum of measurements for both tarsi) and the maximal width of the aedeagus in the fore-part in microns, respectively, for each of three species. We sequentially test the number of components in this data set without utilizing the information on which species each observation is from.
As shown in Table \ref{table7}, the $p$-values of EM test for testing $H_0: M=1$ and $H_0: M=2$ are $0.00$ and $0.01$, suggesting that the number of components is larger than two. On the other hand, the $p$-values of the EM test for testing $H_0: M=3$ are between 0.32 and 0.36; consistent with the actual number of species in this data set, we fail to reject $H_0: M=3$. In contrast, both the AIC and the BIC incorrectly indicate that there is only one component. Table \ref{table8} compares the estimated three-component bivariate normal mixture model in the first panel  with  the single component models estimated from a subsample of each of three species in the second panel, showing that each of estimated three component distributions accurately captures the corresponding species.

\subsection{Analysis of differential gene expression}

A multivariate normal mixture model can be used to find differentially expressed genes by means of the posterior probability that an individual gene is non-differentially expressed. We analyze the rat dataset of 1,176 genes
in middle-ear mucosa of six rat samples, the first two without  pneumococcal middle-ear infection and the latter four with the disease \citep{pan02bio,he06csda}. As in \citet{pan02bio}, the data were normalized by  log-transformation and median centering. Denote the resulting expression levels of gene $i$  of sample $j$ by $x_{ij}$.  
We apply finite bivariate normal mixtures to model the sample average expression levels for gene $i$ under the two conditions,  $(z_{i0},z_{i1})=(\sum_{j=1}^2 x_{ij}/2,\sum_{j=3}^6 x_{ij}/4)$ for $i=1,\ldots,1176$.   
 
As shown in Table \ref{table9},  the sequential hypothesis testing based on EM test and the AIC indicate that there are six components; on the other hand, consistent with the result in \cite{he06csda}, the BIC chooses the five-component model.  Table \ref{table10} presents the estimates from the six component model. We classify each pair of gene expression levels into six clusters using their posterior probabilities and plot them in Figure \ref{figure2}. 

32 genes classified into cluster 5 show some evidence for differential expression with a mean difference of 0.23. Similarly,  13 genes in cluster 6 demonstrate strong  evidence for differential expression, albeit with large variability. In contrast, the genes in clusters  1--4 show a flat expression pattern, where the observations in each cluster center around the 45 degree line.

\newpage

\appendix
\def\thesection{Appendix \Alph{section}}

\section{Proof of propositions}\label{app}

\begin{proof}[Proof of Proposition \ref{P-consis}]

As shown by \citet[][p.\ 248]{alexandrovich14jmva}, $p_n(\bs{\vartheta}_M)$ satisfies Assumptions C1--C3 of \citet{chentan09jmva} under the stated condition on $a_n$. Therefore, the stated result follows from Theorem 1 of \citet{chentan09jmva} and Corollary 3 of \citet{alexandrovich14jmva}.
\end{proof}

\begin{proof}[Proof of Proposition \ref{P-LR-N1}]

The proof is similar to that of Proposition 3 of \citet{kasaharashimotsu15jasa}.  Let $\bs{t}_{\bs{\eta}}: = \bs{\eta} - \bs{\eta}^*$, so that $\bs{t}(\bs{\psi},\alpha)$ in (\ref{tpsi_defn}) is written as $(\bs{t}_{\bs{\eta}}\t,\bs{t}_{\bs{\lambda}}(\bs{\lambda},\alpha)\t)\t$. Let
\[
\bs{G}_{ n} := \nu_n (\bs{s}(\bs{x},\bs{z})) =
\begin{bmatrix}
\bs{G}_{\bs{\eta} n} \\
\bs{G}_{\bs{\lambda} n}
\end{bmatrix}, \quad
\begin{aligned}
\bs{G}_{\bs{\lambda}.\bs{\eta} n} &:= \bs{G}_{\bs{\lambda} n} - \bs{\mathcal{I}}_{\bs{\lambda} \bs{\eta}}\bs{\mathcal{I}}_{\bs{\eta} }^{-1} \bs{G}_{\bs{\eta} n}, \quad \bs{Z}_{\bs{\lambda} . \bs{\eta} n} := \bs{\mathcal{I}}_{\bs{\lambda}.\bs{\eta} }^{-1}\bs{G}_{\bs{\lambda}.\bs{\eta} n},\\
\bs{t}_{\bs{\eta}.\bs{\lambda} } &:= \bs{t}_{\bs{\eta}} + \bs{\mathcal{I}}_{\bs{\eta} }^{-1}\bs{\mathcal{I}}_{\bs{\eta}\bs{\lambda} } \bs{t}_{\bs{\lambda}}(\bs{\lambda},\alpha) .
\end{aligned}
\]
 Write 
\begin{align*}
LR_n(\epsilon_1) &= \max_{\alpha \in [\epsilon_1,1-\epsilon_1]} 2\{L_n(\widehat{\bs{\psi}},\alpha) - L_n(\bs{\psi}^*,\alpha) - [ L_{0,n}(\widehat{\bs{\gamma}}_0,\widehat{\bs{\mu}}_0,\widehat{\bs{\Sigma}}_0)- L_n(\bs{\psi}^*,\alpha) ]\} \\
 &= \max_{\alpha \in [\epsilon_1,1-\epsilon_1]} 2\{L_n(\widehat{\bs{\psi}},\alpha) - L_n(\bs{\psi}^*,\alpha)\} - 2\{ L_{0,n}(\widehat{\bs{\gamma}}_0,\widehat{\bs{\mu}}_0,\widehat{\bs{\Sigma}}_0)- L_{0,n}(\bs{\gamma}^*,\bs{\mu}^*,\bs{\Sigma}^*) \}. 
\end{align*}
We apply Lemma \ref{P-quadratic} in \ref{section:quadratic} and Lemma \ref{Ln_thm2} in \ref{section:expansion} to these two terms.

Note that the penalized MLE, $\widehat{\bs{\psi}}$, is consistent and that $p_n(\widehat {\bs{\vartheta}}_2) =o_p(1)$ from $a_n=O(1)$, $p_{n}(\bs{\Sigma};\bs{\Sigma})=0$ and $\widehat{\bs{\Sigma}}_1, \widehat{\bs{\Sigma}}_2, \widehat{\bs{\Omega}} \to_p \bs{\Sigma}$. Therefore, $\widehat{\bs{\psi}}$ is in the set $A_{n\varepsilon}(\delta)$ in Lemma \ref{P-quadratic}, and  Lemma \ref{P-quadratic} holds under the current set of assumptions.  Split the quadratic form in Lemma \ref{P-quadratic}(b) and write it as
\begin{equation} \label{LR_appn}
\sup_{\bs{\vartheta} \in A_{n\varepsilon}(\delta) } \left| 2 \left[L_n(\bs{\psi},\alpha) - L_n(\bs{\psi}^*,\alpha) \right] - B_n(\sqrt{n} \bs{t}_{\bs{\eta}.\bs{\lambda} }) - C_n(\sqrt{n} \bs{t}_{\bs{\lambda}}(\bs{\lambda},\alpha)) \right| =o_{p \varepsilon}(1),
\end{equation}
where
\begin{equation} \label{B_pi}
\begin{aligned}
B_n(\bs{t}_{\bs{\eta}.\bs{\lambda} }) & = 2\bs{t}_{\bs{\eta}.\bs{\lambda} }\t\bs{G}_{\bs{\eta} n} - \bs{t}_{\bs{\eta}.\bs{\lambda} }\t\bs{\mathcal{I}}_{\bs{\eta}}\bs{t}_{\bs{\eta}.\bs{\lambda} }, \\
C_n(\bs{t}_{\bs{\lambda}}(\bs{\lambda},\alpha)) &= 2\bs{t}_{\bs{\lambda}}(\bs{\lambda},\alpha)\t \bs{G}_{\bs{\lambda}.\bs{\eta} n} - \bs{t}_{\bs{\lambda}}(\bs{\lambda},\alpha)\t \bs{\mathcal{I}}_{\bs{\lambda}.\bs{\eta} } \bs{t}_{\bs{\lambda}}(\bs{\lambda},\alpha) \\
& = \bs{Z}_{\bs{\lambda} . \bs{\eta} n}\t \bs{\mathcal{I}}_{\bs{\lambda}.\bs{\eta} } \bs{Z}_{\bs{\lambda} . \bs{\eta} n}- (\bs{t}_{\bs{\lambda}}(\bs{\lambda},\alpha) - \bs{Z}_{\bs{\lambda} . \bs{\eta} n})\t\bs{\mathcal{I}}_{\bs{\lambda}.\bs{\eta} }(\bs{t}_{\bs{\lambda}}(\bs{\lambda},\alpha) - \bs{Z}_{\bs{\lambda} . \bs{\eta} n}).
\end{aligned}
\end{equation}
Observe that $2[L_{0,n}(\widehat{\bs{\gamma}}_0,\widehat{\bs{\mu}}_0,\widehat{\bs{\Sigma}}_0) - L_{0,n}(\bs{\gamma}^*,\bs{\mu}^*,\bs{\Sigma}^*)]   = \max_{\bs{t}_{\bs{\eta}.\bs{\lambda} }} B_n(\sqrt{n} \bs{t}_{\bs{\eta}.\bs{\lambda} }) + o_p(1)$ from applying Lemma \ref{Ln_thm2} to $L_{0,n}(\bs{\gamma},\bs{\mu},\bs{\Sigma})$ and noting that the set of possible values of both $\sqrt{n} \bs{t}_{\bs{\eta}}$ and $\sqrt{n}\bs{t}_{\bs{\eta}.\bs{\lambda} }$ approaches $\mathbb{R}^{d_\eta}$.  Therefore, in conjunction with  $p_{n}(\widehat{\bs{\vartheta}}_2)= o_p(1)$ and  (\ref{LR_appn}), we obtain
\begin{equation} \label{LR_appn2}
2[L_n(\widehat{\bs{\psi}},{\alpha}) -L_{0,n}(\widehat{\bs{\gamma}}_0,\widehat{\bs{\mu}}_0,\widehat{\bs{\Sigma}}_0)] = C_n(\sqrt{n}\bs{t}_{\bs{\lambda}}(\widehat{\bs{\lambda}},\alpha)) + o_p(1). 
\end{equation}
Split $\bs{t}_{\bs{\lambda}}(\bs{\lambda},\alpha)$ as $\bs{t}_{\bs{\lambda}}(\bs{\lambda},\alpha)=(\bs{t}_{\bs{\mu v}}(\bs{\lambda},\alpha)\t,\bs{t}_{\bs{\mu}^4}(\bs{\lambda},\alpha)\t)\t = (12c(\alpha) \bs{\lambda}_{\bs{\mu v}}\t, 
c(\alpha)[12\bs{\lambda}_{\bs{v}^2}+ b(\alpha) \bs{\lambda}_{\bs{\mu}^4}]\t)\t$ with $c(\alpha):=\alpha(1-\alpha)$. Partition the parameter space as $\Theta_{\bs{\lambda}} = \Theta_{\bs{\lambda}}^{1} \cup \Theta_{\bs{\lambda}}^2$ with
\begin{align*}
\Theta_{\bs{\lambda}}^{1} &:= \{ |\lambda_{\mu_i}| \leq n^{-1/8} (\log n)^{-1} \text{ for all $i\in\{1,\ldots,d\}$} \},\\ \Theta_{\bs{\lambda}}^{2} &:= \{ |\lambda_{\mu_i}| \geq n^{-1/8} (\log n)^{-1} \text{ for some $i\in\{1,\ldots,d\}$} \}.
\end{align*}
For $j\in \{1,2\}$, define $\ddot{\bs{\lambda}}^{j}$ by $C_n(\sqrt{n} \bs{t}(\ddot{\bs{\lambda}}^{j}, \alpha)) = \max_{\bs{\lambda} \in \Theta_{\bs{\lambda}}^{j}} C_n(\sqrt{n} \bs{t}_{\bs{\lambda}}(\bs{\lambda},\alpha))$. Then, we have
\begin{align} 
& \bs{t}(\ddot{\bs{\lambda}}^{j}, \alpha) =(\bs{t}_{\bs{\mu v}}(\ddot{\bs{\lambda}}^{j},\alpha)\t,\bs{t}_{\bs{\mu}^4}(\ddot{\bs{\lambda}}^{j},\alpha)\t)\t = O_p(n^{-1/2}), \label{ddot_rate} \\
& 2[L_n(\widehat{\bs{\psi}},{\alpha}) -L_{0,n}(\widehat{\bs{\gamma}}_0,\widehat{\bs{\mu}}_0,\widehat{\bs{\Sigma}}_0)] = \max_{j\in \{1,2\}} C_n(\sqrt{n} \bs{t}(\ddot{\bs{\lambda}}^{j}, \alpha)) + o_p(1), \label{LR_appn3}
\end{align}
where (\ref{ddot_rate}) follows from noting that $C_n(\sqrt{n}\bs{t}(\ddot{\bs{\lambda}}^{j}, \alpha)) \geq o_p(1)$ and using the argument following (\ref{rk_lower2}) in the proof of Lemma \ref{Ln_thm2}, and (\ref{LR_appn3}) holds because (i) $\max_{j \in \{1,2\}} C_n(\sqrt{n}\bs{t}(\ddot{\bs{\lambda}}^{j}, \alpha)) \geq 2[L_n(\widehat{\bs{\psi}},{\alpha}) -L_{0,n}(\widehat{\bs{\gamma}}_0,\widehat{\bs{\mu}}_0,\widehat{\bs{\Sigma}}_0)] + o_p(1)$ from the definition of $\bs{t}(\ddot{\bs{\lambda}}^{j}, \alpha)$ and (\ref{LR_appn2}), and (ii) $2[L_n(\widehat{\bs{\psi}},{\alpha}) -L_{0,n}(\widehat{\bs{\gamma}}_0,\widehat{\bs{\mu}}_0,\widehat{\bs{\Sigma}}_0)]\geq \max_{j \in \{1,2\}} C_n(\sqrt{n}\bs{t}(\ddot{\bs{\lambda}}^{j}, \alpha)) + o_p(1)$ from the definition of $\widehat{\bs{\psi}}$ and (\ref{LR_appn}).

We proceed to construct a parameter space $\tilde\Lambda_{\bs{\lambda}}^{j}$ that is locally equal to the cone $\Lambda_{\bs{\lambda}}^{j}$ defined in (\ref{Lambda-e}). Define $\ddot{\bs{\lambda}}_{\bs{\mu v}}^{j}$, $\ddot{\bs{\lambda}}_{\bs{v}^2}^{j}$, and $\ddot{\bs{\lambda}}_{\bs{\mu}^4}^j $ similarly to $\bs{\lambda}_{\bs{\mu v}}$, $\bs{\lambda}_{\bs{v}^2}$, $\bs{\lambda}_{\bs{\mu}^4}$ but using $\ddot{\bs{\lambda}}^{j}$ in place of $\bs{\lambda}$. Observe that the definition of $\Theta_{\bs{\lambda}}^{1}$  and $\Theta_{\bs{\lambda}}^{2}$, (\ref{ddot_rate}), and Lemma \ref{lemma_lambda_e}  in \ref{section:auxiliary}  imply that  $\ddot{\bs{\lambda}}_{\bs{\mu}^4}^{1} = o_p(n^{-1/2})$ and $\ddot{\bs{\lambda}}_{\bs{v}^2}^{2} = o_p(n^{-1/2})$. Therefore, 
\begin{align*}
&\bs{t}_{\bs{\mu v}}(\ddot{\bs{\lambda}}^{j},\alpha) =  12c(\alpha)\ddot{\bs{\lambda}}_{\bs{\mu v}}^{j}\ \text{ for $ j=1,2$}, \\
&\bs{t}_{\bs{\mu}^4}(\ddot{\bs{\lambda}}^{j},\alpha)= 
\begin{cases}
12 c(\alpha) \ddot{\bs{\lambda}}_{\bs{v}^2}^{1}+ o_p(n^{-1/2}) &\text{if } j=1,\\
c(\alpha) b(\alpha)\ddot{\bs{\lambda}}_{\bs{\mu}^4}^{2}+ o_p(n^{-1/2}) &\text{if } j=2.\\
\end{cases}
\end{align*}
Define 
\begin{equation*} 
\widetilde{\bs{t}}_{\bs{\mu v}} (\bs{\lambda},\alpha ) : = 12 c(\alpha)\bs{\lambda}_{\bs{\mu v}} \quad \text{and}\quad \widetilde{\bs{t}}_{\bs{\mu}^4}^{j}(\bs{\lambda},\alpha ): = 
\begin{cases}
12 c(\alpha) \bs{\lambda}_{\bs{v}^2}&\text{if } j=1,\\
c(\alpha) b(\alpha)\bs{\lambda}_{\bs{\mu}^4} & \text{if } j=2,
\end{cases} 
\end{equation*}
and
\begin{equation*} 
\widetilde{\Lambda}_{\bs{\lambda} }^{j}(\alpha) := \left\{ \left(\bs{t}_{\bs{\mu v}}\t, \bs{t}_{\bs{\mu}^4}\t\right)\t: \bs{t}_{\bs{\mu v}} = \widetilde{\bs{t}}_{\bs{\mu v}} (\bs{\lambda},\alpha ),\ \bs{t}_{\bs{\mu}^4}=\widetilde{\bs{t}}_{\bs{\mu}^4}^{j}(\bs{\lambda},\alpha)\ \text{for some $\bs{\lambda} \in \Theta_{\bs{\lambda}}$} \right\}.
\end{equation*}
Define $\widetilde{\bs{t}}_{\bs{\lambda}}^{j}$ by $C_n(\sqrt{n} \widetilde{\bs{t}}_{\bs{\lambda}}^{j}) = \max_{\bs{t}_{\bs{\lambda}} \in\widetilde{\Lambda}_{\bs{\lambda} }^{j} } C_n(\sqrt{n} \bs{t}_{\bs{\lambda}})$, then we have  $\max_{j \in \{1,2\}} C_n(\sqrt{n} \bs{t}(\ddot{\bs{\lambda}}^{j}, \alpha))= \max_{j \in \{1,2\}} C_n(\sqrt{n} \widetilde{\bs{t}}_{\bs{\lambda}}^{j}) +o_p(1)$. 
Therefore, in view of (\ref{LR_appn3}), we have
\begin{equation*}
2[L_n(\widehat{\bs{\psi}},{\alpha}) -L_{0,n}(\widehat{\bs{\gamma}}_0,\widehat{\bs{\mu}}_0,\widehat{\bs{\Sigma}}_0)] =\max_{j \in \{1,2\} } C_n(\sqrt{n} \widetilde{\bs{t}}_{\bs{\lambda}}^{j}) + o_p(1) .
\end{equation*}
The asymptotic distribution of the LRTS follows from applying Theorem 3(c) of \citet[][p.\ 1362]{andrews99em} to $C_n(\sqrt{n} \widetilde{\bs{t}}_{\bs{\lambda}}^{j})$. First, Assumption 2 of \citet{andrews99em} holds trivially for $C_n(\sqrt{n} \widetilde{\bs{t}}_{\bs{\lambda}}^{j})$. Second, Assumption 3 of \citet{andrews99em} holds with $B_T=n^{1/2}$ because $\bs{G}_{\bs{\lambda}.\bs{\eta} n} \to_d \bs{G}_{\bs{\lambda}.\bs{\eta}} \sim N(0,\bs{\mathcal{I}}_{\bs{\lambda}.\bs{\eta}})$ and $\bs{\mathcal{I}}_{\bs{\lambda}.\bs{\eta}}$ is nonsingular. Assumption 4 of \citet{andrews99em} holds from the same argument as (\ref{ddot_rate}). Assumption $5$ of \citet{andrews99em} follows from Assumption $5^*$ of \citet{andrews99em} because $\widetilde{\Lambda}_{\bs{\lambda}}^{j}$ is locally equal to the cone ${\Lambda}_{\bs{\lambda}}^{j}$. Therefore, it follows from Theorem 3(c) of \citet{andrews99em} that
$\max_{j \in \{1,2\}} C_n(\sqrt{n} \widetilde{\bs{t}}_{\bs{\lambda}}^{j}) \rightarrow_d \max_{j \in \{1,2\}} (\widehat{\bs{t}}_{\bs{\lambda}}^{j})^{\top}\bs{\mathcal{I}}_{\bs{\lambda},\bs{\eta}}\widehat{\bs{t}}_{\bs{\lambda}}^{j}$, 
giving the stated result.
\end{proof}

\begin{proof}[Proof of Proposition \ref{local_lr-2}]

For $m = 1,\ldots,M_0$, let $\mathcal{N}_{m}^* \subset \Theta_{\bs{\vartheta}_{M_0+1}}(\epsilon_1)$ be a sufficiently small closed neighborhood of $\Upsilon_{1m}^*$, such that  $\alpha_m,\alpha_{m + 1} > 0$ hold and $\Upsilon_{1k}^* \notin \mathcal{N}_{m}^*$ if $k\neq m$. For $\bs{\vartheta}_{M_0 + 1} \in \mathcal{N}_{m}^*$, we introduce the following one-to-one reparameterization, which is similar to (\ref{repara2}):
\begin{align*}
&\beta_{m}: = \alpha_{m} + \alpha_{m + 1}, \quad \tau: = \alpha_{m} /(\alpha_{m} + \alpha_{m + 1}), \\
&(\beta_1,\ldots,\beta_{m - 1},\beta_{m + 1}\ldots,\beta_{M_0 - 1})^{\top}: = (\alpha_1,\ldots,\alpha_{m - 1},\alpha_{m + 2},\ldots,\alpha_{M_0})^{\top},\\
&
\begin{pmatrix}
{\bs{\mu}}_m\\
{\bs{\mu}}_{m+1}\\
\bs{v}_m\\
\bs{v}_{m+1}
\end{pmatrix}
=
\begin{pmatrix}
\bs{\nu}_{\bs\mu} + (1-\tau) \bs{\lambda}_{\bs\mu} \\
\bs{\nu}_{\bs\mu} -\tau \bs{\lambda}_{\bs\mu}\\
\bs{\nu}_{\bs{v}} + (1- \tau)(2\bs{\lambda}_{\bs{v}}+ C_1 \bs{w}(\bs{\lambda}_{\bs\mu}\bs{\lambda}_{\bs\mu}\t) )\\
\bs{\nu}_{\bs{v}} - \tau(2\bs{\lambda}_{\bs{v}}+ C_2 \bs{w}(\bs{\lambda}_{\bs\mu}\bs{\lambda}_{\bs\mu}\t)
\end{pmatrix}, 
\end{align*}
where $\beta_{M_0} = 1 - \sum_{m = 1}^{M_0 - 1} \beta_m$,  and we suppress the dependence of $(\bs{\lambda}_{\bs{\mu}},\bs{\nu}_{\bs{\mu}},\bs{\lambda}_{\bs{v}},\bs{\nu}_{\bs{v}})$ on $\tau$. With this reparameterization, the null restriction $(\bs{\mu}_{m},\bs{\Sigma}_{m}) = (\bs{\mu}_{m + 1},\bs{\Sigma}_{m + 1})$ implied by $H_{0, 1m}$ holds if and only if $(\bs{\lambda}_{\bs{\mu}},\bs{\lambda}_{\bs{v}}) = \bs{0}$. Collect the reparameterized parameters except for $\tau$ into one vector $\bs{\psi}^m$, and let $\bs{\psi}^{m*}$ denote its true value. Define the reparameterized density as 
\begin{align*}
f_{\bs{M_0+1}}^m(\bs{x}|\bs{z}; \bs{\psi}^{m},\tau) & : = \beta_{m} g^m(\bs{x}|\bs{z}; \bs{\psi}^{m},\tau) + \sum_{j = 1}^{m - 1} \beta_j f_v(\bs{x}|\bs{z}; \bs{\gamma},\bs{\mu}_j,\bs{\Sigma}_j) + \sum_{j = m + 1}^{M_0} \beta_{j} f_v(\bs{x}|\bs{z}; \bs{\gamma},\bs{\mu}_{j + 1},\bs{\Sigma}_{j + 1}),
\end{align*}
where, similar to (\ref{loglike}),
\begin{align*}
g^m(\bs{x}|\bs{z}; \bs{\psi}^{m},\tau) & := \tau f_v \left(\bs{x}|\bs{z}; \bs{\gamma}, \bs{\nu}_{\bs{v}} + (1- \tau)(2\bs{\lambda}_{\bs{v}}+ C_1 \bs{w}(\bs{\lambda}_{\bs\mu}\bs{\lambda}_{\bs\mu}\t) ) \right)\\
& \quad + (1 - \tau) f_v \left(\bs{x}|\bs{z}; \bs{\gamma},\bs{\nu}_{\bs{\mu}} - \tau \bs{\lambda}_{\bs{\mu}}, \bs{\nu}_{\bs{v}} - \tau(2\bs{\lambda}_{\bs{v}}+ C_2 \bs{w}(\bs{\lambda}_{\bs\mu}\bs{\lambda}_{\bs\mu}\t) \right).
\end{align*} 
Observe that Lemma \ref{dv3}  in \ref{section:auxiliary} is applicable to $g^m(\bs{x}|\bs{z}; \bs{\psi}^{m},\tau)$ by replacing $\alpha$ with $\tau$. Define $L_n^m(\bs{\psi}^{m},\tau): = \sum_{i = 1}^n \log[f_{M_0+1}^m(\bs{X}_i|\bs{Z}_i; \bs{\psi}^{m},\tau)] $. Then, $L_{n}^m(\bs{\psi}^m,\tau) - L_n^m(\bs{\psi}^{m*},\tau)$ admits the same expansion as  $L_n(\bs{\psi},\alpha) - L_n(\bs{\psi}^*,\alpha)$ in Lemma \ref{P-quadratic} \ in \ref{section:quadratic} by replacing $(\bs{t}(\bs{\psi},\alpha), \bs{s}(\bs{x},\bs{z}), \bs{\mathcal{I}})$ with $(\bs{t}_{m}(\bs{\psi}^m,\tau), \bs{s}_{m}(\bs{x},\bs{z}), \bs{\mathcal{I}}^m)$, where $(\bs{s}_{m}(\bs{x},\bs{z}),\bs{\mathcal{I}}^m)$ is defined in the same manner as $(\bs{s}(\bs{x},\bs{z}),\bs{\mathcal{I}})$ but using $(\widetilde{\bs{s}}_{\bs{\eta}},\bs{s}_{\bs{\mu v}}^m,\bs{s}_{\bs{\mu}^4}^m)$ in place of $(\bs{s}_{\boldsymbol{\eta}},\bs{s}_{\bs{\lambda}})$. 

Define the local penalized MLE of $\bs{\psi}^{m}$ by
\begin{equation}
\widehat{\bs{\psi}}^m: = \arg\max_{\bs{\psi}^{m} \in\mathcal{N}_{m}^*} PL_n^m(\bs{\psi}^{m},\tau),\ \text{ where} \  PL_n^m(\bs{\psi}^{m},\tau): =L_n^m(\bs{\psi}^{m},\tau) + p_n(\bs{\psi}^{m}). \label{local_mle}
\end{equation}
Because $\bs{\psi}^{m*}$ is the only parameter value in $\mathcal{N}_m^*$ that generates true density, $\widehat{\bs{\psi}}^m - \bs{\psi}^{m*} = o_p(1)$ follows from a straightforward extension of Proposition \ref{P-consis}. For $\epsilon_\tau\in(0, 1/2)$, define the LRTS for testing $H_{0,1m}$ as $LR_{n,1m}(\epsilon_\tau) : = \max_{\tau\in[\epsilon_\tau,1 - \epsilon_\tau]}2\{L_n^m(\widehat{\bs{\psi}}^m,\tau) - L_{0,n}(\widehat{\bs{\vartheta}}_{M_0})\}$.  Observe that $p_n(\widehat{\bs{\psi}}^{m}) = o_p(1)$ because $a_n = o(1)$. Repeating the proof of Proposition \ref{P-LR-N1} for each local penalized MLE by replacing $\bs{G}_{n}$ with $\bs{G}_{n,m} := \nu_n (\bs{s}_m(\bs{x},\bs{z}))$ and collecting the results while noting that $(\bs{G}_{n,1}^{\top},\ldots ,\bs{G}_{n,M_0}^{\top})^{\top}\rightarrow_d (\bs{G}_1^{\top},\ldots ,\bs{G}_{M_0}^{\top})^{\top}$, we obtain
\begin{equation*}
(LR_{n,11}(\epsilon_\tau), \ldots, LR_{n,1M_0}(\epsilon_\tau))^{\top}\rightarrow_d (v_1,\ldots, v_{M_0})^{\top}, \end{equation*}
with $v_m$'s defined in Proposition \ref{local_lr-2}. Therefore, the stated result holds.
\end{proof}
 
\begin{proof}[Proof of Proposition \ref{EM_stat-1}]

For $j=1,2$, let $\omega_{n,m}^{j}$ be the sample counterpart of $(\widehat{\bs{t}}_{\bs{\lambda},m}^{j})^{\top} \bs{\mathcal{I}}_{\bs{\lambda}.\bs{\eta}}^m \widehat{\bs{t}}_{\bs{\lambda},m}^{j}$ in Proposition \ref{local_lr-2} such that the local LRTS satisfies $2 [L_n^m(\widehat{\bs{\psi}}_{\tau}^{m},\tau) - L_{0, n}(\widehat{\bs{\vartheta}}_{m_0}) ] = \max_{j}\{\omega_{n,m}^{j}\} + o_p(1)$, where $\widehat{\bs{\psi}}_{\tau}^{m}$ is the local penalized MLE defined as in (\ref{local_mle}) but using the penalty function $p_n^m(\bs{\vartheta}_{M_0+1})$ in (\ref{penalty-em}) in place of $p_n(\bs{\vartheta}_{M_0+1})$ in (\ref{pen_pmle}).

First, we show $\text{EM}_n^{m(1)} = \max_{j}\{\omega_{n,m}^{j}\} + o_p(1)$.  For $\tau \in (0, 1)$, define $\bs{\vartheta}_{M_0+1}^{m*}(\tau):= \{\bs{\vartheta}_{M_0+1} \in\Upsilon_{1m}^* : \alpha_m/(\alpha_m+\alpha_{m+1})=\tau\}$, which gives the true density.  Because $\bs{\vartheta}_{M_0+1}^{m*}(\tau_0)$ is the only value of $\bs{\vartheta}_{M_0+1}$ that yields the true density if $\bs{\varsigma} \in \Xi^*_m$ and $\alpha_m/(\alpha_m+\alpha_{m+1})=\tau_0$, $\bs{\vartheta}_{M_0+1}^{m(1)}(\tau_0)$ equals a reparameterized local penalized MLE in the neighborhood of $\bs{\vartheta}_{M_0+1}^{m*}(\tau_0)$.  Therefore, $2[PL_{n}(\bs{\vartheta}_{M_0+1}^{m(1)}(\tau_0)) - L_{0,n}(\widehat{\bs{\vartheta}}_{M_0})] =\max_{j}\{\omega_{n,m}^{j}\} + o_p(1)$ follows from repeating the proof of Proposition \ref{local_lr-2}, and  $\text{EM}_n^{m(1)} =\max_{j}\{\omega_{n,m}^{j}\} + o_p(1)$  holds by noting that $\{0.5\} \in \mathcal{T}$. 

We proceed to show that $\text{EM}_n^{m(K)} = \max_{j}\{\omega_{n,m}^{j}\} + o_p(1)$ for any finite $K$. Because a generalized EM step never decreases the likelihood value \citep{dempster77jrssb}, we have $PL_{n}^m(\bs{\vartheta}_{M_0+1}^{m(K)}(\tau_0))  + p( \tau^{(K)}) \geq PL_{n}^m(\bs{\vartheta}_{M_0+1}^{m(1)}(\tau_0))  + p( \tau_0)$. Therefore, it follows from Theorem 1 of \citet{chentan09jmva}, Lemma \ref{tau_update} in \ref{section:auxiliary}, and induction that $\bs{\vartheta}_{M_0+1}^{m(K)}(\tau_0) - \bs{\vartheta}_{M_0+1}^{m*}(\tau_0)= o_p(1)$ for any finite $K$. Let $\widetilde{\bs{\vartheta}}_{M_0+1}^{m}$ be the maximizer of $PL_{n}^m(\bs{\vartheta}_{M_0+1})$ under the constraint $\alpha_m/(\alpha_m + \alpha_{m+1}) = \tau^{(K)}$ in an arbitrary small closed neighborhood of $\bs{\vartheta}_{M_0+1}^{m*}(\tau^{(K)})$. Then, we have $PL_{n}(\widetilde{\bs{\vartheta}}_{M_0+1}^{m}) \geq PL_{n}^m({\bs{\vartheta}}_{M_0+1}^{m(K)}(\tau_0))+o_p(1)$ from the consistency of ${\bs{\vartheta}}_{M_0+1}^{m(K)}(\tau_0)$, and $2[PL_{n}(\widetilde{\bs{\vartheta}}_{M_0+1}^{m}) - L_{0,n}(\widehat{\bs{\vartheta}}_{M_0}) ] = \max_{j}\{\omega_{n,m}^{j}\} + o_p(1)$ holds from the definition of $\widetilde{\bs{\vartheta}}_{M_0+1}^{m}$. Furthermore, note that $PL_{n}(\bs{\vartheta}_{M_0+1}^{m(K)}(\tau_0)) \geq PL_{n}(\bs{\vartheta}_{M_0+1}^{m(1)}(\tau_0)) + o_p(1)$ from the definition of $\bs{\vartheta}_{M_0+1}^{m(K)}(\tau_0)$ and $\tau^{(K)} - \tau_0 = o_p(1)$, and we have already shown $2[PL_{n}(\bs{\vartheta}_{M_0+1}^{m(1)}(\tau_0)) - L_{0,n}(\widehat{\bs{\vartheta}}_{M_0})] = \max_{j}\{\omega_{n,m}^{j}\} + o_p(1)$. Therefore,  $2[PL_{n}({\bs{\vartheta}}_{M_0+1}^{m(K)}(\tau_0)) - L_{0, n}(\widehat{\bs{\vartheta}}_{M_0}) ] = \max_{j}\{\omega_{n,m}^{j}\} + o_p(1)$ holds  for all $m$, and  follows because $\tau^{(K)} - \tau_0 = o_p(1)$ and $\{0.5\} \in \mathcal{T}$. The stated result then follows from the definition of $\text{EM}_n^{(K)}$.
\end{proof}

\begin{proof}[Proof of Proposition \ref{P-LAN2}]
The proof follows the argument in the proof of Proposition \ref{P-LR-N1}. Observe that $\bs{h}_{\bs \eta}=0$ and $\bs{h}_{\bs\lambda} =\sqrt{n}\bs{t}_{\bs{\lambda}}(\bs{\lambda}_n,\alpha_n)+o(1)$ hold under $H_{1n}$. Therefore, Lemma \ref{P-LAN}   in \ref{section:auxiliary}  holds under $\mathbb{P}_{\bs{\vartheta}_n}^n$ implied by $H_{1n}$, and, in conjunction with Theorem 12.3.2 of \citet{lehmannromano05book}, Lemma \ref{P-quadratic} holds under $\mathbb{P}_{\bs\vartheta_n }^n$. Consequently, the proof of Proposition \ref{P-LR-N1} goes through if we replace $\bs{G}_{\bs{\lambda.\eta} n}\rightarrow_d \bs{G}_{\bs{\lambda.\eta} }$ with $G_{\bs{\lambda.\eta} n} \rightarrow_d G_{\bs{\lambda.\eta}} + ( \bs{\mathcal{I}_{\lambda}} - \bs{\mathcal{I}_{\lambda\eta}} \bs{\mathcal{I}_{\eta}}^{-1} \bs{\mathcal{I}_{\eta\lambda}} ) \bs{h}_{\bs\lambda} = \bs{G}_{\bs{\lambda.\eta}} + \bs{\mathcal{I}_{\lambda.\eta}} \bs{h_{\lambda}}$, and the stated result follows.
\end{proof}

\begin{proof}[Proof of Proposition \ref{P-bootstrap}]
The proof follows the argument in the proof of Theorem 15.4.2 in \citet{lehmannromano05book}. Define $\bf{C}_{\bs\eta}$ as the set of sequences $\{{\bs{\eta}}_n\}$ satisfying $\sqrt{n}({\bs{\eta}}_n - {\bs{\eta}}^*) \to \bs{h}_{\bs\eta}$ for some finite $ \bs{h}_{\bs\eta}$. 
Denote the MLE of the model with $M_0=1$ by $\hat{\bs{\eta}}_n$, then $\sqrt{n}(\hat{\bs{\eta}}_n - \bs{\eta}^*)$ converges in distribution to a $\mathbb{P}_{\bs{\vartheta}^*}$-a.s. finite random variable. Then, by the Almost Sure Representation Theorem (e.g., Theorem 11.2.19 of \citet{lehmannromano05book}), there exist random variables $\widetilde{\bs{\eta}}_n$ and $\widetilde{\bs{h}}_{\bs\eta}$ defined on a common probability space such that $\hat{\bs{\eta}}_n$ and $\widetilde{\bs{\eta}}_n$ have the same distribution and $\sqrt{n}(\widetilde{\bs{\eta}}_n -\bs \eta^*)\rightarrow \widetilde{\bs{h}}_{\bs\eta}$ almost surely. Therefore, $\{ \widetilde{\bs{\eta}}_n \}\in \bf{C}_{\bs\eta}$ with probability one, and the stated result under $H_0$ follows from Lemma \ref{lemma_btsp} in \ref{section:auxiliary} because $\hat{\bs{\eta}}_n$ and $\widetilde{\bs{\eta}}_n$ have the same distribution. 
For the MLE under $H_{1n}$, note that the proof of Proposition \ref{P-LAN2} goes through when $\bs{h}_{\bs\eta}$ is finite. Therefore, $\sqrt{n}(\hat{\bs{\eta}}_n - \bs{\eta}^*)$ converges in distribution to a $\mathbb{P}_{\bs\vartheta_n}$-a.s. finite random variable under $H_{1n}$. Hence, the stated result for LRTS follows from   repeating the argument in the case of $H_0$. The corresponding result for EM test  follows from the  asymptotic equivalence of  $LR_n^{M_0}$ and $EM_n^{(K)}$.
\end{proof}

\begin{proof}[Proof of Proposition \ref{P-LR-N1-homo-1}]
 
The proof is similar to that of Proposition \ref{P-LR-N1}. Let $(\widehat{\bs{\psi}},\widehat{\alpha})$ denote the reparameterization of $\widehat{\bs{\vartheta}}_2^1$. 
Write $\bs{t}(\bs{\psi},\alpha)$ in (\ref{tpsi_defn-homo}) as $(\bs{t}_{\bs{\eta}}\t,\bs{t}_{\bs{\lambda}}(\bs{\lambda},\alpha)\t)\t=(\bs{t}_{\bs{\eta}}\t,\bs{t}_{\bs{\mu}^3}(\bs{\lambda},\alpha)\t,\bs{t}_{\bs{\mu}^4}(\bs{\lambda},\alpha)\t)\t$. Repeating the argument that leads to (\ref{LR_appn2}) in the proof of Proposition \ref{P-LR-N1} but using Lemma \ref{P-quadratic-homo-1} in place of  Lemma \ref{P-quadratic}, we obtain
\begin{equation*}
2[L_n(\widehat{\bs{\psi}},\widehat{\alpha}) -L_{0,n}(\widehat{\bs{\gamma}}_0,\widehat{\bs{\mu}}_0,\widehat{\bs{\Sigma}}_0)] = C_n(\sqrt{n}\bs{t}(\widehat{\bs{\psi}},\widehat{\alpha})) + o_p(1),
\end{equation*}
where $C_n(\cdot)$ is defined as in (\ref{B_pi}) but using $\bs{s}(\bs{x},\bs{z})$ defined in (\ref{score_defn-homo}) in place of (\ref{score_defn}).

Partition the parameter space as $\Theta_{\bs{\lambda\alpha}}=\Theta_{\bs{\lambda\alpha}}^1\cup \Theta_{\bs{\lambda\alpha}}^2$ with
\begin{align} \label{theta_lambda_e-homo}
&\Theta_{\bs{\lambda\alpha}}^1 := \{ (\bs{\lambda},\alpha) \in \Theta_{\bs{\lambda\alpha}} : |1-2\alpha| \geq n^{-1/8} \log n \}\ \text{and}\ \Theta_{\bs{\lambda\alpha}}^2 := \{ (\bs{\lambda},\alpha) \in \Theta_{\bs{\lambda\alpha}} : |1-2\alpha| < n^{-1/8} \log n \},
\end{align} 
where $\Theta_{\bs\lambda\alpha}$ be the set of values of $(\bs\lambda,\alpha)$ such that the value of ${\bs{\vartheta}}_2$ implied by $(\bs{\lambda},\alpha)$ is in $\Theta_{{\bs{\vartheta}}_2,\zeta}^1$. 

Define $(\ddot{\bs{\lambda}}^j,\ddot{\alpha}^j)$ by $C_n(\sqrt{n} \bs{t}(\ddot{\bs{\lambda}}^j, \ddot{\alpha}^j)) = \max_{(\bs{\lambda},\alpha) \in \Theta_{\bs{\lambda\alpha}}^j} C_n(\sqrt{n} \bs{t}_{\bs{\lambda}}(\bs{\lambda},\alpha))$ for $j=1,2$. Then, repeating the argument following (\ref{ddot_rate})--(\ref{LR_appn3}), we have
\begin{align} 
& \bs{t}(\ddot{\bs{\lambda}}^j,\ddot{\alpha}^j) = (\bs{t}_{\bs{\mu}^3}(\ddot{\bs{\lambda}}^j,\ddot{\alpha}^j)\t,\bs{t}_{\bs{\mu}^4}(\ddot{\bs{\lambda}}^j,\ddot{\alpha}^j)\t)\t=O_p(n^{-1/2})\quad\text{for $j=1,2$}, \label{ddot_rate-homo} \\
& 2[L_n(\widehat{\bs{\psi}},\widehat{\alpha}) -L_{0,n}(\widehat{\bs{\gamma}}_0,\widehat{\bs{\mu}}_0,\widehat{\bs{\Sigma}}_0)] = \max_{j\in\{1,2\}}\left\{ C_n(\sqrt{n} \bs{t}(\ddot{\bs{\lambda}}^j, \ddot{\alpha}^j)) \right\}+ o_p(1).\label{LR_appn3-homo}
\end{align} 
 
Define $ \ddot{\bs{\lambda}}_{\bs{\mu}^3}^j$ and $\ddot{\bs{\lambda}}_{\bs{\mu}^4}^j $ similarly to $\bs{\lambda}_{\bs{\mu}^3}$ and $\bs{\lambda}_{\bs{\mu}^4}$ but using $\ddot{\bs{\lambda}}^j$ in place of $\bs{\lambda}$.
Observe that (\ref{theta_lambda_e-homo}) and (\ref{ddot_rate-homo}) imply that, with $c(\alpha):=\alpha(1-\alpha)$, 
\begin{equation} \label{Lambda-ddot-homo} 
\begin{aligned} 
&\bs{t}_{\bs{\mu}^3}(\ddot{\bs{\lambda}}^j,\ddot{\alpha}^j) = 
 c(\ddot{\alpha}^j)(1-2\ddot{\alpha}^j) \ddot{\bs{\lambda}}_{\bs{\mu}^3}^j + o_p(n^{-1/2})\quad\text{for $j=1, 2$}, \\
&\bs{t}_{\bs{\mu}^4}(\ddot{\bs{\lambda}}^j,\ddot{\alpha}^j) = 
\begin{cases} 
o_p(n^{-1/2}) &\text{if } j=1\\
 c(\ddot{\alpha}^j)(1-6\ddot{\alpha}^j+6(\ddot{\alpha}^j)^2) \ddot{\bs{\lambda}}_{\bs{\mu}^4}^j + o_p(n^{-1/2})&\text{if } j=2, 
 \end{cases} 
\end{aligned}
\end{equation}
where $\bs{t}_{\bs{\mu}^4}(\ddot{\bs{\lambda}}^1,\ddot{\alpha}^1) =o_p(n^{-1/2})$ holds because $|1-2\ddot{\alpha}^1|\geq n^{-1/8}\log n$ and $c(\ddot{\alpha}^{1})(1-2\ddot{\alpha}^1) \ddot{\bs{\lambda}}_{\bs{\mu}^3}^1=O_p(n^{-1/2})$ from (\ref{theta_lambda_e-homo}) and (\ref{ddot_rate-homo}) imply that $\ddot{\lambda}_i^1 = O_p(n^{-1/8}(\log n)^{-1/3})$ for any $i=1,\ldots,d$.

For $j=1,2$, consider the following set:
\begin{equation} \label{Lambda-tilde-e-homo}
\widetilde{\Lambda}_{\bs{\lambda} }^j := \left\{ \left( \bs{t}_{\bs{\mu}^3}\t, \bs{t}_{\bs{\mu}^4}\t\right)\t: \bs{t}_{\bs{\mu}^3} = \tilde{\bs{t}}_{\bs{\mu^3}}(\bs{\lambda},\alpha),\ \bs{t}_{\bs{\mu}^4} = \tilde{\bs{t}}_{\bs{\mu^4}}^j(\bs{\lambda},\alpha) \ 
\text{for some $(\bs{\lambda},\alpha) \in \Theta_{\bs{\lambda}\alpha}$}
\right\}, 
\end{equation}
where 
\begin{equation} \label{Lambda-tilde-constraint-homo}
 \tilde{\bs{t}}_{\bs{\mu}^3} (\bs{\lambda},\alpha) := 
 c(\alpha)(1-2\alpha) \bs{\lambda}_{\bs{\mu}^3}\quad\text{and}\quad
 \tilde{\bs{t}}_{\bs{\mu}^4}^j(\bs{\lambda},\alpha) := 
\begin{cases} 
0& \text{if } j=1,\\
 c(\alpha)(1-6\alpha+6\alpha^2) \bs{\lambda}_{\bs{\mu}^4} &\text{if } j=2.
 \end{cases} 
\end{equation} 
Define $\widetilde{\bs{t}}_{\bs{\lambda}}^j$ by $C_n(\sqrt{n} \widetilde{\bs{t}}_{\bs{\lambda}}^j) = \max_{\bs{t}_{\bs{\lambda}} \in\widetilde{\Lambda}_{\bs{\lambda} }^j } C_n(\sqrt{n} \bs{t}_{\bs{\lambda}})$. Then, it follows from (\ref{Lambda-ddot-homo}) and (\ref{Lambda-tilde-constraint-homo}) that $ C_n(\sqrt{n} \bs{t}(\ddot{\bs{\lambda}}^j, \ddot{\alpha}^j))= C_n(\sqrt{n} \widetilde{\bs{t}}_{\bs{\lambda}}^j) + o_p(1)$ for $j=1,2$. Therefore, in view of (\ref{LR_appn3-homo}), we have
\[
2[L_n(\widehat{\bs{\psi}},{\alpha}) -L_{0,n}(\widehat{\bs{\gamma}}_0,\widehat{\bs{\mu}}_0,\widehat{\bs{\Sigma}}_0)] = \max_{j\in\{1,2\}}\left\{C_n(\sqrt{n} \widetilde{\bs{t}}_{\bs{\lambda}}^{j}) \right\} + o_p(1) .
\]
Note that $\widetilde{\Lambda}_{\bs{\lambda} }^j$ in (\ref{Lambda-tilde-e-homo}) is locally (in a neighborhood of $\bs{\lambda}=0$) equal to the cone $\Lambda_{\bs{\lambda} }^j$ in (\ref{Lambda-e-homo}) for $j=1,2$ given that $\lim_{\alpha\rightarrow 1/2}(1-6\alpha+6\alpha^2) =-1/2$. Therefore, the stated result follows from applying Theorem 3(c) of \citet{andrews99em} by repeating the argument in the last paragraph of the proof of Proposition \ref{P-LR-N1}.
\end{proof}

\begin{proof}[Proof of Proposition \ref{P-LR-N1-homo-2}]

The proof is similar to that of Proposition \ref{P-LR-N1}. Let $(\widehat{\bs{\psi}}_{\bs{\lambda}}(\bs{\lambda}))=\arg\max_{{\bs{\psi}}_{\bs{\lambda}}\in \Theta_{\bs{\psi_\lambda}}} L_n( {\bs{\psi}}_{\bs{\lambda}}, {\bs{\lambda}}) $ for $\bs{\lambda}$ such that $|\bs{\lambda}|\geq \zeta$. 
 Let
\[
\bs{G}_{ n}(\bs{\lambda}) := \nu_n (\bs{s}(\bs{x},\bs{z};\bs{\lambda})) =
\begin{bmatrix}
\bs{G}_{\bs{\eta} n} \\
 {G}_{\alpha n}(\bs{\lambda})
\end{bmatrix}, \quad
\begin{aligned}
 {G}_{\alpha.\bs{\eta} n}(\bs{\lambda}) &:= {G}_{\alpha n}(\bs{\lambda}) - \bs{\mathcal{I}}_{\alpha\bs{\eta}}(\bs{\lambda})\bs{\mathcal{I}}_{\bs{\eta} }^{-1} \bs{G}_{\bs{\eta} n}, \\ 
 {Z}_{\alpha . \bs{\eta} n}(\bs{\lambda}) &:= \bs{\mathcal{I}}_{\alpha.\bs{\eta} }^{-1}{G}_{\alpha.\bs{\eta} n}(\bs{\lambda}).
\end{aligned}
\]

Following the argument that leads to (\ref{LR_appn2}) in the proof of Proposition \ref{P-LR-N1} but using Lemma \ref{P-quadratic-homo-2} in place of  Lemma \ref{P-quadratic}, we obtain
\begin{equation*} 
2[L_n(\widehat{\bs{\psi}}_{\bs{\lambda}}(\bs{\lambda}),\bs{\lambda})-L_{0,n}(\widehat{\bs{\gamma}}_0,\widehat{\bs{\mu}}_0,\widehat{\bs{\Sigma}}_0)] = C_n(\sqrt{n}|\bs{\lambda}|^3\widehat{\alpha}(\bs{\lambda}) ;\bs{\lambda}) + o_p(1),
\end{equation*}
where
\begin{equation*} 
C_n(t_\alpha;\bs{\lambda}) : = (Z_{\alpha .\bs{\eta} n}(\bs{\lambda}))^2{\mathcal{I}}_{\alpha.\bs{\eta} } (\bs{\lambda})- ( t_\alpha- {Z}_{\alpha.\bs{\eta} n}(\bs{\lambda}))^2 {\mathcal{I}}_{\alpha.\bs{\eta} }(\bs{\lambda}).
\end{equation*}

Define $\ddot{\alpha}(\bs{\lambda})$ by $C_n(\sqrt{n}|\bs{\lambda}|^3\ddot{\alpha}(\bs{\lambda});\bs{\lambda})=\max_{\alpha\in [0,3/4]} C_n(\sqrt{n} |\bs{\lambda}|^3\alpha ;\bs{\lambda}) $. Then, repeating the argument following (\ref{ddot_rate})--(\ref{LR_appn3}), we have 
\begin{align} 
& |\bs{\lambda}|^3 \ddot{\alpha}(\bs{\lambda}) = O_p(n^{-1/2}), \label{ddot_rate-homo-2} \\
& 2[L_n(\widehat{\bs{\psi}}_{\bs{\lambda}}(\bs{\lambda}),\bs{\lambda})-L_{0,n}(\widehat{\bs{\gamma}}_0,\widehat{\bs{\mu}}_0,\widehat{\bs{\Sigma}}_0)] = C_n(\sqrt{n} |\bs{\lambda}|^3\ddot{\alpha}(\bs{\lambda}) ;\bs{\lambda}) + o_p(1). \label{LR_appn3-homo2}
\end{align}
Define $\widetilde{t}_{\alpha}(\bs{\lambda})$ by $C_n(\sqrt{n}\widetilde{t}_{\alpha}(\bs{\lambda});\bs{\lambda})=\max_{t_\alpha\in \widetilde{\Lambda}_\alpha(\bs{\lambda})} C_n(\sqrt{n} t_{\alpha} ;\bs{\lambda}) $, where 
 $\widetilde{\Lambda}_\alpha(\bs{\lambda}): = \{t_\alpha: t_\alpha=|\bs{\lambda}|^3 \alpha \text{ for some $\alpha\in[0,3/4]$}\}$. Then, because $C_n(\sqrt{n}\widetilde{t}_{\alpha}(\bs{\lambda});\bs{\lambda})=C_n(\sqrt{n} |\bs{\lambda}|^3\ddot{\alpha}(\bs{\lambda}) ;\bs{\lambda}) $, (\ref{LR_appn3-homo2}) implies that
\[
2[L_n(\widehat{\bs{\psi}}_{\bs{\lambda}}(\bs{\lambda}),\bs{\lambda})-L_{0,n}(\widehat{\bs{\gamma}}_0,\widehat{\bs{\mu}}_0,\widehat{\bs{\Sigma}}_0)] = C_n(\sqrt{n}\widetilde{t}_{\alpha}(\bs{\lambda});\bs{\lambda}) + o_p(1).
\] 
The stated result follows from Theorem 1(c) of \citet{andrews01em} as
\begin{equation*} 
\sup_{\Theta_{\bs{\lambda}} \cap \{ |\bs{\lambda}|\geq \zeta \}} C_n(\sqrt{n}\widetilde{t}_{\alpha}(\bs{\lambda});\bs{\lambda}) \rightarrow_d \sup_{\Theta_{\bs{\lambda}} \cap \{ |\bs{\lambda}|\geq \zeta \}} (\widehat{{t}_\alpha}(\bs{\lambda}) )^2 {\mathcal{I}}_{\alpha.\bs{\eta}}(\bs{\lambda}),
\end{equation*} 
where Assumption 2 of \citet{andrews01em} trivially holds for $C_n(\sqrt{n}\widetilde{t}_{\alpha};\bs{\lambda})$; Assumption 3 of \citet{andrews01em} holds with $B_T=n^{1/2}$ because ${G}_{\alpha.\bs{\eta} n}(\bs{\lambda})\Rightarrow {G}_{\alpha.\bs{\eta}}(\bs{\lambda})$ from (\ref{weak_cgce}); Assumption 4 of \citet{andrews01em} holds from the same argument as (\ref{ddot_rate-homo-2}); Assumption 5 of \citet{andrews01em} holds because $\tilde \Lambda_\alpha(\bs{\lambda})$ is locally equal to the cone $\mathbb{R}_+$. 
\end{proof}

\begin{proof}[Proof of Proposition \ref{local_lr-2-homo}]

We omit the proof of Proposition \ref{local_lr-2-homo} because its argument is similar to that of Proposition \ref{local_lr-2} except that the proof of Proposition \ref{local_lr-2-homo} will refer to the proof of Proposition \ref{P-LR-N1-homo} in place of that of Proposition \ref{P-LR-N1}.
\end{proof}

\section{Quadratic approximation of the log-likelihood function}\label{section:quadratic}

 When testing the number of components by the likelihood ratio test, the Fisher information matrix becomes singular and the log-likelihood function will be approximated by a quadratic function of polynomials of parameters. Further, a part of parameter is not identified under the null hypothesis. This section establishes a quadratic approximation of the log-likelihood function  using the results in \ref{section:expansion} and \ref{section:auxiliary}. Lemma \ref{P-quadratic} considers the case of testing $H_0:M=1$ against $H_A:M=2$ in the heteroscedastic case. For a sequence $X_{n\varepsilon}$ indexed by $n=1,2,\ldots$ and $\varepsilon$, we write $X_{n\varepsilon} = O_{p\varepsilon}(a_n)$ if, for any $\Delta>0$, there exist $\varepsilon>0$ and $M, n_0 <\infty$ such that $\mathbb{P}(|X_{n\varepsilon}/a_n| \leq M) \geq 1- \Delta$ for all $n > n_0$, and we write $X_{n\varepsilon} = o_{p\varepsilon}(a_n)$ if, for any $\Delta_1,\Delta_2>0$, there exist $\varepsilon>0$ and $n_0$ such that $\mathbb{P}(|X_{n\varepsilon}/a_n| \leq \Delta_1) \geq 1- \Delta_2$ for all $n > n_0$. Loosely speaking, $X_{n\varepsilon} = O_{p\varepsilon}(a_n)$ and $X_{n\varepsilon} = o_{p\varepsilon}(a_n)$ mean that $X_{n\varepsilon} = O_{p}(a_n)$ and $X_{n\varepsilon} = o_{p}(a_n)$ when $\varepsilon$ is sufficiently small, respectively. 

\begin{lemma} \label{P-quadratic}
Suppose that Assumptions \ref{assn_consis} and \ref{A-taylor1} hold and $\bs{X}$ given $\bs{Z}$ has the density $f(\bs{x}|\bs{z}; \bs{\gamma},{\bs{\mu}},\bs{\Sigma})$ defined in (\ref{normal_density}). Let $L_n(\bs{\psi},\alpha): = \sum_{i = 1}^n \log g(\bs{X}_i|\bs{Z}_i;\bs{\psi},\alpha)$ with $g(\bs{x}|\bs{z};\bs{\psi},\alpha)$ defined in (\ref{loglike}). For $\alpha \in (0,1)$, define $\bs{s}(\bs{x},\bs{z})$ and $\bs{t}(\bs{\psi},\alpha)$ as in (\ref{score_defn}) and (\ref{tpsi_defn}), and let $\mathcal{N}_{\varepsilon} := \{ \bs{\vartheta}_2 \in \Theta_{\bs{\vartheta}_2} : |\bs{t}(\bs{\psi},\alpha)|< \varepsilon\}$ and $\bs{\mathcal{I}}:=E[\bs{s}(\bs{X},\bs{Z})\bs{s}(\bs{X},\bs{Z})\t]$. Then, for $\epsilon_\sigma\in(0,1)$ and any $\delta>0$, we have (a) $\sup_{\bs{\vartheta}_2 \in A_{n \varepsilon}(\delta)} |\bs{t}(\bs{\psi},\alpha)| = O_{p\varepsilon}(n^{-1/2})$;
\begin{equation*} 
(b)\ \sup_{\bs{\vartheta}_2 \in A_{n \varepsilon}(\delta)}\left|L_n(\bs{\psi},\alpha) - L_n(\bs{\psi}^*,\alpha) - \sqrt{n} \bs{t}(\bs{\psi},\alpha)\t \nu_n(\bs{s}(\bs{x},\bs{z})) + n \bs{t}(\bs{\psi},\alpha)\t \bs{\mathcal{I}} \bs{t}(\bs{\psi},\alpha)/2 \right| = o_{p\varepsilon}(1),
\end{equation*}
where $A_{n \varepsilon}(\delta) := \{\bs{\vartheta}_2 \in \mathcal{N}_\varepsilon: L_n(\bs{\psi},\alpha) - L_n(\bs{\psi}^*,\alpha) \geq -\delta \}$.
\end{lemma}

\begin{proof}[Proof of Lemma \ref{P-quadratic}] We prove the stated result by using Lemma \ref{Ln_thm2}  in \ref{section:expansion}, where $\ell(\bs{y},\bs{\psi},\alpha):=g(\bs{x}|\bs{z};\bs{\psi}^*,\alpha)/g(\bs{x}|\bs{z};\bs{\psi},\alpha)$  with $\bs{y}:=(\bs{x}\t,\bs{z}\t)\t$ plays the role of $\ell(\bs{y},\bs{\vartheta})$ as defined in (\ref{density_ratio})  and $\bs{t}(\bs{\psi},\alpha)$ plays the role of $\bs{t}(\bs{\vartheta})$. Observe that $\bs{t}(\bs{\psi},\alpha)$ defined in (\ref{tpsi_defn}) satisfies $\bs{t}(\bs{\psi},\alpha) = 0$ if and only if $\bs{\psi}=\bs{\psi}^*$ because $\bs{\lambda}=0$ if and only if  the $(i,i,i,i)$th element of  $ 12 \bs{\lambda}_{\bs{ v}^2} + b(\alpha) \bs{\lambda}_{\bs{\mu}^4} $  is 0  for all $1 \leq i \leq d$. We expand $ \ell(\bs{y},\bs{\psi},\alpha) -1$ five times with respect to $\bs{\psi}$ and show that the expansion satisfies Assumption \ref{assn_expansion}  in \ref{section:expansion}.

Define 
\begin{equation} \label{v_defn}
\bs{v}(\bs{y};\bs{\vartheta}_2):=(\nabla_{\bs{\psi}} g(\bs{x}|\bs{z};\bs{\psi},\alpha)\t, \nabla_{\bs{\psi}^{\otimes 2}} g(\bs{x}|\bs{z};\bs{\psi},\alpha)\t, \ldots, \nabla_{\bs{\psi} ^{\otimes 5}} g(\bs{x}|\bs{z};\bs{\psi},\alpha)\t)\t/g(\bs{x}|\bs{z};\bs{\psi}^*,\alpha),
\end{equation}
which satisfies $E[\bs{v}(\bs{Y};\bs{\vartheta}_2)]=0$.  In order to apply Lemma \ref{Ln_thm2} to $ \ell(\bs{y},\bs{\psi},\alpha)  -1$, we first show
\begin{align} 
&\sup_{\bs{\vartheta}_2 \in \mathcal{N}_\varepsilon}\left|P_n [\bs{v}(\bs{y};\bs{\vartheta}_2)\bs{v}(\bs{y};\bs{\vartheta}_2)\t] - E[\bs{v}(\bs{Y};\bs{\vartheta}_2)\bs{v}(\bs{Y};\bs{\vartheta}_2)\t]\right| = o_p(1), \label{uniform_lln} \\
&\nu_n(\bs{v}(\bs{y};\bs{\vartheta}_2)) \Rightarrow \bs{W}(\bs{\vartheta}_2), \label{weak_cgce}
\end{align}
where $\bs{W}(\bs{\vartheta}_2)$ is a mean-zero continuous Gaussian process with $E[\bs{W}(\bs{\vartheta}_2)\bs{W}(\bs{\vartheta}_2')\t] = E[\bs{v}(\bs{Y};\bs{\vartheta}_2)\bs{v}(\bs{Y};\bs{\vartheta}_2')\t]$. (\ref{uniform_lln}) holds because $\bs{v}(\bs{Y}_i;\bs{\vartheta}_2)\bs{v}(\bs{Y}_i;\bs{\vartheta}_2)\t$ satisfies a uniform law of large numbers (see, for example, Lemma 2.4 of \citet{neweymcfadden94hdbk}) because $\bs{v}(\bs{y};\bs{\vartheta}_2)$ is continuous in $\bs{\vartheta}_2$ and $E\sup_{\bs{\vartheta}_2 \in \mathcal{N}_{\epsilon}}|\bs{v}(\bs{Y};\bs{\vartheta}_2)|^2 < \infty$ from the property of the normal density and Assumption \ref{A-taylor1}. (\ref{weak_cgce}) follows from Theorem 10.2 of \citet{pollard90book} if (i) $\Theta_{\bs{\vartheta}_2}$ is totally bounded, (ii) the finite dimensional distributions of $\nu_n(\bs{v}(\bs{y};\bs{\vartheta}_2))$ converge to those of $\bs{W}(\bs{\vartheta}_2)$, and (iii) $\{\nu_n(\bs{v}(\bs{y};\bs{\vartheta}_2)): n\geq 1\}$ is stochastically equicontinuous. Condition (i) holds because $\Theta_{\bs{\vartheta}_2}$ is compact in the Euclidean space. Condition (ii) follows from Assumption \ref{A-taylor1} and the multivariate CLT. Condition (iii) holds Theorem 2 of \citet{andrews94hdbk} because $\bs{v}(\bs{y};\bs{\vartheta}_2)$ is Lipschitz continuous in $\bs{\vartheta}_2$.

Note that the $(p+1)$-th order Taylor expansion of $g(\bs{\psi})$ around $\bs{\psi}=\bs{\psi}^*$ is given by
\begin{align*}
g(\bs{\psi})& = g(\bs{\psi}^*)+ \sum_{j=1}^p \frac{1}{j!} \nabla_{(\bs{\psi}^{\otimes j})\t} g(\bs{\psi}^*) (\bs{\psi}-\bs{\psi}^*)^{\otimes j} + \frac{1}{(p+1)!} \nabla_{(\bs{\psi}^{\otimes (p+1)})\t} g(\overline{\bs{\psi}}) (\bs{\psi}-\bs{\psi}^*)^{\otimes (p+1)} ,
\end{align*}
where $\overline{\bs{\psi}}$ lies between $\bs{\psi}$ and $\bs{\psi}^*$, and $\overline{\bs{\psi}}$ may differ from element to element of $\nabla_{(\bs{\psi}^{\otimes (p+1)})\t} g(\overline{\bs{\psi}})$.

Let $g^*$ and $\nabla g^*$ denote $g(\bs{x}|\bs{z};\bs{\psi}^*,\alpha)$ and $\nabla g(\bs{x}|\bs{z};\bs{\psi}^*,\alpha)$, and let $\nabla\overline{g}$ denote $\nabla g(\bs{x}|\bs{z};\overline{\bs{\psi}},\alpha)$. Let $\dot{\bs{\psi}}:=\bs{\psi} - \bs{\psi}^*$ and $\dot{\bs{\eta}}:=\bs{\eta} - \bs{\eta}^*$. Expanding $\ell(\bs{y};\bs{\psi},\alpha)$ five times around $\bs{\psi}^*$ while fixing $\alpha$ and using Lemma \ref{dv3}  in \ref{section:auxiliary}, we can write $\ell(\bs{y};\bs{\psi},\alpha)-1$ as 
\[
\ell(\bs{y};\bs{\psi},\alpha) -1 = s(\bs{y};\bs{\eta},\bs{\lambda}) + r(\bs{y};\bs{\eta},\bs{\lambda}),
\]
where  
\begin{equation*}
s(\bs{y};\bs{\eta},\bs{\lambda}) :=\frac{\nabla_{\bs{\eta}\t}g^* }{g^*}\dot{\bs{\eta}} + \frac{1}{2!} \frac{\nabla_{(\bs{\lambda}^{\otimes 2})\t} g^*}{g^*} \bs{\lambda}^{\otimes 2} + \frac{1}{4!} \frac{\nabla_{(\bs{\lambda}_{\bs{\mu}}^{\otimes 4})\t} g^* }{g^*} \bs{\lambda}_{\bs{\mu}}^{\otimes 4},
\end{equation*}
and, with $\dot{\bs{\psi}}_-:=(\dot{\bs{\eta}}\t,\bs{\lambda}_{\bs v}\t)\t$,
\begin{align}
r(\bs{y};\bs{\eta},\bs{\lambda}) &:= \frac{1}{2!} \frac{\nabla_{(\bs{\eta}^{\otimes 2})\t}g^*}{g^*} \dot{\bs{\eta}}^{\otimes 2} + \frac{1}{3!} \frac{\nabla_{(\bs{\psi}^{\otimes 3})\t} g^* }{g^*} \dot{\bs{\psi}}^{\otimes 3} \label{l_expand_2} \\
& \quad + \frac{1}{4!}\sum_{p=0}^{3}  \binom{4}{p}  \frac{\nabla_{(\bs{\psi}_-^{\otimes (4-p)} \otimes \bs{\lambda}_{\bs{\mu}}^{\otimes p} )\t} g^* }{g^*} (\dot{\bs{\psi}}_-^{\otimes (4-p)} \otimes \bs{\lambda}_{\bs{\mu}}^{\otimes p} ) \label{l_expand_3} \\
&\quad + \frac{1}{5!} \frac{\nabla_{(\bs{\lambda}_{\bs{\mu}}^{\otimes 5})\t} \overline{g} }{g^*} \bs{\lambda}_{\bs{\mu}}^{\otimes 5} +  \frac{1}{5!}  \sum_{p=0}^{4}  \binom{5}{p}\frac{\nabla_{(\bs{\psi}_-^{\otimes (5-p)} \otimes \bs{\lambda}_{\bs{\mu}}^{\otimes p} )\t} \overline{g} }{g^*} (\dot{\bs{\psi}}_-^{\otimes (5-p)} \otimes \bs{\lambda}_{\bs{\mu}}^{\otimes p} ) \label{l_expand_4}.
\end{align} 
 $s(\bs{y};\bs{\eta},\bs{\lambda})$ is the leading term in the expansion. We first show $s(\bs{y};\bs{\eta},\bs{\lambda}) = \bs{t}(\bs{\psi},\alpha)\t \bs{s}(\bs{x},\bs{z})$ with $\bs{s}(\bs{x},\bs{z})$ and $\bs{t}(\bs{\psi},\alpha)$ defined in (\ref{score_defn})--(\ref{lambda_muv_defn}). Let $f^*$ and $\nabla f^*$ denote $f(\bs{x}|\bs{z};\bs{\gamma}^*,\bs{\mu}^*,\bs{\Sigma}^*)$ and $\nabla f(\bs{x}|\bs{z};\bs{\gamma}^*,\bs{\mu}^*,\bs{\Sigma}^*)$. The first term of $s(\bs{y};\bs{\eta},\bs{\lambda})$  is simply $(\nabla_{(\bs{\gamma}\t,\bs{\mu}\t,\bs{v}\t)\t}f^*/f^*)\dot{\bs{\eta}}$. Using Lemmas \ref{mv_derivative} and \ref{dv3} and commutativity of partial derivatives, the second term of $s(\bs{y};\bs{\eta},\bs{\lambda})$ is written as $(1/2!)(\nabla_{(\bs{\lambda}^{\otimes 2})\t} g^*/g^*) \bs{\lambda}^{\otimes 2}
 = (\nabla_{(\bs{\lambda}_{\bs{\mu}} \otimes \bs{\lambda}_{\bs{v}})^{\top}} g^*/g^*) (\bs{\lambda}_{\bs{\mu}} \otimes \bs{\lambda}_{\bs{v}}) + (1/2) (\nabla_{(\bs{\lambda}_{\bs{v}}^{\otimes 2})\t} g^*/g^*) \bs{\lambda}_{\bs{v}}^{\otimes 2}$. Observe that
\begin{equation} \label{nabla_lambda_v_mu}
\begin{aligned}
\frac{\nabla_{(\bs{\lambda}_{\bs{\mu}} \otimes \bs{\lambda}_{\bs{v}})^{\top}} g^*}{g^*} (\bs{\lambda}_{\bs{\mu}} \otimes \bs{\lambda}_{\bs{v}}) 
& = \alpha(1-\alpha) \sum_{\substack{1\leq i \leq d \\ 1 \leq j \leq k \leq d}} \frac{\nabla_{\mu_i \mu_j \mu_k} f^* }{f^*} \lambda_{\mu_i}\lambda_{v_{jk}} \\
& = \alpha(1-\alpha) \sum_{1 \leq i \leq j \leq k \leq d} \frac{\nabla_{\mu_i \mu_j \mu_k} f^* }{f^*} \sum_{(t_1,t_2,t_3) \in p_{12}(i,j,k)} \lambda_{\mu_{t_1}}\lambda_{v_{t_2t_3}}\\
& = \alpha(1-\alpha) 12 \bs{s}_{\bs{\mu v}}\t \bs{\lambda}_{\bs{\mu v}},
\end{aligned}
\end{equation}
where $\sum_{(t_1,t_2,t_3) \in p_{12}(i,j,k)}$ denotes the sum over all distinct permutations of $(i,j,k)$ to $(t_1,t_2,t_3)$ with $t_2 \leq t_3$, and
\begin{align*}
\frac{1}{2} \frac{\nabla_{(\bs{\lambda}_{\bs{v}}^{\otimes 2})^{\top}} g^*}{g^*} \bs{\lambda}_{\bs{v}}^{\otimes 2} 
& = \frac{\alpha(1-\alpha)}{2} \sum_{\substack{1\leq i \leq j \leq d \\ 1 \leq k \leq \ell \leq d}} \frac{\nabla_{\mu_i \mu_j \mu_k \mu_\ell} f^* }{f^*} \lambda_{v_{ij}}\lambda_{v_{k\ell}} \\
& = \frac{\alpha(1-\alpha)}{2} \sum_{1 \leq i \leq j \leq k \leq \ell \leq d} \frac{\nabla_{\mu_i \mu_j \mu_k \mu_\ell} f^* }{f^*} \sum_{(t_1,t_2,t_3,t_4) \in p_{22}(i,j,k,\ell)} \lambda_{v_{t_1t_2}}\lambda_{v_{t_3t_4}},
\end{align*}
where $\sum_{(t_1,t_2,t_3,t_4) \in p_{22}(i,j,k,\ell)}$ denotes the sum over all distinct permutations of $(i,j,k,\ell)$ to $(t_1,t_2,t_3,t_4)$ with $t_1 \leq t_2$ and $t_3 \leq t_4$. From Lemma \ref{dv3}, the third term  of $s(\bs{y};\bs{\eta},\bs{\lambda})$ is written as
\begin{align*}
\frac{1}{4!} \frac{\nabla_{(\bs{\lambda}_{\bs{\mu}}^{\otimes 4})\t} g^* }{g^*} \bs{\lambda}_{\bs{\mu}}^{\otimes 4} 
& = \frac{\alpha(1-\alpha)}{4!} \sum_{1 \leq i \leq j \leq k \leq \ell \leq d} b(\alpha) \frac{\nabla_{\mu_i \mu_j \mu_k \mu_\ell} f^* }{f^*} \sum_{(t_1,t_2,t_3,t_4) \in p(i,j,k,\ell)} \lambda_{\mu_{t_1}}\lambda_{\mu_{t_2}}\lambda_{\mu_{t_3}}\lambda_{\mu_{t_4}},
\end{align*}
where $\sum_{(t_1,t_2,t_3,t_4) \in p(i,j,k,\ell)}$ denotes the sum over all distinct permutations of $(i,j,k,\ell)$ to $(t_1,t_2,t_3,t_4)$.  Therefore, the sum of $(1/2) (\nabla_{(\bs{\lambda}_{\bs{v}}^{\otimes 2})\t} g^*/g^*) \bs{\lambda}_{\bs{v}}^{\otimes 2}$ and $(1/4!) (\nabla_{(\bs{\lambda}_{\bs{\mu}}^{\otimes 4})\t} g^* / g^*) \bs{\lambda}_{\bs{\mu}}^{\otimes 4} $ is $\alpha(1-\alpha) \bs{s}_{\bs{\mu}^4}\t[12\bs{\lambda}_{\bs{v}^2} + b(\alpha)\bs{\lambda}_{\bs{\mu}^4}]$, and hence $s(\bs{y};\bs{\eta},\bs{\lambda})=\bs{t}(\bs{\psi},\alpha)\t \bs{s}(\bs{x},\bs{z})$ holds.  

$\bs{s}(\bs{x},\bs{z})$ clearly satisfies Assumption \ref{assn_expansion}(a)(b)(e) from Assumption \ref{A-taylor1}, the property of the normal density, (\ref{uniform_lln}), and (\ref{weak_cgce}). Therefore, the stated result holds if $r(\bs{y};\bs{\eta},\bs{\lambda})$ defined in (\ref{l_expand_2})--(\ref{l_expand_4}) satisfies Assumption \ref{assn_expansion}(c)(d). We proceed to show that  (\ref{l_expand_2})--(\ref{l_expand_4}) can be expressed as $\bs{\xi}(\bs{y};\bs{\vartheta}) O(|\bs{\psi}-\bs{\psi}^*||\bs{t}(\bs{\psi},\alpha)|)$ where $\sup_{\bs{\vartheta}_2 \in \mathcal{N}_\epsilon}|\bs{\xi}(\bs{y};\bs{\vartheta}_2)| \leq \sup_{\bs{\vartheta}_2 \in \mathcal{N}_\epsilon}|\bs{v}(\bs{y};\bs{\vartheta}_2)|$  with $\bs{v}(\bs{y};\bs{\vartheta}_2)$ defined in (\ref{v_defn}). Then, Assumption \ref{assn_expansion}(c)(d) follows from  the property of the normal density  and (\ref{weak_cgce}).

First, the first term on the right hand side of (\ref{l_expand_2}) is written as $(\nabla_{\bs{\eta}^{\otimes 2}} g^*/g^*) O(|\dot{\bs{\eta}}|^2)$. Second, write the second term in (\ref{l_expand_2}) as $(1/3!)\sum_{p=0}^3 \binom{3}{p} (\nabla_{(\bs{\eta}^{\otimes p} \otimes \bs{\lambda}^{\otimes (3-p)} )\t} g^*/g^*) (\dot{\bs{\eta}}^{\otimes p} \otimes \bs{\lambda}^{\otimes (3-p)} )$. The terms with $p \geq 1$ are written as $(\nabla_{\bs{\psi}^{\otimes 3}} g^*/g^*) O(|\dot{\bs{\eta}}|)O(|\bs{\lambda}|)$. Write the term with $p=0$ as $(1/3!)\sum_{q=0}^3 \binom{3}{q} A_q$, where $A_q:=(\nabla_{(\bs{\lambda}_{\bs v}^{\otimes q} \otimes \bs{\lambda}_{\bs \mu}^{\otimes (3-q)} )\t} g^*/g^*) (\bs{\lambda}_{\bs v}^{\otimes q} \otimes \bs{\lambda}_{\bs \mu}^{\otimes (3-q)} )$. We have $A_0=0$ because $\nabla_{\lambda_{\mu_i} \lambda_{\mu_j} \lambda_{\mu_k}} g^*=0$ from Lemma \ref{dv3}. From a similar argument to (\ref{nabla_lambda_v_mu}), we obtain $A_1=\sum_{i=1}^d \lambda_{\mu_i} ( \nabla_{(\bs{\lambda}_{\bs v} \otimes \bs{\lambda}_{\bs \mu} )\t} \nabla_{\lambda_{\mu_i}}g^*/g^*) (\bs{\lambda}_{\bs v} \otimes \bs{\lambda}_{\bs \mu}) = (\nabla_{\bs{\lambda}^{\otimes 3}}g^* / g^*) O(|\bs{\lambda}||\bs{\lambda}_{\bs{\mu v}}|)$. A similar argument gives  $A_2=(\nabla_{\bs{\lambda}^{\otimes 3}}g^*/g^*) O(|\bs{\lambda}||\bs{\lambda}_{\bs{\mu v}}|)$.  For $A_3$,  observe that, for any sequence $a_{ijk\ell m n}$,
\begin{align*}
& \sum_{1 \leq i \leq j \leq d} \sum_{1 \leq k \leq \ell \leq d} \sum_{1 \leq m \leq n \leq d} a_{ijk\ell m n} \lambda_{v_{ij}}\lambda_{v_{k\ell}} \lambda_{v_{m n}} \\
& = \sum_{1 \leq m \leq n \leq d} \lambda_{v_{m n}} \sum_{1 \leq i \leq j \leq k \leq \ell \leq d} a_{ijk\ell mn} \left( 12 \sum_{(t_1,t_2,t_3,t_4) \in p_{22}(i,j,k,\ell)} \lambda_{v_{t_1t_2}}\lambda_{v_{t_3t_4}} \right.\\
& \left. \quad + b(\alpha) \sum_{(t_1,t_2,t_3,t_4) \in p(i,j,k,\ell)} \lambda_{\mu_{t_1}}\lambda_{\mu_{t_2}}\lambda_{\mu_{t_3}}\lambda_{\mu_{t_4}} \right)
\\
& \quad -b(\alpha) \sum_{i=1}^d\sum_{j=1}^d\sum_{k=1}^d \lambda_{\mu_{i}}\lambda_{\mu_{j}}\lambda_{\mu_{k}} \sum_{1 \leq \ell \leq m \leq n \leq d}\ a_{ijk\ell m n} \sum_{(t_1,t_2,t_3) \in p_{12}(\ell,m,n)} \lambda_{\mu_{t_1}}\lambda_{v_{t_2t_3}}.
\end{align*}
Using this result with $a_{ijk\ell mn}=\nabla_{\lambda_{v_{i j}} \lambda_{v_{k \ell}} \lambda_{v_{m n}}}g^*/g^*$ gives $A_3=(\nabla_{\bs{\psi}^{\otimes 3}} g^*/g^*) [O(|\bs{\lambda}||12\bs{\lambda}_{\bs{v}^2}+ b(\alpha) \bs{\lambda}_{\bs{\mu}^4}|) + O(|\bs{\lambda}||\bs{\lambda}_{\bs{\mu v}}|)]$. Therefore, the terms on the right hand side of (\ref{l_expand_2}) are $\bs{v}(\bs{y};\bs{\vartheta}_2) O(|\bs{\psi}-\bs{\psi}^*||\bs{t}(\bs{\psi},\alpha)|)$. We proceed to bound (\ref{l_expand_3}). The terms in (\ref{l_expand_3}) with $p \geq 1$ are written as $(\nabla_{\bs{\psi}^{\otimes 4}}g^* / g^*) [ O(|\bs{\lambda}||\bs{\lambda}_{\bs{\mu v}}|) + O(|\bs{\lambda}||\dot{\bs{\eta}}|)]$ because they contain either $\sum_{1 \leq i \leq j \leq d} \lambda_{v_{ij}} (\bs{\lambda}_{\bs v} \otimes \bs{\lambda}_{\bs \mu})$ or $\bs{\bs{\lambda}_{\bs{\mu}} \otimes \dot{\bs{\eta}}}$. The term with $p=0$ is written as $(\nabla_{\bs{\psi}^{\otimes 4}}g^* / g^*) [O(|\bs{\lambda}||\bs{\lambda}_{\bs{\mu}^4}|) + O(|\bs{\lambda}||\bs{\lambda}_{\bs{\mu v}}|)]$ from a similar argument to bound $A_3$.

It remains to bound (\ref{l_expand_4}). For the first term in (\ref{l_expand_4}), observe that, for any sequence $a_{ijk\ell m}$,
\begin{align*}
& \sum_{i=1}^d\sum_{j=1}^d\sum_{k=1}^d \sum_{\ell=1}^d \sum_{m=1}^d a_{ijk\ell m} \lambda_{\mu_{i}} \lambda_{\mu_{j}} \lambda_{\mu_{k}} \lambda_{\mu_{\ell}} \lambda_{\mu_{m}} \\
& = \frac{1}{b(\alpha)} \sum_{m=1}^d \lambda_{\mu_{m}} \sum_{1 \leq i \leq j \leq k \leq \ell \leq d} a_{ijk\ell m} \left(b(\alpha) \sum_{(t_1,t_2,t_3,t_4) \in p(i,j,k,\ell)} \lambda_{\mu_{t_1}}\lambda_{\mu_{t_2}}\lambda_{\mu_{t_3}}\lambda_{\mu_{t_4}} \right.\\
& \left. \quad \quad + 12 \sum_{(t_1,t_2,t_3,t_4) \in p_{22}(i,j,k,\ell)} \lambda_{v_{t_1t_2}}\lambda_{v_{t_3t_4}} \right)
\\
& \quad -\frac{b(\alpha)}{12} \sum_{1 \leq i \leq j \leq d} \lambda_{v_{ij}} \sum_{1 \leq k \leq \ell \leq m \leq d}\ a_{ijk\ell m} \sum_{(t_1,t_2,t_3) \in p_{12}(k,\ell,m)} \lambda_{\mu_{t_1}}\lambda_{v_{t_2t_3}}.
\end{align*}
Therefore, the first term in (\ref{l_expand_4}) can be written as $(\nabla_{\bs{\psi}^{\otimes 5}}\overline{g}/g^*) [O(|\bs{\lambda}||12\bs{\lambda}_{\bs{v}^2}+ b(\alpha) \bs{\lambda}_{\bs{\mu}^4}|) + O(|\bs{\lambda}||\bs{\lambda}_{\bs{\mu v}}|)]$. The second term in (\ref{l_expand_4}) is written as $(\nabla_{\bs{\psi}^{\otimes 5}}\overline{g}/g^*) O(|\bs{\lambda}||\bs{t}(\bs{\psi},\alpha)|)$ from the same argument as (\ref{l_expand_3}), and the stated result follows.
\end{proof}
 
The following lemmas establish quadratic approximations of the log-likelihood function in the case of testing $H_0:M=1$ against $H_A:M=2$ in the homoscedastic case. Lemma \ref{P-quadratic-homo-1} considers the case $|\bs{\mu}_{1}- \bs{\mu}_{2}|\leq \zeta$, and Lemma \ref{P-quadratic-homo-2} considers the case $|\bs{\mu}_{1}- \bs{\mu}_{2}| \geq \zeta$. 
\begin{lemma} \label{P-quadratic-homo-1}
Suppose that Assumptions \ref{assn_consis} and \ref{A-taylor1} hold and $\bs{X}$ given $\bs{Z}$ has the density $f(\bs{x}|\bs{z}; \bs{\gamma},{\bs{\mu}},\bs{\Sigma})$ defined in (\ref{normal_density}). Let $L_n(\bs{\psi},\alpha): = \sum_{i = 1}^n \log g(\bs{X}_i|\bs{Z}_i;\bs{\psi},\alpha)$ with $g(\bs{x}|\bs{z};\bs{\psi},\alpha)$ defined in (\ref{loglike-homo}). For $\alpha \in (0,1)$, define $\bs{s}(\bs{x},\bs{z})$ and $\bs{t}(\bs{\psi},\alpha)$ as in (\ref{score_defn-homo}) and (\ref{tpsi_defn-homo}), and let $\mathcal{N}_{\varepsilon} := \{ \bs{\vartheta}_2 \in \Theta_{\bs{\vartheta}_2} : |\bs{t}(\bs{\psi},\alpha)|< \varepsilon\}$ and $\bs{\mathcal{I}}:=E[\bs{s}(\bs{X},\bs{Z})\bs{s}(\bs{X},\bs{Z})\t]$. Then, for any $\delta>0$  and $\zeta>0$, we have (a) $\sup_{\alpha \in [0,1]}\sup_{\bs{\vartheta}_2 \in A^1_{n \varepsilon}(\delta,\zeta) }|\bs{t}(\bs{\psi},\alpha)| = O_{p\varepsilon}(n^{-1/2})$;
\begin{equation*} 
(b)\ \sup_{\alpha \in [0,1]}\sup_{\bs{\vartheta}_2 \in A^1_{n \varepsilon}(\delta,\zeta)}\left|L_n(\bs{\psi},\alpha) - L_n(\bs{\psi}^*,\alpha) - \sqrt{n} \bs{t}(\bs{\psi},\alpha)\t \nu_n(\bs{s}(\bs{x},\bs{z})) + n \bs{t}(\bs{\psi},\alpha)\t \bs{\mathcal{I}} \bs{t}(\bs{\psi},\alpha)/2 \right| = o_{p\varepsilon}(1),
\end{equation*}
where $A^1_{n \varepsilon}(\delta,\zeta) := \{\bs{\vartheta}_2 \in \mathcal{N}_\varepsilon\cap \Theta_{{\bs{\vartheta}}_2,\zeta}^1: L_n(\bs{\psi},\alpha) - L_n(\bs{\psi}^*,\alpha) \geq -\delta \}$.
\end{lemma}

\begin{proof}[Proof of Lemma \ref{P-quadratic-homo-1}]

The proof is similar to that of Lemma \ref{P-quadratic} but using $g(\bs{x}|\bs{z};\bs{\psi},\alpha)$ defined in (\ref{loglike-homo}) in place of (\ref{loglike}).  We expand $g(\bs{x}|\bs{z};\bs{\psi},\alpha)/g(\bs{x}|\bs{z};\bs{\psi}^*,\alpha)-1$ five times with respect to $\bs{\psi}$ and show that the expansion satisfies Assumption \ref{assn_expansion}.

Observe that $\bs{t}(\bs{\psi},\alpha)$ defined in (\ref{tpsi_defn-homo}) satisfies $\bs{t}(\bs{\psi},\alpha) = 0$ if and only if $\bs{\psi}=\bs{\psi}^*$ because $\bs{\lambda}=0$  only if  the $(i,i,i)$th element of $\bs{\lambda}_{\bs{\mu}^3}$ or the $(i,i,i,i)$th element of $\bs{\lambda}_{\bs{\mu}^4}$ is 0  for all $1 \leq i \leq d$. 
Following the argument in the proof of Lemma \ref{P-quadratic} after (\ref{v_defn}), we may show that (\ref{uniform_lln})--(\ref{weak_cgce}) hold for $g(\cdot)$ defined in (\ref{loglike-homo}). Define $g^*$, $\nabla g^*$, and $\nabla\overline{g}$ as in the proof of Lemma \ref{P-quadratic} but using $g(\bs{x}|\bs{z};\bs{\psi},\alpha)$ defined in (\ref{loglike-homo}) in place of (\ref{loglike}). 
Expanding $\ell(\bs{y};\bs{\psi},\alpha):=g(\bs{x}|\bs{z};\bs{\psi}^*,\alpha)/g(\bs{x}|\bs{z};\bs{\psi},\alpha)$ five times around $\bs{\psi}^*$ while fixing $\alpha$ and using Lemma \ref{dv3-homo}  in \ref{section:auxiliary}, we can write $\ell(\bs{y};\bs{\psi},\alpha)-1$ as  
\begin{equation*}
\ell(\bs{y};\bs{\psi},\alpha) -1 =  s(\bs{y};\bs{\eta},\bs{\lambda})  +r(\bs{y};\bs{\eta},\bs{\lambda}), 
\end{equation*}
where
\[
s(\bs{y};\bs{\eta},\bs{\lambda}) := \frac{\nabla_{\bs{\eta}\t}g^* }{g^*}\dot{\bs{\eta}} + \frac{1}{3!} \frac{\nabla_{(\bs{\lambda}^{\otimes 3})\t} g^* }{g^*} \bs{\lambda}^{\otimes 3} + \frac{1}{4!} \frac{\nabla_{(\bs{\lambda}^{\otimes 4})\t} g^* }{g^*} \bs{\lambda}^{\otimes 4},
\]
with $\dot{\bs{\eta}}:=\bs{\eta}-\bs{\eta}^*$ and
\begin{align}
r(\bs{y};\bs{\eta},\bs{\lambda}) &: = \frac{1}{2!} \frac{\nabla_{({\bs{\eta}}^{\otimes 2})\t} g^* }{g^*}{\dot{\bs{\eta}}}^{\otimes 2}+ \frac{1}{3!} \frac{\nabla_{({\bs{\eta}}^{\otimes 3})\t} g^* }{g^*}{\dot{\bs{\eta}}}^{\otimes 3} \label{l_expand_2-homo} \\ 
& \quad + \sum_{p=0}^{3} \frac{1}{p!(4-p)!}\frac{\nabla_{({\bs{\eta}}^{\otimes (4-p)} \otimes \bs{\lambda}^{\otimes p} )\t} g^* }{g^*} (\dot{\bs{\eta}}^{\otimes (4-p)} \otimes \bs{\lambda}^{\otimes p} ) \label{l_expand_3-homo} \\
&\quad + \frac{1}{5!} \frac{\nabla_{(\bs{\lambda}^{\otimes 5})\t} \overline{g} }{g^*} \bs{\lambda}^{\otimes 5} + \sum_{p=0}^{4} \frac{1}{p!(5-p)!}\frac{\nabla_{({\bs{\eta}}^{\otimes (5-p)} \otimes \bs{\lambda}^{\otimes p} )\t} \overline{g} }{g^*} (\dot{\bs{\eta}}^{\otimes (5-p)} \otimes \bs{\lambda}^{\otimes p} ) \label{l_expand_4-homo}.
\end{align} 
We first show $s(\bs{y};\bs{\eta},\bs{\lambda}) = \bs{t}(\bs{\psi},\alpha)\t \bs{s}(\bs{x},\bs{z})$ with $\bs{s}(\bs{x},\bs{z})$ and $\bs{t}(\bs{\psi},\alpha)$ defined in (\ref{score_defn-homo}) and (\ref{tpsi_defn-homo}). The first term of $s(\bs{y};\bs{\eta},\bs{\lambda})$ is $\nabla_{(\bs{\gamma}\t,\bs{\mu}\t,\bs{v}\t)\t}f^*/f^*$. Using Lemma \ref{dv3-homo}, the second and third terms of $s(\bs{y};\bs{\eta},\bs{\lambda})$ are written as
\begin{align*}
& 
\frac{\alpha(1-\alpha) (1-2\alpha)}{3!} \sum_{1\leq i\leq j\leq k\leq d} \frac{\nabla_{\mu_i \mu_j \mu_k} f^* }{f^*} \sum_{(t_1,t_2,t_3)\in p(i,j,k)} \lambda_{t_1}\lambda_{t_2}\lambda_{t_3} \quad\text{and}\\ 
& 
\frac{ \alpha(1-\alpha)(1-6\alpha+6\alpha^2)}{4!} \sum_{1\leq i\leq j\leq k\leq \ell\leq d} \frac{\nabla_{\mu_i \mu_j \mu_k\mu_{\ell}} f^* }{f^*} \sum_{(t_1,t_2,t_3,t_4)\in p(i,j,k,\ell)} \lambda_{t_1}\lambda_{t_2}\lambda_{t_3}\lambda_{t_4},
\end{align*} 
where $\sum_{(t_1,t_2,t_3) \in p(i,j,k)}$ denotes the sum over all distinct permutations of $(i,j,k)$ to $(t_1,t_2,t_3)$ while $\sum_{(t_1,t_2,t_3,t_4) \in p(i,j,k,\ell)}$ denotes the sum over all distinct permutations of $(i,j,k,\ell)$ to $(t_1,t_2,t_3,t_4)$. Combining these results gives $s(\bs{y};\bs{\eta},\bs{\lambda}) = \bs{s}_{\bs{\eta}}\t \dot{\bs{\eta}} + \alpha(1-\alpha) (1-2\alpha) \bs{s}_{\bs{\mu^3}}\t \bs{\lambda}_{\bs{\mu^3}} + \alpha(1-\alpha)(1-6\alpha+6\alpha^2)\bs{s}_{\bs{\mu}^4}\t \bs{\lambda}_{\bs{\mu}^4}= \bs{t}(\bs{\psi},\alpha)\t \bs{s}(\bs{x},\bs{z})$, 
 where
$(\bs{s}_{\bs{\eta}},\bs{s}_{\bs{\mu^3}}, \bs{s}_{\bs{\mu}^4})$ satisfies Assumption \ref{assn_expansion}(a)(b)(e) from Assumption \ref{A-taylor1}, the property of the normal density, (\ref{uniform_lln}), and (\ref{weak_cgce}).

The stated result holds if  $\bs{r(\bs{y};\bs{\eta},\bs{\lambda})}$  can be written as $\bs{\xi}(\bs{y};\bs{\vartheta}_2) O(|\bs{\psi}||\bs{t}(\bs{\psi},\alpha)|)$ where $\sup_{\bs{\vartheta}_2 \in \mathcal{N}_\epsilon}|\bs{\xi}(\bs{y};\bs{\vartheta}_2)| \leq \sup_{\bs{\vartheta}_2 \in \mathcal{N}_\epsilon}|\bs{v}(\bs{y};\bs{\vartheta}_2)|$ with $\bs{v}(\bs{y};\bs{\vartheta}_2)$ defined in (\ref{v_defn}) but using $g(\bs{x}|\bs{z};\bs{\psi},\alpha)$ defined in (\ref{loglike-homo}), because then $r(\bs{y};\bs{\eta},\bs{\lambda})$ satisfies Assumption \ref{assn_expansion}(c)(d) from (\ref{uniform_lln}) and (\ref{weak_cgce}). First, the terms on the right hand side of (\ref{l_expand_2-homo}) are written as $(\nabla_{\bs{\eta}^{\otimes 2}} g^*/g^*) O(|\dot{\bs{\eta}}|^2)+(\nabla_{\bs{\eta}^{\otimes 3}} g^*/g^*) O(|\dot{\bs{\eta}}|^2)$. Second, the terms in (\ref{l_expand_3-homo}) are written as $(\nabla_{\bs{\psi}^{\otimes 4}} g^*/g^*) [O(|\dot{\bs{\eta}}||\bs{\lambda}|)+ O(|\dot{\bs{\eta}}|^2)]$. Finally, the terms in (\ref{l_expand_4-homo}) are written as 
$(\nabla_{\bs{\psi}^{\otimes 5}} \bar g /g^*) O(|\bs{\lambda_{\mu ^4}}||\bs{\lambda}|)+
(\nabla_{\bs{\psi}^{\otimes 5}} \bar g /g^*)[O(|\dot{\bs{\eta}}||\bs{\lambda}|)+ O(|\dot{\bs{\eta}}|^2)]$, and the stated result follows.
\end{proof}

\begin{lemma} \label{P-quadratic-homo-2}
Suppose that Assumptions \ref{assn_consis} and \ref{A-taylor1} hold  and $\bs{X}$ given $\bs{Z}$ has the density  $f(\bs{x}|\bs{z}; \bs{\gamma},{\bs{\mu}},\bs{\Sigma})$  defined in (\ref{normal_density}). Let $L_n(\bs{\phi},\bs{\lambda}):= \sum_{i=1}^n h(\bs{X}_i|\bs{Z}_i;\bs{\phi},\bs{\lambda})$ with $h(\bs{x}|\bs{z};\bs{\phi},\bs{\lambda})$ defined in (\ref{loglike-homo-2}). Define $\bs{s}(\bs{x},\bs{z};\bs{\lambda})$ and $\bs{t}(\bs{\phi},\bs{\lambda})$ as in (\ref{score_defn-homo-2}) and (\ref{tpsi_defn-homo-2}), and let $\mathcal{N}_{\varepsilon} := \{ \bs{\vartheta}_2 \in \Theta_{\bs{\vartheta}_2} : |\bs{t}(\bs{\phi},\bs{\lambda})|< \varepsilon\}$ and $\bs{\mathcal{I}}(\bs{\lambda}):=E[\bs{s}(\bs{X},\bs{Z};\bs{\lambda})\bs{s}(\bs{X},\bs{Z};\bs{\lambda})\t]$. Then, for any $|\bs{\lambda}|\geq \zeta> 0$ and any $\delta>0$, we have (a) $\sup_{\bs{\vartheta} \in A_{n\varepsilon}^2(\delta,\zeta)} |\bs{t}(\bs{\phi},\bs{\lambda}) | = O_{p\varepsilon}(n^{-1/2})$;
\begin{equation*} 
(b)\ \sup_{\bs{\vartheta} \in A_{n\varepsilon}^2(\delta,\zeta)}\left|L_n(\bs{\phi},\bs{\lambda}) - L_n(\bs{\phi^*},\bs{\lambda}) - \sqrt{n} \bs{t}(\bs{\phi},\bs\lambda) \t \nu_n(\bs{s}(\bs{x},\bs{z};\bs{\lambda})) + n \bs{t}(\bs{\phi},\bs\lambda) \t \bs{\mathcal{I}}(\bs{\lambda}) \bs{t}(\bs{\phi},\bs\lambda)/2 \right| = o_{p\varepsilon}(1),
\end{equation*}
where $A_{n\varepsilon}^2(\delta,\zeta) := \{\bs{\vartheta} \in \mathcal{N}_\varepsilon\cap \Theta_{{\bs{\vartheta}}_2,\zeta}^2: L_n(\bs{\phi},\bs{\lambda}) - L_n(\bs{\phi^*},\bs{\lambda}) \geq -\delta \}$.
\end{lemma}

\begin{proof}[Proof of Lemma \ref{P-quadratic-homo-2}]
The proof is similar to that of Lemma \ref{P-quadratic}. Observe that $\bs{t}(\bs{\phi},\bs{\lambda})$ defined in (\ref{tpsi_defn-homo-2}) satisfies $\bs{t}(\bs{\phi},\bs{\lambda}) = \bs{0}$ if and only if $\bs{\phi}=\bs{\phi}^* = ((\bs{\eta}^*)\t,0)\t$ because $|\bs{\lambda}| \geq \zeta$. Let $\bs{y}:=(\bs{x}\t,\bs{z}\t)\t$, and write $h(\bs{x}|\bs{z};\bs{\phi},\bs{\lambda})$ as $h(\bs{y};\bs{\phi},\bs{\lambda})$. Let $\ell(\bs{y},\bs{\phi},\bs{\lambda}) := h(\bs{y};\bs{\phi},\bs{\lambda})/h(\bs{y};\bs{\phi}^*,\bs{\lambda})$. Expanding $\ell(\bs{y};\bs{\phi},\bs{\lambda})-1$ twice around $\bs{\phi} = \bs{\phi}^*$ while fixing the value of $\bs{\lambda}$ and using $h(\bs{y};\bs{\phi}^*,\bs{\lambda}) = f_v^*$ and $\nabla_{\eta}h(\bs{y};\bs{\phi}^*,\bs{\lambda}) = \nabla_{\eta} f_v^*$ gives
\[
\ell(\bs{y};\bs{\phi},\bs{\lambda})-1 = s(\bs{y};\bs{\phi},\bs{\lambda}) + r(\bs{y};\bs{\phi},\bs{\lambda}),
\]
where
\[
s(\bs{y};\bs{\phi},\bs{\lambda}) := \frac{\nabla_{\alpha} h(\bs{y};\bs{\phi}^*,\bs{\lambda}) }{f_v^*} \alpha+ \frac{\nabla_{\bs{\eta}\t} f_v^*}{f_v^*}\dot{\bs{\eta}},
\]
with $\dot{\bs{\eta}}:=\bs{\eta}-\bs{\eta}^*$ and, for some $\overline{\bs{\phi}}\in (\bs{\phi},\bs{\phi}^*)$,
\begin{align*}
r(\bs{y};\bs{\phi},\bs{\lambda}) & = \frac{1}{2}\frac{\nabla_{\alpha^2} h(\bs{y};\overline{\bs{\phi}},\bs{\lambda})}{f_v^*}\alpha^2 + \frac{1}{2}\frac{\nabla_{(\bs{\eta}^{\otimes 2})\t} h(\bs{y};\overline{\bs{\phi}},\bs{\lambda})}{f_v^*} \dot{\bs{\eta}}^{\otimes 2}+ \frac{\nabla_{\alpha\bs{\eta}\t} h(\bs{y};\overline{\bs{\phi}},\bs{\lambda})}{f_v^*} \alpha\dot{\bs{\eta}}.
\end{align*} 
Let $f_v^*(\bs{\lambda}):= f_v(\bs{x}|\bs{z};\bs{\gamma}^*,\bs{\mu}^*+ \bs{\lambda}, \bs{v}^*)$, so that $f_v^*(\bs{0})=f_v^*$. Define
\begin{equation}\label{v_defn-homo-2}
\bs{v}(\bs{y};\bs{\vartheta}_2) :=\left(\bs{s}_{\bs{\phi}}(\bs{y};\bs{\phi},\bs{\lambda})\t ,\bs{s}_{\bs{\eta}}\t, s_{\alpha}(\bs{\lambda}) \right)\t,
\end{equation} 
where $\bs{s}_{\bs{\phi}}(\bs{y};\bs{\phi},\bs{\lambda}) : = (\nabla_{((\bs{\phi}^{\otimes 2})\t, (\bs{\phi}^{\otimes 3})\t)} s(\bs{y};\bs{\phi},\bs{\lambda}))\t / f_v^*$. In view of (\ref{score_defn-homo-2})--(\ref{tpsi_defn-homo-2}) and the argument in the proof of Lemma \ref{P-quadratic} , the stated result holds if
\begin{equation} \label{gr_condition}
\begin{aligned}
\text{(A)} \quad & \nabla_{\alpha} h(\bs{y};\bs{\phi}^*,\bs{\lambda}) = f_v^*(\bs{\lambda}) - f_v^* - \nabla_{ \bs{\mu}\t} f_v^* \bs{\lambda} - \nabla_{ \bs{v}\t} f_v^* \bs{\lambda}_{\bs{\mu}^2},\\
\text{(B)} \quad & r(\bs{y};\bs{\phi},\bs{\lambda}) = \bs{\xi}(\bs{y};\bs{\vartheta}_2) O(|\bs{\phi}-\bs{\phi}^*||\bs{t}(\bs{\phi},\bs{\lambda})|), \\
\text{(C)} \quad & \bs{v}(\bs{y};\bs{\vartheta}_2)\text{ satisfies (\ref{uniform_lln})--(\ref{weak_cgce})},
\end{aligned}
\end{equation}
where $\sup_{\bs{\vartheta}_2}|\bs{\xi}(\bs{y};\bs{\vartheta}_2)| \leq \sup_{\bs{\vartheta}_2}|\bs{v}(\bs{y};\bs{\vartheta}_2)|$, and the domain of $\bs{\vartheta}_2$ is such that $\bs{\phi} \in \Theta_{\bs{\eta}} \times [0, 3/4]$ and  $\bs{\lambda} \in \Theta_{\bs{\lambda}}$.  

We proceed to show (A)--(C) in (\ref{gr_condition}). (A) is shown in Lemma \ref{s_der_alpha}  in \ref{section:auxiliary}, and (B) follows from Lemma \ref{s_der_alpha} and the definition of $\bs{t}(\bs{\phi},\bs{\lambda})$. For (C), $\bs{s}_{\bs{\phi}}(\bs{y};\bs{\phi},\bs{\lambda})$ and $\bs{s}_{\bs{\eta}}$ clearly satisfy (\ref{uniform_lln})--(\ref{weak_cgce}). The proof completes if we show that $s_{\alpha}(\bs{\lambda})$ satisfies (\ref{uniform_lln})--(\ref{weak_cgce}). We first extend the domain of $s_{\alpha}(\bs{\lambda})$ so that it is well-defined when $|\bs{\lambda}|=0$. Then, $s_{\alpha}(\bs{\lambda})$ satisfies (\ref{uniform_lln})--(\ref{weak_cgce}) if $s_{\alpha}(\bs{\lambda})$ is Lipschitz continuous in $\bs{\lambda}$ in this extended domain and the Lipschitz constant is in $L^{2+\delta}$. Write $\bs{\lambda}$ in the $d$-spherical coordinates as $\bs{\lambda} = r \widetilde{\bs{\lambda}}(\bs{\theta})$, where $r$ is scalar with $r \geq 0$, $\bs{\theta} := (\theta_1,\ldots,\theta_{d-1})\t \in \Theta := [0,\pi)^{d-2}\times [0,2\pi)$, and $\widetilde{\bs{\lambda}}(\cdot)$ is a function from $\mathbb{R}^{d-1}$ to $\mathbb{R}^d$ whose elements are products of $\sin(\theta_j)$'s and $\cos(\theta_j)$'s such that $|\widetilde{\bs{\lambda}}(\bs{\theta})|=1$ (e.g., $\widetilde{\bs{\lambda}}(\theta) = (\cos(\theta),\sin(\theta))\t$ when $d=2$). For $r>0$ and $\bs{\theta} \in \Theta$, define $s_\alpha(r,\bs{\theta}) := s_\alpha(r \widetilde{\bs{\lambda}}(\bs{\theta}))$, and write $s_\alpha(r,\bs{\theta})$ as 
\begin{equation*} 
{s}_{\alpha}(r,\bs{\theta}) := 
\frac{f_v^*(r \widetilde{\bs{\lambda}}(\bs{\theta})) -f_v^*(\bs{0})-\nabla_{\bs{\mu}\t} f^*_v(\bs{0}) r \widetilde{\bs{\lambda}}(\bs{\theta}) - \nabla_{(\bs{\mu}^{\otimes 2})\t} f^*_v(\bs{0}) r^2 \widetilde{\bs{\lambda}}(\bs{\theta})^{\otimes 2}/2 }{ r^3 f^*_v(\bs{0})},
\end{equation*}
where $f_v^*(\bs{\lambda}):=f_v\left(\bs{x}|\bs{z};\bs{\gamma}^*, \bs\mu^*+ \bs{\lambda}, \bs{v}^*\right)$. 
Define $s_\alpha(0,\bs{\theta}) = \nabla_{(\bs{\mu}^{\otimes 3})\t} f^*_v(\bs{0}) \widetilde{\bs{\lambda}}(\bs{\theta})^{\otimes 3}/(3!f^*_v(\bs{0}))$, then ${s}_{\alpha}(r,\bs{\theta})$ converges to $s_\alpha(0,\bs{\theta})$ as $r \to 0$, and $s_\alpha(r,\bs{\theta})$ is continuous in $r \geq 0$ and $\bs{\theta} \in \Theta$. 
 
We show that $s_\alpha(r,\bs{\theta})$ is Lipschitz continuous in $(r,\bs{\theta})\in [0,M]\times \Theta$. Let $\bs{\Lambda}(\bs{\theta}):=\nabla_{\bs{\theta}\t}\widetilde{\bs{\lambda}}(\bs{\theta})$ denote the $d \times (d-1)$ Jacobian matrix of $\widetilde{\bs{\lambda}}(\bs{\theta})$. It follows from a direct calculation that $|\nabla_r {s}_{\alpha}(r,\bs{\theta})| \leq \mathcal{C} \sup_{\bs{\lambda} }|\nabla_{(\bs{\mu}^{\otimes 4})\t} f^*_v(\bs{\lambda})|| \bs{\lambda}^{\otimes 4}|/f^*_v(\bs{0})$ and $|\nabla_{\bs{\theta}} {s}_{\alpha}(r,\bs{\theta})| \leq C \sup_{\bs{\theta}}|\bs{\Lambda}(\bs{\theta})|\sup_{\bs{\lambda} }|\nabla_{(\bs{\mu}^{\otimes 3})\t} f^*_v(\bs{\lambda})|| \bs{\lambda}^{\otimes 2}|/f^*_v(\bs{0})$, which are in $L^{2+\delta}$ from $\sup_{\bs{\theta}}|\bs{\Lambda}(\bs{\theta})| < \infty$ and the property of the normal density. Consequently, $s_\alpha(r,\bs{\theta})$ is Lipschitz continuous in $(r,\bs{\theta})\in [0,M]\times \Theta$ and the Lipschitz constant is in $L^{2+\delta}$. Therefore, (C) of (\ref{gr_condition}) holds, and the stated result is proven.  
\end{proof}

\section{Quadratic expansion under singular Fisher information matrix}\label{section:expansion}

This appendix derives a Le Cam's differentiable in quadratic mean (DQM)-type expansion  that is useful for proving Lemmas \ref{P-quadratic}--\ref{P-quadratic-homo-2} in \ref{section:quadratic}.  \citet{liushao03as} develop a DQM expansion under the loss of identifiability in terms of the generalized score function. Lemmas \ref{Ln_thm1} and \ref{Ln_thm2} modify \citet{liushao03as} to fit our context of parametric models where the derivatives of the density of different orders are linearly dependent. \citet{kasaharashimotsu18markov} derive a similar expansion that accommodates dependent and heterogeneous $Y_i$'s under additional assumptions than ours. Lemmas \ref{Ln_thm1} and \ref{Ln_thm2} may be viewed as a specialization of \citet{kasaharashimotsu18markov} to the random sampling case.

Let $\bs{\vartheta}$ be a parameter vector, and let  $g(\bs{y};\bs{\theta})$ denote the density of $\bs{Y}$. Let $L_n(\bs{\vartheta}) := \sum_{i=1}^n \log g(\bs{Y}_i;\bs{\vartheta})$ denote the log-likelihood function. Split $\bs{\vartheta}$ as $\bs{\vartheta} = (\bs{\psi}\t,\bs{\pi}\t)\t$, and write $L_n(\bs{\vartheta}) = L_n(\bs{\psi},\bs{\pi})$. $\bs{\pi}$ corresponds to the part of $\bs{\vartheta}$ that is not identified under the null. Denote the true parameter value of $\bs{\psi}$ by $\bs{\psi}^*$, and denote the set of $(\bs{\psi},\bs{\pi})$ corresponding to the null hypothesis by $\Gamma^*= \{(\bs{\psi},\bs{\pi})\in\Theta: \bs{\psi}=\bs{\psi}^*\}$. 
Let $\bs{t}(\bs{\vartheta})$ be a continuous function of $\bs{\vartheta}$ such that $\bs{t}(\bs{\vartheta})=0$ if and only if $\bs{\psi}=\bs{\psi}^*$. For $\varepsilon>0$, define a neighborhood of $\Gamma^*$ by
\[
\mathcal{N}_{\varepsilon} := \{ \bs{\vartheta} \in \Theta: |\bs{t}(\bs{\vartheta})|< \varepsilon\}.
\]

We establish a general quadratic expansion that expresses $L_n(\bs{\psi},\bs{\pi})-L_n(\bs{\psi}^*,\bs{\pi})$ as a quadratic function of $\bs{t}(\bs{\vartheta})$ for $\bs{\vartheta}\in \mathcal{N}_{\varepsilon}$. Denote the density ratio by
\begin{equation} \label{density_ratio}
\ell(\bs{y};\bs{\vartheta}) := \frac{g (\bs{y}; \bs{\psi},\bs{\pi})}{ g (\bs{y}; \bs{\psi}^*,\bs{\pi})},
\end{equation} 
so that $L_n(\bs{\psi},\bs{\pi}) - L_n(\bs{\psi}^*,\bs{\pi}) = \sum_{i=1}^n\log \ell(\bs{Y}_i;\bs{\vartheta})$. We assume that $\ell(\bs{y};\bs{\vartheta})$ can be expanded around $\ell(\bs{y};\bs{\vartheta}^*)=1$ as follows. 
\begin{assumption}\label{assn_expansion}
$\ell(\bs{y};\bs{\vartheta}) -1$ admits an expansion
\begin{equation*} 
\ell(\bs{y};\bs{\vartheta}) -1 = \bs{t}(\bs{\vartheta})\t \bs{s}(\bs{y};\bs{\pi}) + r(\bs{y};\bs{\vartheta}),
\end{equation*}
where $\bs{s}(\bs{y};\bs{\pi})$ and $r(\bs{y};\bs{\vartheta})$ satisfy, for some $C \in (0,\infty)$ and $\varepsilon >0$, (a) $E\sup_{\bs{\pi} \in \Theta_{\bs{\pi}}} \left|\bs{s}(\bs{Y};\bs{\pi})\right|^2 < C$, (b) $\sup_{\bs{\pi} \in \Theta_{\bs{\pi}}}| P_n(\bs{s}(\bs{y};\bs{\pi})\bs{s}(\bs{y};\bs{\pi})\t) - \bs{\mathcal{I}}_{\bs{\pi}}| = o_p(1)$ with $0<\inf_{\bs{\pi}\in \Theta_{\bs{\pi}}} \lambda_{\min}(\bs{\mathcal{I}}_{\bs{\pi}}) \leq \sup_{\bs{\pi}\in \Theta_{\bs{\pi}} } \lambda_{\max}(\bs{\mathcal{I}}_{\bs{\pi}})<C$, (c) $E[\sup_{\bs{\vartheta} \in \mathcal{N}_\varepsilon} |r(\bs{Y};\bs{\vartheta})/(|\bs{t}(\bs{\vartheta})||\bs{\psi}-\bs{\psi}^*|)|^2] < \infty$, (d) $\sup_{\bs{\vartheta} \in \mathcal{N}_\varepsilon} [ \nu_n(r(\bs{y};\bs{\vartheta}))/(|\bs{t}(\bs{\vartheta})||\bs{\psi}-\bs{\psi}^*|)] = O_p(1)$, (e) $\sup_{\bs{\pi} \in \Theta_{\bs{\pi}}}|\nu_n(\bs{s}(\bs{y};\bs{\pi}))|=O_p(1)$. 
\end{assumption} 
We first establish an expansion $L_n(\bs{\psi},\bs{\pi})$ in a neighborhood $\mathcal{N}_{c/\sqrt{n}}$ that holds for any $c>0$. 
\begin{lemma} \label{Ln_thm1}
Suppose that Assumption \ref{assn_expansion}(a)--(d) holds. Then, for all $c>0$,
\begin{equation*} 
\sup_{\bs{\vartheta} \in \mathcal{N}_{c/\sqrt{n}}} \left| L_n(\bs{\psi},\bs{\pi}) - L_n(\bs{\psi}^*,\bs{\pi}) - \sqrt{n}\bs{t}(\bs{\vartheta})\t \nu_n (\bs{s}(\bs{y};\bs{\pi})) + n \bs{t}(\bs{\vartheta})\t \bs{\mathcal{I}}_{\bs{\pi}} \bs{t}(\bs{\vartheta})/2 \right| = o_p(1).
\end{equation*}
\end{lemma}
\begin{proof}[Proof of Lemma \ref{Ln_thm1}] 
Define $h(\bs{y},\bs{\vartheta}) := \sqrt{\ell(\bs{y},\bs{\vartheta})}-1$. By using the Taylor expansion of $2 \log(1+x) = 2x - x^2(1+o(1))$ for small $x$, we have, uniformly for $\bs{\vartheta} \in \mathcal{N}_{c/\sqrt{n}}$,
\begin{equation} \label{Ln_expand}
L_n(\bs{\psi},\bs{\pi}) - L_n(\bs{\psi}^*,\bs{\pi}) = 2 \sum_{i=1}^n \log(1+h(\bs{Y}_i,\bs{\vartheta})) = n P_n(2h(\bs{y},\bs{\vartheta}) - [1+o_p(1)] h(\bs{y},\bs{\vartheta})^2).
\end{equation}
The stated result holds if we show
\begin{align} 
& \sup_{\bs{\vartheta} \in \mathcal{N}_{c/\sqrt{n}}} \left| nP_n(h(\bs{y},\bs{\vartheta})^2) - n \bs{t}(\bs{\vartheta})\t \bs{\mathcal{I}}_{\bs{\pi}} \bs{t}(\bs{\vartheta})/4 \right| = o_p(1), \label{hk_appn} \\
& \sup_{\bs{\vartheta} \in \mathcal{N}_{c/\sqrt{n}}} \left| nP_n(h(\bs{y},\bs{\vartheta})) - \sqrt{n} \bs{t}(\bs{\vartheta})\t \nu_n(\bs{s}(\bs{y};\bs{\pi}) )/2 + n\bs{t}(\bs{\vartheta})\t \bs{\mathcal{I}}_{\bs{\pi}} \bs{t}(\bs{\vartheta})/8 \right| = o_p(1), \label{hk_appn2}
\end{align}
because then the right hand side of (\ref{Ln_expand}) equals $\sqrt{n} \bs{t}(\bs{\vartheta})\t \nu_n(\bs{s}(\bs{y};\bs{\pi}) ) - n\bs{t}(\bs{\vartheta})\t \bs{\mathcal{I}}_{\bs{\pi}} \bs{t}(\bs{\vartheta})/2$ uniformly in $\bs{\vartheta} \in \mathcal{N}_{c/\sqrt{n}}$.

To show (\ref{hk_appn}), write $4P_n ( h(\bs{y},\bs{\vartheta})^2)$ as
\begin{equation} \label{B0}
4P_n ( h(\bs{y},\bs{\vartheta})^2) = P_n \left(\frac{4(\ell(\bs{y};\bs{\vartheta})-1)^2}{(\sqrt{\ell(\bs{y};\bs{\vartheta})} + 1)^2}\right) = P_n(\ell(\bs{y},\bs{\vartheta})-1)^2 - P_n \left( (\ell(\bs{y};\bs{\vartheta})-1)^3 \frac{(\sqrt{\ell(\bs{y};\bs{\vartheta})}+3)}{(\sqrt{\ell(\bs{y};\bs{\vartheta})} + 1)^3}\right).
\end{equation}
It follows from Assumption \ref{assn_expansion}(a)(b)(c) and $(E|XY|)^2 \leq E|X|^2E|Y|^2$ that, uniformly for $\bs{\vartheta} \in \mathcal{N}_\varepsilon$,
\begin{align} \label{lk_lln}
P_n(\ell(\bs{y};\bs{\vartheta})-1)^2 & = \bs{t}(\bs{\vartheta})\t P_n(\bs{s}(\bs{y};\bs{\pi})\bs{s}(\bs{y};\bs{\pi})\t)\bs{t}(\bs{\vartheta}) + 2 \bs{t}(\bs{\vartheta})\t P_n[\bs{s}(\bs{y};\bs{\pi}) r(\bs{y};\bs{\vartheta})] + P_n(r(\bs{y};\bs{\vartheta}))^2 \nonumber \\
& = (1+o_p(1))\bs{t}(\bs{\vartheta})\t\bs{\mathcal{I}}_{\bs{\pi}} \bs{t}(\bs{\vartheta}) +O_p(|\bs{t}(\bs{\vartheta})|^2|\bs{\psi}-\bs{\psi}^*|). 
\end{align} 
Therefore, the second term on the right of (\ref{B0}) is $\bs{t}(\bs{\vartheta})\t\bs{\mathcal{I}}_{\bs{\pi}} \bs{t}(\bs{\vartheta})$ + $o_p(n^{-1})$. Note that, if $X_1,\ldots,X_n$ are random variables with $\max_{1\leq i \leq n}\mathbb{E}|X_i|^q < C$ for some $q>0$ and $C < \infty$, then we have $\max_{1 \leq i \leq n} |X_i|= o_p(n^{1/q})$. Therefore, from Assumption \ref{assn_expansion}(a)(c), we have
\begin{equation*} 
\max_{1\leq i \leq n} \sup_{\bs{\vartheta} \in \mathcal{N}_{c/\sqrt{n}}} |\ell(\bs{Y}_i,\bs{\vartheta})-1| = \max_{1\leq i \leq n} \sup_{\bs{\vartheta} \in \mathcal{N}_{c/\sqrt{n}}} |\bs{t}(\bs{\vartheta})\t\bs{s}(\bs{Y}_i;\bs{\pi}) + r(\bs{Y}_i;\bs{\vartheta})| = o_p(1) . 
\end{equation*}
Consequently, the third term on the right of (\ref{B0}) is $o_p(n^{-1})$, and (\ref{hk_appn}) follows.

We proceed to show (\ref{hk_appn2}). Consider the following expansion of $h(\bs{y},\bs{\vartheta})$: 
\begin{equation} \label{hk_1}
h(\bs{y},\bs{\vartheta}) = (\ell(\bs{y};\bs{\vartheta})-1)/2 - h(\bs{y},\bs{\vartheta})^2/2 = (\bs{t}(\bs{\vartheta})\t \bs{s}(\bs{y};\bs{\pi}) + r(\bs{y};\bs{\vartheta}))/2 - h(\bs{y},\bs{\vartheta})^2/2 .
\end{equation}
Then, (\ref{hk_appn2}) follows from (\ref{hk_1}), (\ref{hk_appn}), and Assumption \ref{assn_expansion}(d), and the stated result follows.
\end{proof}

The next lemma expands $L_n(\bs{\psi},\bs{\pi})$ in $A_{n\varepsilon}(\delta) := \{\bs{\vartheta} \in \mathcal{N}_\varepsilon: L_n(\bs{\psi},\bs{\pi}) - L_n(\bs{\psi}^*,\bs{\pi}) \geq -\delta \}$ for $\delta \in (0,\infty)$. This lemma is useful for deriving the asymptotic distribution of the LRTS because a consistent MLE is in $A_{n\varepsilon}(\delta)$ by definition.  Define $O_{p\varepsilon}(\cdot)$ and $o_{p\varepsilon}(\cdot)$ as in \ref{section:quadratic}.   
\begin{lemma} \label{Ln_thm2}
Suppose that Assumption \ref{assn_expansion} holds. Then, for any $\delta>0$, (a) $\sup_{\vartheta \in A_{n\varepsilon}(\delta)} |\bs{t}(\bs{\vartheta})| = O_{p\varepsilon}(n^{-1/2})$; 
\begin{equation*} 
(b)\ \sup_{\bs{\vartheta} \in A_{n\varepsilon}(\delta)}\left|L_n(\bs{\psi},\bs{\pi}) - L_n(\bs{\psi}^*,\bs{\pi}) - \sqrt{n} \bs{t}(\bs{\vartheta})\t \nu_n (\bs{s}(\bs{y};\bs{\pi})) + n \bs{t}(\bs{\vartheta})\t \bs{\mathcal{I}}_{\bs{\pi}} \bs{t}(\bs{\vartheta})/2 \right| = o_{p\varepsilon}(1).
\end{equation*}
\end{lemma}
\begin{proof}[Proof of Lemma \ref{Ln_thm2}]
For part (a), applying the inequality $\log(1+x) \leq x$ to the log-likelihood ratio function and using (\ref{hk_1}) give
\begin{equation} \label{Ln_ineq}
L_n(\bs{\psi},\bs{\pi}) - L_n(\bs{\psi}^*,\bs{\pi}) = 2 \sum_{i=1}^n \log(1+h(\bs{Y}_i,\bs{\vartheta})) \leq 2 n P_n(h(\bs{y},\bs{\vartheta})) = \sqrt{n} \nu_n(\ell(\bs{y};\bs{\vartheta})-1) - n P_n(h(\bs{y},\bs{\vartheta})^2). 
\end{equation}
We derive a lower bound on $P_n (h(\bs{y},\bs{\vartheta})^2)$. Observe that $h(\bs{y},\bs{\vartheta})^2= {(\ell(\bs{y};\bs{\vartheta})-1)^2}/(\sqrt{\ell(\bs{y};\bs{\vartheta})} + 1)^2\geq \mathbb{I}{\{\ell(\bs{y};\bs{\vartheta}) \leq \kappa\}(\ell(\bs{y};\bs{\vartheta})-1)^2}/(\sqrt{\kappa} + 1)^2$ for any $\kappa>0$. Therefore,
\begin{align*}
P_n (h(\bs{y},\bs{\vartheta})^2) & \geq (\sqrt{\kappa}+1)^{-2}P_n \left( \mathbb{I}{\{\ell(\bs{y};\bs{\vartheta}) \leq \kappa\}} (\ell(\bs{y};\bs{\vartheta})-1)^2\right) \\
& \geq (\sqrt{\kappa}+1)^{-2} \left[ P_n ((\ell(\bs{y};\bs{\vartheta})-1)^2) - P_n \left( \mathbb{I}{\{\ell(\bs{y};\bs{\vartheta}) > \kappa\}} (\ell(\bs{y};\bs{\vartheta})-1)^2\right) \right] .
\end{align*} 
Let $B:=\sup_{\bs{\vartheta} \in \mathcal{N}_\varepsilon}|\ell(\bs{y};\bs{\vartheta})-1|$. From Assumption \ref{assn_expansion}(a)(c), we have $EB^2 < \infty$, and hence $\lim_{\kappa \rightarrow \infty} \sup_{\bs{\vartheta} \in \mathcal{N}_\varepsilon} P_n \left( \mathbb{I}{\{\ell(\bs{y};\bs{\vartheta}) > \kappa\}} (\ell(\bs{y};\bs{\vartheta})-1)^2\right) \leq \lim_{\kappa\rightarrow \infty} P_n \left( \mathbb{I}\{B+1 > \kappa\}B^2\right) =0$ almost surely. Let $\tau=(\sqrt{\kappa}+1)^{-2}/2$. By choosing $\kappa$ sufficiently large, it follows from (\ref{lk_lln}) and Assumption \ref{assn_expansion}(e) that, uniformly for $\bs{\vartheta} \in \mathcal{N}_\varepsilon$,
\begin{equation} \label{rk_lower}
P_n (h(\bs{y},\bs{\vartheta})^2) \geq \tau (1+o_p(1))\bs{t}(\bs{\vartheta})\t \bs{\mathcal{I}}_{\bs{\pi}} \bs{t}(\bs{\vartheta}) +O_p(|\bs{t}(\bs{\vartheta})|^2|\bs{\psi}-\bs{\psi}^*|) .
\end{equation}

Because $\sqrt{n} \nu_n(\ell(\bs{y};\bs{\vartheta})-1) = \sqrt{n} \bs{t}(\bs{\vartheta})\t [\nu_n(\bs{s}(\bs{y};\bs{\pi})) + O_p(1)]$ from Assumption \ref{assn_expansion}(d), it follows from (\ref{Ln_ineq}) and (\ref{rk_lower}) that
\begin{align} 
-\delta& \leq L_n(\bs{\psi},\bs{\pi}) - L_n(\bs{\psi}^*,\bs{\pi})\nonumber\\
& \leq \sqrt{n} \bs{t}(\bs{\vartheta})\t [\nu_n(\bs{s}(\bs{y};\bs{\pi})) + O_p(1)] - \tau (1+o_p(1)) n \bs{t}(\bs{\vartheta})\t \bs{\mathcal{I}}_{\bs{\pi}} \bs{t}(\bs{\vartheta}) + O_p(n|\bs{t}(\bs{\vartheta})|^2|\bs{\psi}-\bs{\psi}^*|) .\label{rk_lower2}
\end{align} 
Let $\bs{T}_{n}:= \bs{\mathcal{I}_{\pi}}^{1/2}\sqrt{n} \bs{t}(\bs{\vartheta})$. From (\ref{rk_lower2}), Assumption \ref{assn_expansion}(c)(e), and the fact $\bs{\psi}-\bs{\psi}^* \to 0$ if $\bs{t}(\bs{\vartheta}) \to 0$, we obtain the following result: for any $\Delta>0$, there exist $\varepsilon>0$ and $M,n_0<\infty$ such that
\begin{equation*}
\Pr\left( \inf_{\bs{\vartheta} \in \mathcal{N}_{\varepsilon}} \left( |\bs{T}_n| M - (\tau/2) |\bs{T}_n|^2 + M \right) \geq 0\right) \geq 1-\Delta,\ \text{ for all }\ n > n_0.
\end{equation*}
Rearranging the terms inside $\Pr(\cdot)$ gives $ \sup_{\bs{\vartheta} \in \mathcal{N}_{\varepsilon}} (|\bs{T}_{n}|-(M/\tau))^2 \leq 2M/\tau+(M/\tau)^2$, and part (a) follows. Part (b) follows from part (a) and Lemma \ref{Ln_thm1}.
\end{proof}

\section{Auxiliary results and their proofs}\label{section:auxiliary}

\begin{lemma} \label{mv_derivative}
Let $f_v(\bs{x};\bs{\mu},\bs{v}) := (2\pi)^{-d/2}(\det \bs{S}(\bs{v}))^{-1/2} \exp(-(\bs{x}-\bs{\mu})\t\bs{S}(\bs{v})^{-1}(\bs{x}-\bs{\mu})/2)$ denote the density of a $d$-variate normal distribution with mean $\bs{\mu}= (\mu_1,\ldots,\mu_d)\t$ and variance $\bs{S}(\bs{v})$ with $\bs{v} = \{ v_{ij} \}_{1\leq i \leq j \leq d}$ as specified in (\ref{fv_defn}). Then, the following holds for any $t_1,t_2,t_3,t_4,t_5,t_6 \in \{1,\ldots,d\}$:
\begin{align*}
\frac{\partial f_v(\bs{x};\bs{\mu},\bs{v})}{\partial v_{t_1t_2}} &= \frac{1}{2}\frac{\partial^2 f_v(\bs{x};\bs{\mu},\bs{v})}{\partial \mu_{t_1} \partial \mu_{t_2}}, \quad
\frac{\partial^2 f_v(\bs{x};\bs{\mu},\bs{v})}{\partial v_{t_1t_2}\partial v_{t_3t_4}} = \frac{1}{4}\frac{\partial^4 f_v(\bs{x};\bs{\mu},\bs{v})}{\partial \mu_{t_1} \partial \mu_{t_2} \partial \mu_{t_3} \partial \mu_{t_4}}, \\
\frac{\partial^3 f_v(\bs{x};\bs{\mu},\bs{v})}{\partial v_{t_1t_2}\partial v_{t_3t_4}\partial v_{t_5t_6}} & = \frac{1}{8}\frac{\partial^6 f_v(\bs{x};\bs{\mu},\bs{v})}{\partial \mu_{t_1} \partial \mu_{t_2} \partial \mu_{t_3} \partial \mu_{t_4} \partial \mu_{t_5} \partial \mu_{t_6}} .
\end{align*}
\end{lemma}

\begin{proof}
Henceforth, we suppress $(\bs{x};\bs{\mu},\bs{\Sigma})$ and $(\bs{x};\bs{\mu},\bs{v})$ and from $f(\bs{x};\bs{\mu},\bs{\Sigma})$ and $f_v(\bs{x};\bs{\mu},\bs{v})$ unless confusions might arise. In view of the definition of $\bs{S}(\bs{v})$ in (\ref{fv_defn}), the following holds for any function $g(\bs{\Sigma})$ of $\bs{\Sigma}$:
\begin{equation} \label{del_v_sigma}
\frac{\partial g(\bs{S}(\bs{v}))}{\partial v_{t_1t_2}} = \frac{\partial g(\bs{\Sigma}) / \partial \Sigma_{t_1t_2} + \partial g(\bs{\Sigma}) / \partial \Sigma_{t_2t_1}}{2} = \frac{\partial g(\bs{\Sigma})}{\partial \Sigma_{t_1t_2}} .
\end{equation}
Let $\bs{s}_i$ denote the $i$th column of $\bSig^{-1}$, and let $s_{ij}$ denote the $(i,j)$th element of $\bSig^{-1}$. A direct calculation gives $\partial^2 f(\bs{x};\bs{\mu},\bs{\Sigma}) / \partial \bs{\mu}\partial \bs{\mu}\t = - \bs{\Sigma}^{-1} f + \bs{\Sigma}^{-1}(\bs{x}-\bs{\mu}) (\bs{x}-\bs{\mu})\t \bs{\Sigma}^{-1} f$ and $\partial f(\bs{x};\bs{\mu},\bs{\Sigma}) / \partial \bSig = - (1/2) \bs{\Sigma}^{-1} f + (1/2)\bs{\Sigma}^{-1}(\bs{x}-\bs{\mu}) (\bs{x}-\bs{\mu})\t \bs{\Sigma}^{-1} f$. Therefore, the first result follows immediately from (\ref{del_v_sigma}).

To prove the second result, we first derive ${\partial^4 f(\bs{x};\bs{\mu},\bs{\Sigma})}/{\partial \mu_{t_1} \partial \mu_{t_2} \partial \mu_{t_3} \partial \mu_{t_4}}$. Noting that ${\partial \bs{s}_{j}\t (\bs{x}-\bs{\mu})}/{\partial \mu_{i}} = -s_{ji}$ and ${\partial f(\bs{x};\bs{\mu},\bs{\Sigma})}/{\partial \mu_{i}} = \bs{s}_{i}\t (\bs{x}-\bs{\mu})f$ and differentiating ${\partial^2 f(\bs{x};\bs{\mu},\bs{\Sigma})}/{\partial \mu_{t_1} \partial \mu_{t_2}} = [ - s_{t_1t_2} + \bs{s}_{t_1}\t (\bs{x}-\bs{\mu}) \bs{s}_{t_2}\t (\bs{x}-\bs{\mu}) ] f$ with respect to $\mu_{t_3}$ and $\mu_{t_4}$, we obtain
\begin{equation} \label{f_del_mu_4}
\begin{aligned}
& \frac{\partial^4 f(\bs{x};\bs{\mu},\bs{\Sigma})}{\partial \mu_{t_1} \partial \mu_{t_2} \partial \mu_{t_3} \partial \mu_{t_4}} = \left( \sum_{\{i,j\},\{k,\ell\}} s_{t_it_j} s_{t_kt_\ell} - \sum_{\{i,j\},\{k\}, \{\ell\}} s_{t_it_j} \bs{s}_{t_k}\t (\bs{x}-\bs{\mu}) \bs{s}_{t_\ell}\t (\bs{x}-\bs{\mu}) + \prod_{i=1}^4\bs{s}_{t_i}\t (\bs{x}-\bs{\mu}) \right) f ,
\end{aligned}
\end{equation}
where $\sum_{\{i,j\},\{k,\ell\}}$ denotes the sum over all 3 possible partitions of $\{1,2,3,4\}$ into $\{\{i,j\},\{k,\ell\}\}$, and $\sum_{\{i,j\},\{k\}, \{\ell\}}$ denotes the sum over all 6 possible partitions of $\{1,2,3,4\}$ into three sets $\{\{i,j\},\{k\},\{\ell\}\}$. Recall that
\begin{equation} \label{f_del_st_1t_2}
\frac{\partial f(\bs{x};\bs{\mu},\bs{\Sigma}) }{ \partial \Sigma_{t_1t_2}} = (1/2)[ - s_{t_1t_2} + \bs{s}_{t_1}\t (\bs{x}-\bs{\mu}) \bs{s}_{t_2}\t (\bs{x}-\bs{\mu}) ] f.
\end{equation}
Let ${\bs 1}_{i}$ denote a $d\times 1$ vector whose elements are 0 except for the $i$th element, which is 1. We then have $s_{t_1t_2} = {\bs 1}_{t_1}\t \bSig^{-1} {\bs 1}_{t_2} $ and $\bs{s}_{t_1} =\bSig^{-1} {\bs 1}_{t_1} $. Using the symmetry of $\bSig$, we obtain
\begin{align*}
\frac{\partial s_{t_1t_2}}{\partial \Sigma_{t_3t_4}} & = \frac{\partial ( s_{t_1t_2}+ s_{t_2t_1})/2}{\partial \Sigma_{t_3t_4}} \\
& = \frac{\partial }{\partial \Sigma_{t_3t_4}} \left({\bs 1}_{t_1}\t \bSig^{-1} {\bs 1}_{t_2} + {\bs 1}_{t_2}\t \bSig^{-1} {\bs 1}_{t_1} \right)/2 \\
& = - (1/2) \left(\bSig^{-1} {\bs 1}_{t_1} {\bs 1}_{t_2}\t \bSig^{-1} + \bSig^{-1} {\bs 1}_{t_2} {\bs 1}_{t_1}\t \bSig^{-1} \right)_{t_3t_4} \\
&= - (1/2) \left(s_{t_3t_1} s_{t_2t_4} + s_{t_3t_2} s_{t_1t_4} \right),
\end{align*}
and
\begin{align*}
\frac{\partial \bs{s}_{t_1}\t(\bs{x}-\bs{\mu})}{\partial \Sigma_{t_3t_4}} & = \frac{\partial }{\partial \Sigma_{t_3t_4}} \left({\bs 1}_{t_1}\t \bSig^{-1} (\bs{x}-\bs{\mu}) +(\bs{x}-\bs{\mu})\t \bSig^{-1} {\bs 1}_{t_1} \right)/2 \\
& = - (1/2) \left(\bSig^{-1} {\bs 1}_{t_1} (\bs{x}-\bs{\mu})\t \bSig^{-1} + \bSig^{-1} (\bs{x}-\bs{\mu}) {\bs 1}_{t_1}\t \bSig^{-1} \right)_{t_3t_4} \\
&= - (1/2) \left(s_{t_3t_1} (\bs{x}-\bs{\mu})\t \bs{s}_{t_4} + \bs{s}_{t_3}\t (\bs{x}-\bs{\mu}) s_{t_1t_4} \right).
\end{align*}
Therefore, taking the derivative of the right hand side of (\ref{f_del_st_1t_2}) with respect to $\Sigma_{t_3t_4}$ gives
\begin{equation} \label{f_del_sigma_2}
\begin{aligned}
\frac{\partial^2 f(\bs{x};\bs{\mu},\bs{\Sigma}) }{ \partial \Sigma_{t_1t_2} \partial \Sigma_{t_3t_4}} & = \frac{1}{4} \left[ s_{t_3t_1} s_{t_2t_4} + s_{t_3t_2} s_{t_1t_4} - \left(s_{t_3t_1} (\bs{x}-\bs{\mu})\t \bs{s}_{t_4} + \bs{s}_{t_3}\t (\bs{x}-\bs{\mu}) s_{t_1t_4} \right) \bs{s}_{t_2}\t (\bs{x}-\bs{\mu}) \right. \\
& \quad \left. - \bs{s}_{t_1}\t (\bs{x}-\bs{\mu})\left(s_{t_3t_2} (\bs{x}-\bs{\mu})\t \bs{s}_{t_4} + \bs{s}_{t_3}\t (\bs{x}-\bs{\mu}) s_{t_2t_4} \right) \right] f \\
& \quad + \frac{1}{2} \left( - s_{t_1t_2} + \bs{s}_{t_1}\t (\bs{x}-\bs{\mu}) \bs{s}_{t_2}\t (\bs{x}-\bs{\mu}) \right) \frac{\partial f(\bs{x};\bs{\mu},\bs{\Sigma}) }{ \partial \Sigma_{t_3t_4}} \\
& = \frac{1}{4}\left( \sum_{\{i,j\},\{k,\ell\}} s_{t_it_j} s_{t_kt_\ell} - \sum_{\{i,j\},\{k\}, \{\ell\}} s_{t_it_j} \bs{s}_{t_k}\t (\bs{x}-\bs{\mu}) \bs{s}_{t_\ell}\t (\bs{x}-\bs{\mu}) + \prod_{i=1}^4\bs{s}_{t_i}\t (\bs{x}-\bs{\mu}) \right) f.
\end{aligned}
\end{equation}
Comparing this with (\ref{f_del_mu_4}) and using (\ref{del_v_sigma}) gives the second result.

For the third result, differentiating (\ref{f_del_mu_4}) with respect to $\mu_{t_5}$ and $\mu_{t_6}$ gives
\begin{equation} \label{f_del_mu_6}
\begin{aligned}
& \frac{\partial^6 f(\bs{x};\bs{\mu},\bs{\Sigma})}{\partial \mu_{t_1} \partial \mu_{t_2} \partial \mu_{t_3} \partial \mu_{t_4} \partial \mu_{t_5} \partial \mu_{t_6}} \\
& = \left( - \sum_{\{i,j\},\{k,\ell\},\{m,n\}} s_{t_it_j} s_{t_kt_\ell}s_{t_m t_n} + \sum_{\{i,j\},\{k,\ell\}, \{m\},\{n\}} s_{t_it_j} s_{t_kt_\ell} \bs{s}_{t_m}\t (\bs{x}-\bs{\mu}) \bs{s}_{t_n}\t (\bs{x}-\bs{\mu}) \right. \\
& \quad \left. - \sum_{\{i,j\},\{k,\ell,m,n\}} s_{t_it_j} \bs{s}_{t_k}\t (\bs{x}-\bs{\mu}) \bs{s}_{t_\ell}\t (\bs{x}-\bs{\mu}) \bs{s}_{t_m}\t (\bs{x}-\bs{\mu}) \bs{s}_{t_n}\t (\bs{x}-\bs{\mu}) + \prod_{i=1}^6\bs{s}_{t_i}\t (\bs{x}-\bs{\mu}) \right) f ,
\end{aligned}
\end{equation}
where $\sum_{\{i,j\},\{k,\ell\},\{m,n\}}$ denotes the sum over all 15 possible partitions of $\{1,2,3,4,5,6\}$ into $\{\{i,j\},\{k,\ell\},\{m,n\}\}$, $\sum_{\{i,j\},\{k,\ell\}, \{m\},\{n\}}$ denotes the sum over all 45 possible partitions of $\{1,2,3,4,5,6\}$ into three sets $\{\{i,j\},\{k,\ell\},\{m\},\{n\}\}$, and $\sum_{\{i,j\},\{k,\ell,m,n\}}$ denotes the sum over all 15 possible partitions of $\{1,2,3,4,5,6\}$ into $\{\{i,j\},\{k,\ell,m,n\}\}$. Differentiating (\ref{f_del_sigma_2}) with respect to $\Sigma_{t_5t_6}$ gives (\ref{f_del_mu_6}) divided by 8, and the third result follows.
\end{proof}

\begin{lemma} \label{dv3}
Suppose that $g(\bs{x}|\bs{z};\bs{\psi},\alpha)$ is given by (\ref{loglike}), where $\bs{\psi} = (\bs{\eta}^{\top},\bs{\lambda}_{\bs\mu}\t,\bs{\lambda}_{\bs v}\t)^{\top}$ and $\bs{\eta} = (\bs{\gamma}^{\top},{\bs \nu}_{\bs \mu}\t,\bs{\nu}_{\bs{v}}\t)^{\top}$. Let $g^*$ and $\nabla g^*$ denote $g(\bs{x}|\bs{z}; \bs{\psi},\alpha)$ and $\nabla g(\bs{x}|\bs{z}; \bs{\psi},\alpha)$ evaluated at $(\bs{\psi}^*,\alpha)$, respectively. Let $\nabla f^*$ denote $\nabla f(\bs{x}|\bs{z};\bs{\gamma}^*,\bs{\mu}^*,\bs{\Sigma}^*)$. Then, with $b(\alpha): = -(2/3) (\alpha^2 - \alpha + 1)<0$, 
\begin{align*}
&(a)\ \text{for}\ k = 1, 2, 3\ \text{and}\ \ell = 0, 1,\ldots,\ \nabla_{\bs{\lambda}_{\bs\mu}^{\otimes k} \otimes \bs{\eta}^{\otimes \ell} } g^* = \bs{0}; \\
&(b)\ \nabla_{\lambda_{\mu_i}\lambda_{\mu_j}\lambda_{\mu_k}\lambda_{\mu_\ell}} g^* = \alpha(1 - \alpha)b(\alpha) \nabla_{\mu_i\mu_j\mu_k\mu_\ell} f^*; \\
&(c)\ \text{for}\ \ell = 0, 1,\ldots,\ \nabla_{\bs{\lambda}_{\bs v} \otimes \bs{\eta}^{\otimes\ell} } g^* = \bs{0}; \\
&(d)\ \nabla_{\lambda_{\mu_{i}} \lambda_{v_{jk}}} g^* = \alpha(1 - \alpha)\nabla_{\mu_i\mu_j\mu_k} f^*; \\
&(e)\ \nabla_{\lambda_{v_{ij}} \lambda_{v_{k\ell}}} g^* = \alpha(1 - \alpha)\nabla_{\mu_i\mu_j\mu_k\mu_\ell} f^*.
\end{align*}
\end{lemma}

\begin{proof}
We prove part (a) for $\ell = 0$ first. Suppress all arguments in $g(\bs{x}|\bs{z}; \bs{\psi}, \alpha)$ and $f_v(\bs{x}|\bs{z}; \bs{\gamma}, \bs{\mu}, \bs{v})$ except for $\bs{\lambda}_{\bs\mu}$, and rewrite (\ref{loglike}) as follows:
\begin{equation} \label{g_function}
g(\bs{\lambda}_{\bs\mu}) = \alpha f_v((1 - \alpha) \bs{\lambda}_{\bs\mu}, (1 - \alpha)C_1 \bs{w}(\bs{\lambda}_{\bs\mu}\bs{\lambda}_{\bs\mu}^{\top}) ) + (1 - \alpha) f_v(-\alpha \bs{\lambda}_{\bs\mu}, - \alpha C_2 \bs{w}(\bs{\lambda}_{\bs\mu}\bs{\lambda}_{\bs\mu}^{\top}) ) . 
\end{equation}
 For a composite function $h(\bs{a}, {\bs r}(\bs{a}))$ of a $d\times 1$ vector ${\bs a}=(a_1,\ldots,a_d)^{\top}$, the following result holds:
\begin{equation}\label{comp1}
\begin{aligned}
\nabla_{a_{i_1} \cdots a_{i_k}} h({\bs a}, {\bs r}({\bs a})) 
& = \{ (\nabla_{a_{i_1}} + \nabla_{u_{i_1}}) \cdots (\nabla_{a_{i_k}} + \nabla_{u_{i_k}}) \} h({\bs a}, {\bs r}({\bs u}))|_{{\bs u} = {\bs a}} \\
& = \sum_{j = 0}^k \sum_{p(j,\{i_1,\ldots,i_k\})} \nabla_{ u_{t_1} \cdots u_{t_j} a_{t_{j+1}} \cdots a_{t_k} } h({\bs a}, {\bs r}({\bs u}))|_{{\bs u} = {\bs a}} , 
\end{aligned}
\end{equation}
where $\sum_{p(j,\{i_1,\ldots,i_k\})}$ denotes the sum over all the partitions of $\{i_1,\ldots,i_k\}$ into two sets $\{t_1,\ldots,t_j\}$ and $\{t_{j+1},\ldots,t_k\}$. Applying (\ref{comp1}) to the right hand side of (\ref{g_function})  with $\bs{a} = \bs{\lambda}_{\bs \mu}$ gives the derivatives of $g(\bs{\lambda}_{\bs\mu})$. First, we derive $\nabla_{ u_{t_1} \cdots u_{t_j} } f_v((1 - \alpha) \bs{\lambda}_{\bs\mu}, (1 - \alpha)C_1 \bs{w}(\bs{u}\bs{u}^{\top}) )|_{\bs{u} = \bs{0}}$. Let $\tilde c:=(1-\alpha)C_1$. For notational convenience, if $i >j$, define $\nabla_{v_{ij}}h(\bs{v}):=\nabla_{v_{ji}}h(\bs{v})$ for any function $h(\bs{v})$. Using the fact $\nabla_{u_k} w_{ij}(\bs{u}\bs{u}\t) = 2 u_i \mathbb{I}\{j=k\} + 2 u_j\mathbb{I}\{i=k\}$  if $i<j$ and $\nabla_{u_k} w_{ii}(\bs{u}\bs{u}\t) = 2 u_i \mathbb{I}\{i=k\}$, we obtain 
\begin{align*}
\nabla_{ u_{t_1} } f_v( \cdot, \tilde c \bs{w}(\bs{u}\bs{u}^{\top}) ) &= \sum_{i=1}^d \sum_{j=i}^d \nabla_{v_{ij}} f_v( \cdot, \tilde c \bs{w}(\bs{u}\bs{u}^{\top}) ) \tilde c \nabla_{u_{t_1}} w_{ij}(\bs{u}\bs{u}\t) \\
&= 2\sum_{i=1}^{t_1} \nabla_{v_{i t_1}} f_v( \cdot, \tilde c \bs{w}(\bs{u}\bs{u}^{\top}) ) \tilde c u_i + 2\sum_{j=t_1+1}^d \nabla_{v_{t_1 j}} f_v( \cdot, \tilde c \bs{w}(\bs{u}\bs{u}^{\top}) ) \tilde c u_j \\
&= 2\sum_{i=1}^d \nabla_{v_{t_1 i}} f_v( \cdot, \tilde c \bs{w}(\bs{u}\bs{u}^{\top}) ) \tilde c u_i .
\end{align*}
Differentiating the right hand side with respect to $u_{t_2}$ give
\begin{align*}
\nabla_{ u_{t_1} u_{t_2} } f_v( \cdot, \tilde c \bs{w}(\bs{u}\bs{u}^{\top}) ) & = 
4 \sum_{i=1}^d \sum_{j=1}^d \nabla_{v_{t_1 i} v_{t_2 j}} f_v( \cdot, \tilde c \bs{w}(\bs{u}\bs{u}^{\top}) ) \tilde c^2 u_i u_j
+ 2 \nabla_{v_{t_1 t_2}} f_v( \cdot, \tilde c \bs{w}(\bs{u}\bs{u}^{\top}) ) \tilde c .
\end{align*}
Differentiating the right hand side with respect to $u_{t_3}$ gives
\begin{align*}
\nabla_{ u_{t_1} u_{t_2} u_{t_3} } f_v( \cdot, \tilde c \bs{w}(\bs{u}\bs{u}^{\top}) ) & = 
8 \sum_{i=1}^d \sum_{j=1}^d \sum_{k=1}^d \nabla_{v_{t_1 i} v_{t_2 j} v_{t_3 k} } f_v( \cdot, \tilde c \bs{w}(\bs{u}\bs{u}^{\top}) ) \tilde c^3 u_i u_j u_k \\
& \quad + 4 \sum_{i=1}^d\nabla_{v_{t_1 i} v_{t_2 t_3}} f_v( \cdot, \tilde c \bs{w}(\bs{u}\bs{u}^{\top}) ) \tilde c^2 u_i + 4 \sum_{j=1}^d\nabla_{v_{t_1 t_3} v_{t_2 j}} f_v( \cdot, \tilde c \bs{w}(\bs{u}\bs{u}^{\top}) ) \tilde c^2 u_j \\
& \quad + 4 \sum_{k=1}^d\nabla_{v_{t_1 t_2} v_{t_3 k}} f_v( \cdot, \tilde c \bs{w}(\bs{u}\bs{u}^{\top}) ) \tilde c^2 u_k .
\end{align*}
Finally, evaluating these derivatives at $\bs{u}=\bs{0}$ and differentiating $\nabla_{ u_{t_1} u_{t_2} u_{t_3} } f_v( \cdot, \tilde c \bs{w}(\bs{u}\bs{u}^{\top}) )$ with respect to $u_{t_4}$ and evaluating at $\bs{u}=\bs{0}$ gives
\begin{equation} \label{comp2}
\begin{aligned}
\nabla_{ u_{t_1} } f_v( \cdot, \tilde c \bs{w}(\bs{u}\bs{u}^{\top}) ) |_{\bs{u}=\bs{0}} & = 0,\\
\nabla_{ u_{t_1} u_{t_2}} f_v( \cdot, \tilde c \bs{w}(\bs{u}\bs{u}^{\top}) ) |_{\bs{u}=\bs{0}} & = 2 \tilde c \nabla_{v_{t_1 t_2}} f_v( \cdot, \tilde c \bs{w}(\bs{u}\bs{u}^{\top}) ) ,\\
\nabla_{ u_{t_1} u_{t_2} u_{t_3} } f_v( \cdot, \tilde c \bs{w}(\bs{u}\bs{u}^{\top}) ) |_{\bs{u}=\bs{0}} & = 0 ,\\
\nabla_{ u_{t_1} u_{t_2} u_{t_3} u_{t_4}} f_v( \cdot, \tilde c \bs{w}(\bs{u}\bs{u}^{\top}) ) |_{\bs{u}=\bs{0}} & = 
4 \tilde c^2 \nabla_{v_{t_1 t_4} v_{t_2 t_3}} f_v( \cdot, \bs{0}) + 4 \tilde c^2\nabla_{v_{t_1 t_3} v_{t_2 t_4}} f_v( \cdot, \bs{0}) \\
& \quad + 4 \tilde c^2 \nabla_{v_{t_1 t_2} v_{t_3 t_4}} f_v( \cdot, \bs{0}) ,
\end{aligned}
\end{equation}
and a similar result holds for $\nabla_{u_{t_1} \cdots u_{t_j} \bs{\lambda}_{\bs{\mu}}^{k-j} } f_v((1 - \alpha) \bs{\lambda}_{\bs\mu}, (1 - \alpha)C_1 \bs{w}(\bs{u}\bs{u}^{\top}) )|_{\bs{u} = \bs{0}}$ and \\$ \nabla_{ u_{t_1} \cdots u_{t_j} \bs{\lambda}_{\bs{\mu}}^{k-j}} f (- \alpha \bs{\lambda}_{\bs\mu}, - \alpha C_2 \bs{w}(\bs{u}\bs{u}^{\top}) )|_{\bs u = \bs{0}}$.

With (\ref{comp1}) and (\ref{comp2}) at hand, we are ready to derive part (a).
Differentiating (\ref{g_function}) with respect to $\bs{\lambda}_{\bs{\mu}}$ and using (\ref{comp1}), (\ref{comp2}), $C_1-C_2 = -1$, $3((1-\alpha)C_1+\alpha C_2) = 2\alpha - 1$, and Lemma \ref{mv_derivative}, we obtain
\begin{align*}
\nabla_{\bs{\lambda}_{\bs\mu}}g(\bs{0}) & = \bs{0}, \\
\nabla_{\lambda_{\mu_i} \lambda_{\mu_j} } g(\bs{0}) & = \alpha(1 - \alpha) \nabla_{\mu_i \mu_j} f_v(\bs{0}, \bs{0}) + 2\alpha(1 - \alpha)(C_1 - C_2) \nabla_{v_{ij}} f_v(\bs{0}, \bs{0}) = 0, \\
\nabla_{\lambda_{\mu_i}\lambda_{\mu_j}\lambda_{\mu_k}}g(\bs{0}) & = \alpha(1 - \alpha)(1 - 2\alpha) \nabla_{\mu_i \mu_j \mu_k} f_v(\bs{0}, \bs{0}) \\
& \quad + 3\alpha(1 - \alpha) ((1 - \alpha)C_1 + \alpha C_2) 2 \nabla_{\mu_i v_{jk}} f_v(\bs{0}, \bs{0}) = 0,
\end{align*}
and part (a) for $\ell = 0$ follows. Repeating the same argument with $\nabla_{\bs{\eta}^{\otimes \ell}}g(\bs{\lambda}_{\bs\mu},\bs{\eta})$ gives part (a) for $\ell \geq 1$.

For part (b), differentiating (\ref{g_function}) and using (\ref{comp1}), (\ref{comp2}), and Lemma \ref{mv_derivative} gives
\begin{align*}
& \nabla_{\lambda_{\mu_i}\lambda_{\mu_j}\lambda_{\mu_k}\lambda_{\mu_\ell}} g(\bs{0}) \\
& = \alpha(1 - \alpha)[(1 - \alpha)^3 + \alpha^3] \nabla_{\mu_i \mu_j \mu_k \mu_\ell} f_v(\bs{0}, \bs{0}) + 6 \alpha(1 - \alpha)((1 - \alpha)^2C_1 - \alpha^2C_2) 2 \nabla_{\mu_i \mu_j v_{k\ell}} f_v(\bs{0}, \bs{0}) \nonumber \\
& \quad + 12 \alpha(1 - \alpha)((1 - \alpha)C_1^2 + \alpha C_2^2) \nabla_{v_{ij}v_{k\ell}} f_v(\bs{0}, \bs{0}) \nonumber \\
&= \alpha(1 - \alpha) [-(2/3) (\alpha^2 - \alpha + 1)]\nabla_{\mu_i \mu_j \mu_k \mu_\ell} f_v(\bs{0}, \bs{0}),
\end{align*}
and the stated result follows because $\nabla_{\mu_i \mu_j \mu_k \mu_\ell} f_v(\bs{0}, \bs{0})=\nabla_{\mu_i \mu_j \mu_k \mu_\ell} f(\bs{0}, \bs{0})$. Part (c) follows from a direct calculation.

Parts (d) and (e) follow from direct calculation and using (\ref{comp1}), (\ref{comp2}) and Lemma \ref{mv_derivative}.
\end{proof}

\begin{lemma} \label{dv3-homo}
Suppose that $g(\bs{x}|\bs{z};\bs{\psi},\alpha)$ is given by (\ref{loglike-homo}), where $\bs{\psi} = (\bs{\eta}^{\top},\bs{\lambda})^{\top}$ and $\bs{\eta} = (\bs{\gamma}^{\top},{\bs \nu}_{\bs \mu},\bs{\nu}_{\bs{v}})^{\top}$. Let $g^*$ and $\nabla g^*$ denote $g(\bs{x}|\bs{z}; \bs{\psi},\alpha)$ and $\nabla g(\bs{x}|\bs{z}; \bs{\psi},\alpha)$ evaluated at $(\bs{\psi}^*,\alpha)$, respectively. Let $\nabla f^*$ denote $\nabla f(\bs{x}|\bs{z};\bs{\gamma}^*,\bs{\mu}^*,\bs{\Sigma}^*)$. Then, 
\begin{align*}
&(a)\ \text{for}\ k = 1, 2\ \text{and}\ \ell = 0, 1,\ldots,\ \nabla_{\bs{\lambda}^{\otimes k} \otimes \bs{\eta}^{\otimes \ell} } g^* = \bs{0}; \\
&(b)\ \nabla_{\lambda_{i}\lambda_{j}\lambda_{k}} g^* = \alpha(1 - \alpha)(1-2\alpha) \nabla_{\mu_i\mu_j\mu_k} f^*; \\
&(b)\ \nabla_{\lambda_{i}\lambda_{j}\lambda_{k}\lambda_{\ell}} g^* = \alpha(1 - \alpha)(1-6\alpha+6\alpha^2) \nabla_{\mu_i\mu_j\mu_k\mu_\ell} f^*.
\end{align*}
\end{lemma} 

\begin{proof}
We prove part (a) for $\ell = 0$ first. Suppress all arguments in $g(\bs{x}|\bs{z}; \bs{\psi}, \alpha)$ and $f_v(\bs{x}|\bs{z}; \bs{\gamma}, \bs{\mu}, \bs{v})$ except for $\bs{\lambda}$, and rewrite as follows:
\begin{equation} \label{g_function-homo}
g(\bs{\lambda}_{\bs\mu}) = \alpha f_v((1 - \alpha) \bs{\lambda}_{\bs\mu}, -\alpha(1 - \alpha) \bs{w}(\bs{\lambda}_{\bs\mu}\bs{\lambda}_{\bs\mu}^{\top}) ) + (1 - \alpha) f_v(-\alpha \bs{\lambda}_{\bs\mu}, -\alpha(1 - \alpha) \bs{w}(\bs{\lambda}_{\bs\mu}\bs{\lambda}_{\bs\mu}^{\top}) ) . 
\end{equation}
Differentiating (\ref{g_function-homo}) with respect to $\bs{\lambda}$ and using (\ref{comp1}), (\ref{comp2}), and Lemma \ref{mv_derivative}, we obtain
\begin{align*}
\nabla_{\bs{\lambda}_{\bs\mu}}g(\bs{0}) & = \bs{0}, \\
\nabla_{\lambda_{i} \lambda_{j} } g(\bs{0}) & = \alpha(1 - \alpha) \nabla_{\mu_i \mu_j} f_v(\bs{0}, \bs{0}) - 2\alpha(1 - \alpha) \nabla_{v_{ij}} f_v(\bs{0}, \bs{0}) = 0, \end{align*}
and part (a) for $\ell = 0$ follows. Repeating the same argument with $\nabla_{\bs{\eta}^{\otimes \ell}}g(\bs{\lambda},\bs{\eta})$ gives part (a) for $\ell \geq 1$.

For parts (b) and (c), differentiating (\ref{g_function-homo}) and using (\ref{comp1}), (\ref{comp2}), and Lemma \ref{mv_derivative} gives
\begin{align*} 
& \nabla_{\lambda_{i}\lambda_{j}\lambda_{k}}g(\bs{0}) \\
&= \alpha(1 - \alpha)(1 - 2\alpha) \nabla_{\mu_i \mu_j \mu_k} f_v(\bs{0}, \bs{0}) + 3\alpha(1 - \alpha) (-\alpha(1 - \alpha) + \alpha (1-\alpha)) 2 \nabla_{\mu_i v_{jk}} f_v(\bs{0}, \bs{0}) \\
&= \alpha(1 - \alpha)(1 - 2\alpha) \nabla_{\mu_i \mu_j \mu_k} f_v(\bs{0}, \bs{0}) ,\\
& \nabla_{\lambda_{i}\lambda_{j}\lambda_{k}\lambda_{\ell}} g(\bs{0}) \\
& = \alpha(1 - \alpha)[(1 - \alpha)^3 + \alpha^3] \nabla_{\mu_i \mu_j \mu_k \mu_\ell} f_v(\bs{0}, \bs{0}) + 6 \alpha(1 - \alpha)(-\alpha(1 - \alpha)^2 - \alpha^2(1-\alpha)) 2 \nabla_{\mu_i \mu_j v_{k\ell}} f_v(\bs{0}, \bs{0}) \nonumber \\
& \quad + 12 \alpha(1 - \alpha)((1 - \alpha)\alpha^2 + \alpha (1-\alpha)^2) \nabla_{v_{ij}v_{k\ell}} f_v(\bs{0}, \bs{0}) \nonumber \\
&= \alpha(1 - \alpha) (1-6\alpha+6\alpha^2)\nabla_{\mu_i \mu_j \mu_k \mu_\ell} f_v(\bs{0}, \bs{0}),
\end{align*}
and the stated result follows because $\nabla_{\mu_i \mu_j \mu_k \mu_\ell} f_v(\bs{0}, \bs{0})=\nabla_{\mu_i \mu_j \mu_k \mu_\ell} f(\bs{0}, \bs{0})$. 
\end{proof}

\begin{lemma} \label{lemma_lambda_e}
Suppose $\bs{\lambda} = (\bs{\lambda}_{\bs{\mu}}\t,\bs{\lambda}_{\bs{v}}\t)\t \in \Theta_{\bs{\lambda}}$ satisfies $\bs{t}_{\bs{\lambda}}(\bs{\lambda},\alpha) =O_p(n^{-1/2})$ for some $\alpha \in (0,1)$ with $\bs{t}_{\bs{\lambda}}(\bs{\lambda},\alpha)$ defined in (\ref{tpsi_defn}). Then, if $|\lambda_{\mu_i}| \geq n^{-1/8} (\log n)^{-1}$ for some $i \in \{1,\ldots,d\}$, we have $\bs{\lambda}_{\bs{v}} = O_p(n^{-3/8}(\log n)^3)$.
\end{lemma} 
 
\begin{proof}
The stated result holds if we show, for all $(j,k)$,
\[
\begin{aligned}
(A)\ {\lambda}_{{v}_{ii}} &= O_p(n^{-3/8} \log n), \quad (B)\ {\lambda}_{{v}_{ij}} = O_p(n^{-3/8} (\log n)^2), \\
(C)\ {\lambda}_{{v}_{jj}} &= O_p(n^{-3/8} (\log n)^3), \quad (D)\ {\lambda}_{{v}_{jk}} = O_p(n^{-3/8} (\log n)^3).
\end{aligned}
\]
Observe that $\bs{t}_{\bs{\lambda}}(\bs{\lambda},\alpha) =O_p(n^{-1/2})$ implies that, for any $(i,j,k)$, 
\begin{align} 
{\lambda}_{{\mu}_{i}} {\lambda}_{{v}_{ii}} & = O_p(n^{-1/2}), \label{lambda_e_bound_1}\\
({\lambda}_{{\mu}_{i}} {\lambda}_{{v}_{ij}} + {\lambda}_{{\mu}_{j}}{\lambda}_{{v}_{ii}}) & = O_p(n^{-1/2}), \label{lambda_e_bound_2} \\
({\lambda}_{{\mu}_{i}} {\lambda}_{{v}_{jk}} + {\lambda}_{{\mu}_{j}}{\lambda}_{{v}_{ik}} + {\lambda}_{{\mu}_{k}}{\lambda}_{{v}_{ij}}) & = O_p(n^{-1/2}), \label{lambda_e_bound_3} \\
[12 ({\lambda}_{{v}_{ii}})^2 + b(\alpha) ({\lambda}_{{\mu}_{i}})^4] & = O_p(n^{-1/2}) . \label{lambda_e_bound_4} 
\end{align} 
Part (A) follows from $|\lambda_{\mu_i}| \geq n^{-1/8} (\log n)^{-1}$ and (\ref{lambda_e_bound_1}). Before deriving part (B), we first show that $\lambda_{\mu_j} = O_p(n^{-1/8})$ holds for any $j$. Consider the two cases, $|\lambda_{\mu_j}| \leq n^{-1/8} (\log n)^{-1}$ and $|\lambda_{\mu_j}| \geq n^{-1/8} (\log n)^{-1}$. When $|\lambda_{\mu_j}| \leq n^{-1/8} (\log n)^{-1}$, the result $\lambda_{\mu_j} = O_p(n^{-1/8})$ follows immediately. When $|\lambda_{\mu_j}| \geq n^{-1/8} (\log n)^{-1}$, (\ref{lambda_e_bound_1}) implies ${\lambda}_{{v}_{jj}} = O_p(n^{-3/8} \log n)$, and in view of (\ref{lambda_e_bound_4}) we obtain $\lambda_{\mu_j} = O_p(n^{-1/8})$. Therefore, $\lambda_{\mu_j} = O_p(n^{-1/8})$ holds for any $j$. Combining this with (\ref{lambda_e_bound_2}) and part (A), we obtain ${\lambda}_{{\mu}_{i}} {\lambda}_{{v}_{ij}} = O_p(n^{-1/2}\log n)$. Hence, noting that $|\lambda_{\mu_i}| \geq n^{-1/8} (\log n)^{-1}$ gives part (B).

For part (C), reversing the role of $i$ and $j$ in (\ref{lambda_e_bound_2}) gives ${\lambda}_{{\mu}_{j}} {\lambda}_{{v}_{ij}} + {\lambda}_{{\mu}_{i}}{\lambda}_{{v}_{jj}} = O_p(n^{-1/2})$. In conjunction with $\lambda_{\mu_j} = O_p(n^{-1/8})$ and part (B), we obtain ${\lambda}_{{\mu}_{i}}{\lambda}_{{v}_{jj}} = O_p(n^{-1/2}(\log n)^2)$. Then, part (C) follows from $|\lambda_{\mu_i}| \geq n^{-1/8} (\log n)^{-1}$.

For part (D), we already show that $\lambda_{\mu_j}, \lambda_{\mu_k} = O_p(n^{-1/8})$, and part (B) implies ${\lambda}_{{v}_{ij}}, {\lambda}_{{v}_{ik}}= O_p(n^{-3/8} (\log n)^2)$. Substituting this to (\ref{lambda_e_bound_3}) gives ${\lambda}_{{\mu}_{i}} {\lambda}_{{v}_{jk}} = O_p(n^{-1/2} (\log n)^2)$, and part (D) follows from $|\lambda_{\mu_i}| \geq n^{-1/8} (\log n)^{-1}$.
\end{proof}

\begin{lemma} \label{tau_update}
Suppose that Assumptions\ref{assn_consis} and \ref{A-vec-2} hold. If $\bs{\vartheta}_{m_0+1}^{m(k)}(\tau_0) - \bs{\vartheta}_{m_0+1}^{m*}(\tau_0)= o_p(1)$  and $\tau^{(k)} - \tau_0 = o_p(1)$, then (a) $\alpha_m^{(k+1)}/[\alpha_m^{(k+1)}+\alpha_{m+1}^{(k+1)}] - \tau_0 = o_p(1)$  and (b) $\tau^{(k+1)} - \tau_0 = o_p(1)$.
\end{lemma}
\begin{proof}
We suppress $(\tau_0)$ from $\bs{\vartheta}_{M_0+1}^{m(k)}(\tau_0)$ and $\bs{\vartheta}_{M_0+1}^{m*}(\tau_0)$. The proof is similar to the proof of Lemma 3 of \citet{lichen10jasa}. We suppress $\bs{Z}$ for brevity. Let $f_i( \bs{\mu}, \bs{\Sigma})$ and $f_i(\bs{\vartheta}_{M_0+1})$ denote $f(\bs{X}_i; \bs{\mu}, \bs{\Sigma})$ and $f_{M_0+1}(\bs{X}_i; \bs{\vartheta}_{M_0+1})$, respectively. Applying a Taylor expansion to $\alpha_m^{(k+1)}= n^{-1}\sum_{i = 1}^n w_{im}^{(k)}$ and using $\bs{\vartheta}_{M_0+1}^{m(k)} - \bs{\vartheta}_{M_0+1}^{m*} = o_p(1)$, we obtain
\[
\begin{aligned}
\alpha_m^{(k+1)}& = \frac{1}{n} \sum_{i = 1}^n \frac{\tau^{(k)}(\alpha_{m}^{(k)}+\alpha_{m+1}^{(k)})f_i(\bs{\mu}_m^{(k)},\bs{\Sigma}_m^{(k)})}{f_i(\bs{\vartheta}_{M_0+1}^{m(k)})} \\
& = \frac{1}{n} \sum_{i = 1}^n \frac{\tau_0 \alpha_{m}^*f_i(\bs{\mu}_{m}^{*},\bs{\Sigma}_{m}^{*})}{f_i(\bs{\vartheta}_{M_0+1}^{m*})} + o_p(1)
= \tau_0 \alpha_{m}^*+ o_p(1),
\end{aligned}
\]
where  the last equality follows from $E[f_i(\bs{\mu}_{m}^{*},\bs{\Sigma}_{m}^{*})/f_i(\bs{\vartheta}_{M_0+1}^{m*}) ] = 1$ and the law of large numbers. A similar argument gives $\alpha_{m+1}^{(k+1)} = (1 - \tau_0)\alpha_{m}^* + o_p(1)$, and part (a) follows.

For part (b), define $H(\tau) := \sum_{i=1}^n w_{im}^{(k)} \log(\tau) + \sum_{i=1}^n w_{i,m+1}^{(k)} \log(1-\tau) = n \alpha_m^{(k+1)} \log(\tau) + n \alpha_{m+1}^{(k+1)} \log(1-\tau) $, then $\tau^{(k+1)}$ maximizes $H(\tau) + p(\tau)$. $H(\tau)$ is maximized at $\tilde \tau = \alpha_m^{(k+1)}/[\alpha_m^{(k+1)}+\alpha_{m+1}^{(k+1)}] = \tau_0 + o_p(1)$. Expanding $H(\tau)$ twice around $\tilde \tau$ gives $H(\tilde\tau) - H(\tau) \geq (\epsilon+o_p(1)) \alpha_m^* n (\tau - \tilde \tau)^2$ for some $\epsilon>0$. In conjunction with $H(\tau^{(k+1)}) + p(\tau^{(k+1)}) - H(\tilde \tau) - p(\tilde \tau) \geq 0$, we obtain $p(\tau^{(k+1)}) - p(\tilde \tau) \geq (\epsilon+o_p(1)) \alpha_m^* n (\tau^{(k+1)} - \tilde \tau)^2$. Because $p(\tau) \leq 0$ and $p(\tilde \tau) = O_p(1)$, we have $n (\tau^{(k+1)} - \tilde \tau)^2 = O_p(1)$, and part (b) follows.
\end{proof}

The following lemma follows from Le Cam's first and third lemmas and facilitates the derivation of the asymptotic distribution of the LRTS under $\mathbb{P}_{\vartheta_n}^n$. 

\begin{lemma} \label{P-LAN} Suppose that the assumptions of Lemma \ref{P-quadratic} hold and $\bs{\vartheta}_n$ is given by (\ref{local-alternative}). Then, (a) $\mathbb{P}_{\bs{\vartheta}_n}^n$ is mutually contiguous with respect to $\mathbb{P}_{\bs{\vartheta}^*}^n$, and (b) under $\mathbb{P}_{\bs{\vartheta}_n}^n$, we have $\log (d\mathbb{P}_{\bs{\vartheta}_n}^n/d \mathbb{P}_{\bs{\vartheta}^*}^n) = \bs{h}\t \nu_n(\bs{s}(\bs{x},\bs{z})) - \bs{h}\t \bs{\mathcal{I}} \bs{h}/2+ o_{p}(1)$ with $\nu_n(\bs{s}(\bs{x},\bs{z})) \overset{d}{\rightarrow} N(\bs{\mathcal{I}} \bs{h}, \bs{\mathcal{I}})$, where $\bs{s}(\bs{x},\bs{z})$ is defined in (\ref{score_defn}) and $\bs{\mathcal{I}}:=E[\bs{s}(\bs{X},\bs{Z})\bs{s}(\bs{X},\bs{Z})\t]$.
\end{lemma} 
\begin{proof}
Observe that Lemma \ref{Ln_thm1} holds under $\mathbb{P}_{\bs\vartheta^*}^n$ under the assumptions of Lemma \ref{P-quadratic}. Because $\bs{\vartheta}_{n} = (\bs{\eta}_n\t,\bs\lambda_n\t,\alpha_n)\t\in \mathcal{N}_{c/\sqrt{n}}$ by choosing $c> |\bs h|$, it follows from Lemma \ref{Ln_thm1} that
\begin{equation}\label{expansion}
 \left| \log \frac{d\mathbb{P}_{\bs{\vartheta}_n}^n}{d \mathbb{P}_{\bs{\vartheta}^*}^n}- \bs{h}\t \nu_n(\bs{s}(\bs{x},\bs{z})) - \bs{h}\t \bs{\mathcal{I}} \bs{h}/2 \right|=o_{\mathbb{P}_{\bs\vartheta^*}^n}(1).
\end{equation}
Furthermore, $\nu_n(\bs{s}(\bs{x},\bs{z})) \rightarrow_d \bs{G} \sim N(\bs{0},\bs{\mathcal{I}})$ under $\mathbb{P}_{\bs\vartheta^*}^n$. Therefore, $d\mathbb{P}_{\bs\vartheta_n}^n / d \mathbb{P}_{\bs\vartheta^*}^n$ converges in distribution under $\mathbb{P}_{\vartheta^*}^n$ to $\exp\left( N( \mu,\sigma^2) \right)$ with $\mu=-(1/2) \bs{h}\t \bs{\mathcal{I}} \bs{h}$ and $\sigma^2= \bs{h}\t \bs{\mathcal{I}} \bs{h}$, so that $E(\exp\left( N( \mu,\sigma^2) \right))=1$. Consequently, part (a) follows from Le Cam's first lemma (see, e.g., Corollary 12.3.1 of \citet{lehmannromano05book}). Part (b) follows from Le Cam's third lemma (see, e.g., Corollary 12.3.2 of \citet{lehmannromano05book}) because part (a) and (\ref{expansion}) imply that
\[
\begin{pmatrix}
\nu_n(\bs{s}(\bs{x},\bs{z})) \\
\log\frac{d\mathbb{P}_{\bs{\vartheta}_n}^n}{d \mathbb{P}_{\bs{\vartheta}^*}^n}
\end{pmatrix}
 \overset{d}{\rightarrow} 
N\left(
\begin{pmatrix}
0\\
-\frac{1}{2} \bs{h}\t \bs{\mathcal{I}} \bs{h}
\end{pmatrix}, 
\begin{pmatrix}
 \bs{\mathcal{I}}& \bs{\mathcal{I}}\bs{h}\\
\bs{h}\t \bs{\mathcal{I}}&\bs{h}\t \bs{\mathcal{I}} \bs{h}
\end{pmatrix}
\right)\quad\text{under $\mathbb{P}_{\bs{\vartheta}^*}^n$.} 
\] 
\end{proof}

\begin{lemma} \label{lemma_btsp} 
Suppose that the assumptions of Proposition \ref{P-LR-N1} hold. Let $\bf{C}_{\bs\eta}$ be a set of sequences $\{\bs\eta_n\}$ satisfying $\sqrt{n}(\bs\eta_n - \bs\eta^*) \to \bs h_{\bs\eta}$ for some finite $\bs h_{\bs\eta}$. Let $\mathbb{P}^n_{\bs{\eta}_n} := \prod_{i=1}^n g(X_i|Z_i;\bs{\eta}_n,\bs{0},\alpha)$ denote the probability measure under $\bs{\eta}_n$ with $\bs{\lambda}_n=\bs{0}$. Then, for every sequence $\{\bs\eta_n\} \in \bf{C}_{\bs\eta}$, the LRTS under $\{\mathbb{P}^n_{\bs\eta_n}\}$ converges in distribution to $\max_{ j \in \{1,2\}} \left( (\widehat{\bs{t}}_{\bs{\lambda}}^{j})\t \bs{\mathcal{I}}_{\bs{\lambda}.\bs{\eta}} \widehat{\bs{t}}_{\bs{\lambda}}^{j} \right)$ given in Propositions \ref{P-LR-N1}.
\end{lemma}
\begin{proof}
Observe that $\bs{\vartheta}_n:=(\bs\eta_n\t,\bs\lambda_n\t,\alpha_n)\t = ( (\bs\eta^*+\bs{h}_{\bs\eta}/\sqrt{n})\t,\bs{0}\t,\alpha)$ satisfies the assumptions of Lemma \ref{P-LAN}. Therefore, Lemma \ref{P-LAN} holds under $\bs{\vartheta}_n$ with $\nu_n(\bs{s}(\bs{x},\bs{z})) \overset{d}{\rightarrow} N(\bs{\mathcal{I}} \bs{h}, \bs{\mathcal{I}})$ with $\bs{h}=(\bs{h}_{\bs\eta}\t,\bs{0}\t)\t$ under $\mathbb{P}_{\bs\vartheta_n}^n$. Furthermore, the log-likelihood function of the one-component model admits a similar expansion, and $\log (d \mathbb{P}_{\bs\eta_n}^n / d \mathbb{P}_{\bs\eta^*}^n ) = \bs{h}_{\bs\eta}\t \nu_n (\bs{s}_{\bs\eta}(\bs{x},\bs{z})) - (1/2)\bs{h}_{\bs\eta}\t \bs{\mathcal{I}}_{\bs \eta} \bs{h}_{\bs\eta} + o_p(1)$ holds under $\mathbb{P}_{\bs\eta_n}^n$. Therefore, the proof of Proposition \ref{P-LR-N1} goes through by replacing $\bs{G}_{n}$ with $\bs{G}_{n}^{\bs h} = \left[\begin{smallmatrix} \bs{G}_{\bs\eta n}^{\bs h} \\ \bs{G}_{\bs\lambda n}^{\bs h} \end{smallmatrix} \right] := \bs G_{ n} + \bs{\mathcal{I}}\bs h$. In view of $\bs{G}_{\bs \eta n}^{\bs h} = \bs G_{\bs \eta n} + \bs{\mathcal{I}}_{\bs\eta}\bs{h}_{\bs \eta}$ and $\bs{G}_{\bs\lambda n}^{\bs h} = \bs{G}_{\bs\lambda n} + \bs{\mathcal{I}}_{\bs\lambda \bs \eta } \bs h_{\bs\eta}$, we have $\bs G_{\bs \lambda.\bs \eta n}^{\bs h} := \bs G_{\bs \lambda n}^{\bs h} - \bs{\mathcal{I}}_{\bs \lambda \bs \eta}\bs{\mathcal{I}}_{\bs \eta}^{-1} \bs G_{\bs \eta n}^{\bs h} =\bs G_{\bs \lambda n} - \bs{\mathcal{I}}_{\bs \lambda \bs \eta }\bs{\mathcal{I}}_{\bs \eta}^{-1}\bs G_{\bs \eta n} =\bs G_{\bs \lambda n}$. Therefore, the asymptotic distribution of the LRTS under $\mathbb{P}_{\bs \eta_n}^n$ is the same as that under $\mathbb{P}_{\bs \eta^*}^n$, and the stated result follows.
\end{proof}

\begin{lemma}\label{s_der_alpha}
For $h(\bs{x}|\bs{z};\bs{\phi},\bs{\lambda})$ and $\bs{t}(\bs{\phi},\bs{\lambda})$ defined in (\ref{loglike-homo-2}) and (\ref{tpsi_defn-homo-2}), the following holds:
\begin{align*}
\text{(A)} \quad & \nabla_{\alpha} h(\bs{y};\bs{\phi}^*,\bs{\lambda}) = f_v^*(\bs{\lambda}) - f_v^* - \nabla_{ \bs{\mu}\t} f_v^* \bs{\lambda} - \nabla_{ \bs{v}\t} f_v^* \bs{\lambda}_{\bs{\mu}^2},\\
\text{(B)} \quad & \nabla_{\alpha^2} h(\bs{y};\bs{\phi},\bs{\lambda}) = \bs{\xi}(\bs{y};\bs{\vartheta}_2) O(|\bs{\lambda}|^3), 
\end{align*}
where $\sup_{\bs{\vartheta}_2 } |\bs{\xi}(\bs{y};\bs{\vartheta}_2)| \leq \sup_{\bs{\vartheta}_2 } |\bs{v}(\bs{y};\bs{\vartheta}_2)|$ with $\bs{v}(\bs{y};\bs{\vartheta}_2)$ defined in (\ref{v_defn-homo-2}) in the proof of Lemma \ref{P-quadratic-homo-2}, and the domain of $\bs{\vartheta}_2$ is such that $\bs{\phi} \in \Theta_{\bs{\eta}} \times [0, 3/4]$.
\end{lemma}

\begin{proof}
Define
\begin{align*}
f_v^1 &:=f_v\left(\bs{x}\middle|\bs{z};\bs{\gamma}, \bs{\nu}_{\bs\mu}+(1-\alpha)\bs{\lambda}, \bs{\nu}_{\bs{v}} - \alpha(1-\alpha) \bs{w}(\bs{\lambda}\bs{\lambda}\t)\right) , \\
f_v^2 &:= f_v \left(\bs{x}\middle|\bs{z};\bs{\gamma}, \bs{\nu}_{\bs\mu} -\alpha\bs{\lambda},\bs{\nu}_{\bs{v}} - \alpha(1-\alpha) \bs{w}(\bs{\lambda}\bs{\lambda}\t) \right),
\end{align*}
and define $\nabla f_v^1$ and $\nabla f_v^2$ analogously. With this definition, we have $h(\bs{y};\bs{\phi},\bs{\lambda})= \alpha ( f_v^1 - f_v^2 ) + f_v ^2$.
First, we collect the derivatives of $f_v^1$ and $f_v^2$. Noting that $\nabla_\alpha ( - \alpha(1-\alpha) ) = 2\alpha -1$, we obtain, for $j=1,2$,
\begin{equation} \label{g_der_alpha}
\begin{aligned}
\nabla_\alpha f_v^j & = - \nabla_{\bs{\mu}\t} f_v^j \bs{\lambda} + \nabla_{\bs{v}\t} f_v^j (2\alpha-1) \bs{w}(\bs{\lambda}\bs{\lambda}\t), \\
\nabla_{\alpha^2} f_v^j & = \nabla_{(\bs{\mu}^{\otimes 2})\t} f_v^j \bs{\lambda}^{\otimes 2} - 2 \nabla_{(\bs{\mu} \otimes \bs{v})\t} f_v^j (2\alpha -1) (\bs{\lambda} \otimes \bs{w}(\bs{\lambda}\bs{\lambda}\t)) \\
& \quad + \nabla_{(\bs{v}^{\otimes 2})\t} f_v^j (2\alpha -1)^2 \bs{w}(\bs{\lambda}\bs{\lambda}\t)^{\otimes 2} + \nabla_{\bs{v}\t} f_v^j 2 \bs{w}(\bs{\lambda}\bs{\lambda}\t).
\end{aligned}
\end{equation}

Part (A) follows from differentiating $h(\bs{y};\bs{\phi},\bs{\lambda})= \alpha ( f_v^1 - f_v^2 ) + f_v ^2$ with respect to $\alpha$, applying (\ref{g_der_alpha}), evaluating it at $(\bs{\eta} = \bs{\eta}^*, \alpha=0)$, and noting that $\bs{w}(\bs{\lambda}\bs{\lambda}\t) = \bs{\lambda}_{\bs{\mu}^2}$.

For part (B), expanding $\nabla_{\alpha^2} h(\bs{y};\bs{\phi},\bs{\lambda})$ around $\alpha=0$ gives
\begin{equation} \label{h_del_alpha_3}
\nabla_{\alpha^2} h(\bs{y};\bs{\phi},\bs{\lambda}) = \nabla_{\alpha^2} h(\bs{y};(\bs{\eta}\t,0)\t,\bs{\lambda}) + \nabla_{\alpha^3} h(\bs{y};(\bs{\eta}\t,\bar{\alpha})\t,\bs{\lambda}) \alpha. 
\end{equation}
Define $f_v(\bs{\lambda}):=f_v \left(\bs{x}\middle|\bs{z};\bs{\gamma}, \bs{\nu}_{\bs{\mu}} + \bs{\lambda},\bs{\nu}_{\bs{v}} \right)$. For the first term on the right hand side of (\ref{h_del_alpha_3}), a direct calculation and $2 \nabla_{\bs{v}\t} f_v \left( \bs{\lambda} \right) \bs{w}(\bs{\lambda}\bs{\lambda}\t) = \nabla_{(\bs{\mu}^{\otimes 2})\t} f_v \left( \bs{\lambda} \right) \bs{\lambda}^{\otimes 2}$ gives
\begin{align*}
& \nabla_{\alpha^2} h(\bs{y};(\bs{\eta}\t,0)\t,\bs{\lambda}) \\
& = - 2 \left[ \nabla_{\bs{\mu}\t} f_v \left(\bs{\lambda} \right) \bs{\lambda} - \nabla_{\bs{\mu}\t} f_v \left(\bs{0}\right) \bs{\lambda} -\nabla_{(\bs{\mu}^{\otimes 2})\t} f_v \left( \bs{0} \right) \bs{\lambda}^{\otimes 2} \right] - 2 \left[ \nabla_{\bs{v}\t} f_v \left(\bs{\lambda} \right) \bs{\lambda}_{\bs{\mu}^2} - \nabla_{\bs{v}\t} f_v \left(\bs{0} \right) \bs{\lambda}_{\bs{\mu}^2} \right] \\
& \quad - 2 \nabla_{(\bs{\mu} \otimes \bs{v})\t} f_v \left(\bs{0}\right) (\bs{\lambda} \otimes \bs{w}(\bs{\lambda}\bs{\lambda}\t)) + \nabla_{(\bs{v}^{\otimes 2})\t} f_v \left(\bs{0}\right) \bs{w}(\bs{\lambda}\bs{\lambda}\t)^{\otimes 2}.
\end{align*}
Applying a Taylor expansion to the terms in the brackets, the right hand side is written as $\nabla_{(\bs{\mu}^{\otimes 3})\t} f_v(\overline{\bs{\lambda}})O(|\bs{\lambda}|^3) + \nabla_{(\bs{\mu}^{\otimes 3})\t} f_v(\bs{\lambda})O(|\bs{\lambda}|^3)+ \nabla_{(\bs{\mu}^{\otimes 4})\t} f_v(\bs{\lambda})O(|\bs{\lambda}|^4)$ with $\overline{\bs{\lambda}} \in (\bs{0},\bs{\lambda})$. Finally, it follows from a direct calculation in conjunction with (\ref{g_der_alpha}) that $\nabla_{\alpha^3} h(\bs{y};\bs{\psi},\bs{\lambda})$ is bounded by the product of the derivatives of $f_v \left(\bs{x}\middle|\bs{z};\bs{\gamma}, \bs{\mu},\bs{v} \right)$ and an $O(|\bs{\lambda}|^3)$ term, and the required result follows.
\end{proof}

\singlespace{ 
\bibliography{/Users/shimotsu/Dropbox/mv_normal/mvnormal}
}

\clearpage

\begin{table}[h]\caption{Type I errors (\%) of the EM test of $H_0:M=1$} \label{table1}
\centering
\small{
\begin{tabular}{cc|ccc|ccc}
\hline \hline
&& \multicolumn{3}{c|}{$n=200$} & \multicolumn{3}{c}{$n=400$} \\
&Level & $K=1$ & $K=2$ & $K=3$ &$K=1$ & $K=2$ & $K=3$  \\
\hline 
Model 1&10\% &9.95&10.15&10.20&9.65&9.85&9.80  \\ 
$a_n = n^{-1/2}$&5\% &4.05&3.75&4.10&5.35&5.35&5.45  \\
&1\%  &0.85&0.80&0.60&0.90&0.90&1.00
\\ \hline
Model 2&10\% &9.10&8.85&8.95&9.55&9.70&9.70  \\
$a_n = n^{-1/2}$&5\% & 4.35&4.25&4.35&4.50&4.60&4.40   \\
&1\% & 0.40&0.40&0.45&1.10&0.95&0.90 \\ \hline  \hline 
Model 1&10\% &8.30&8.25&8.25&8.95&8.95&8.90\\
$a_n = 1$&5\% &3.95&3.90&3.90&4.45&4.60&4.60\\
&1\%  &0.40&0.40&0.40&0.90&0.90&0.85\\\hline
Model 2&10\% &8.20&8.20&8.10&8.95&9.00&9.00\\
$a_n = 1$&5\% &3.75&3.70&3.75&4.60&4.65&4.70\\
&1\%  &0.60&0.60&0.60&1.20&1.20&1.20\\\hline
\end{tabular}
} \\ \smallskip
\begin{flushleft} 
Notes: Based on 2000 replications with 399 bootstrapped samples.  Model 1 is $\bs{\mu} = \begin{pmatrix} 0 \\ 0 \end{pmatrix}$,\ $\bs{\Sigma}=\begin{pmatrix} 1 & 0 \\ 0 & 1 \end{pmatrix}$.
Model 2 is $\bs{\mu} = \begin{pmatrix} 0 \\ 0 \end{pmatrix}$,\ $\bs{\Sigma}=\begin{pmatrix} 1 & 0.5 \\ 0.5 & 1 \end{pmatrix}$.
\end{flushleft}
\end{table}

\begin{table}[h]\caption{Parameter specifications for testing the power of the EM test of $H_0:M=1$} \label{table2}
\centering
\small{
\begin{tabular}{c|ccccc}
\hline \hline
& $\bs{\alpha}$ & $\bs{\mu}_1$ & $\bs{\mu}_2$ &  $\bs{\Sigma}_1$ & $\bs{\Sigma}_2$  \\ \hline
Model 1&  $\begin{pmatrix}  0.3 \\  0.7 \end{pmatrix}$ &$\begin{pmatrix} -0.5 \\ -0.5 \end{pmatrix}$ & $\begin{pmatrix}0.5 \\ 0.5 \end{pmatrix}$ &     $\begin{pmatrix} 1 & 0 \\ 0 & 1 \end{pmatrix}$ & $\begin{pmatrix} 1 & 0 \\ 0 & 1 \end{pmatrix}$  \\
Model 2& $\begin{pmatrix}  0.3 \\  0.7 \end{pmatrix}$ &$ \begin{pmatrix} -1 \\ -1 \end{pmatrix}$& $ \begin{pmatrix} 1 \\ 1 \end{pmatrix}$&  $ \begin{pmatrix} 1 & 0 \\ 0 & 1 \end{pmatrix}$&  $\begin{pmatrix} 1 & 0 \\ 0 & 1 \end{pmatrix}$\\
Model 3&$\begin{pmatrix}  0.3 \\  0.7 \end{pmatrix}$ & $\begin{pmatrix} -0.5 \\ -0.5 \end{pmatrix}$& $ \begin{pmatrix} 0.5 \\ 0.5 \end{pmatrix}$&  $ \begin{pmatrix} 2 & 0 \\ 0 & 2 \end{pmatrix}$&  $\begin{pmatrix} 1 & 0 \\ 0 & 1 \end{pmatrix}$\\  \hline
\end{tabular}
}
\end{table}

\begin{table}[h]\caption{Powers (\%) of the EM test of $H_0:M=1$} \label{table3}
\centering
\small{
\begin{tabular}{cc|ccc|ccc}
\hline \hline
&& \multicolumn{3}{c|}{$n=200$} & \multicolumn{3}{c}{$n=400$} \\
&Level & $K=1$ & $K=2$ & $K=3$ &$K=1$ & $K=2$ & $K=3$  \\ \hline
Model 1&10\% &13.20&13.20&13.25&18.95&19.00&19.10   \\
&5\%  &7.05&7.05&7.10&10.85&10.75&10.80  \\
&1\%  &1.25&1.25&1.25&2.50&2.50&2.50   \\\hline
Model 2&10\% &97.25&97.25&97.25&99.95&99.95&99.95   \\
&5\%  &94.55&94.50&94.50&99.90&99.90&99.90  \\
&1\%  &81.70&81.65&81.80&99.80&99.80&99.80    \\\hline
Model 3 &10\% &36.05&36.10&36.00&60.85&60.70&60.65   \\
&5\%  &23.15&23.20&23.35&48.35&48.35&48.30  \\
&1\%  &8.25&8.30&8.20&25.40&25.25&25.20   \\  \hline
\end{tabular}
}\\ \smallskip.\\
\begin{flushleft} 
Notes: Based on 2000 replications with 399 bootstrapped samples. We set $a_n = 1$. Model 1, 2, and 3 are given in Table \ref{table2}. 
\end{flushleft}
\end{table}

\begin{table}[h]\caption{Type I errors (\%) of the EM test of $H_0:M=2$}\label{table4}
\centering
\small{
\begin{tabular}{cc|ccc|ccc}
\hline \hline
&& \multicolumn{3}{c|}{$n=200$} & \multicolumn{3}{c}{$n=400$} \\
&Level & $K=1$ & $K=2$ & $K=3$ &$K=1$ & $K=2$ & $K=3$  \\  
\hline
Model 1&10\% &10.1&9.9&9.9&8.5&8.4&8.6 \\
$a_n = n^{-1/2}$&5\%  &6.4&6.1&5.9&3.7&3.8&3.7 \\
&1\%  &0.8&0.6&0.7&0.9&0.9&0.9\\\hline
Model 2&10\% &9.1&9.2&9.3&10.6&11.1&10.8 \\
$a_n = n^{-1/2}$&5\%  &4.2&4.3&4.2&5.1&5.2&5.3 \\
&1 \% &0.9&1.0&1.0&0.3&0.3&0.4\\
\hline\hline  
Model 1&10\% &8.5&8.3&8.5&9.4&9.4&9.5 \\
$a_n = 1$&5\%  &3.6&3.6&3.5&3.5&3.5&3.5  \\
&1\%  &1.0&1.0&0.9&0.8&0.8&0.7 \\\hline
Model 2&10\% &9.2&9.1&9.0&10.3&10.3&10.4 \\
$a_n = 1$&5\%  &4.0&4.1&4.0&5.2&5.1&4.9 \\
&1 \% &0.5&0.4&0.4&1.0&1.0&1.0 \\
\hline
\end{tabular}
}\\ \smallskip
\begin{flushleft} 
Based on 1000 replications with 199 bootstrapped samples.\\ 
Model 1:  $\bs{\alpha} = \begin{pmatrix} 0.7 \\ 0.3 \end{pmatrix}$,\ $\bs{\mu}_1 = \begin{pmatrix} -1 \\ -1 \end{pmatrix}$,\ $\bs{\mu}_2 = \begin{pmatrix} 1 \\ 1 \end{pmatrix}$,\  $\bs{\Sigma}_1 = \begin{pmatrix} 1 & 0 \\ 0 & 1 \end{pmatrix}$,\  $\bs{\Sigma}_2 = \begin{pmatrix} 1 & 0 \\ 0 & 1 \end{pmatrix}$.\\
Model 2: $\bs{\alpha} = \begin{pmatrix} 0.7 \\ 0.3 \end{pmatrix}$,\ $\bs{\mu}_1 = \begin{pmatrix} -2 \\ -2 \end{pmatrix}$,\ $\bs{\mu}_2 = \begin{pmatrix} 2 \\ 2 \end{pmatrix}$,\  $\bs{\Sigma}_1 = \begin{pmatrix} 1 & 0 \\ 0 & 1 \end{pmatrix}$,\  $\bs{\Sigma}_2 = \begin{pmatrix} 1 & 0 \\ 0 & 1 \end{pmatrix}$.
\end{flushleft}
\end{table}

\begin{table}[h]\caption{Parameter specifications for testing the power of the EM test of $H_0:M=2$} \label{table5}
\centering
\small{
\begin{tabular}{cccccccc}
\hline \hline
& $\bs{\alpha}$ & $\bs{\mu}_1$ & $\bs{\mu}_2$ & $\bs{\mu}_3$ & $\bs{\Sigma}_1$ & $\bs{\Sigma}_2$ & $\bs{\Sigma}_3$ \\ \hline
Model 1& $\begin{pmatrix}0.35 \\ 0.35 \\ 0.3 \end{pmatrix}$ & $\begin{pmatrix} -2 \\ -2 \end{pmatrix}$ & $\begin{pmatrix} 0 \\ 0 \end{pmatrix}$ &  $\begin{pmatrix} 2 \\ 2 \end{pmatrix}$ &  $\begin{pmatrix} 1 & 0 \\ 0 & 1 \end{pmatrix}$ & $\begin{pmatrix} 1 & 0 \\ 0 & 1 \end{pmatrix}$ & $\begin{pmatrix} 1 & 0 \\ 0 & 1 \end{pmatrix}$ \\
Model 2& $\begin{pmatrix}0.35 \\ 0.35 \\ 0.3\end{pmatrix}$ & $\begin{pmatrix} -2 \\ -2 \end{pmatrix}$ & $\begin{pmatrix} 0 \\ 0 \end{pmatrix}$ & $\begin{pmatrix} 2 \\ 2 \end{pmatrix}$ &  $\begin{pmatrix} 0.5 & 0 \\ 0 & 0.5 \end{pmatrix}$ & $\begin{pmatrix} 1 & 0 \\ 0 & 1 \end{pmatrix}$ & $\begin{pmatrix} 2 & 0 \\ 0 & 2 \end{pmatrix}$ \\ \hline
\end{tabular}
}
\end{table}

\begin{table}[h]\caption{Powers (\%) of the EM test of $H_0:M=2$} \label{table6}
\centering
\small{
\begin{tabular}{cc|ccc|ccc}
\hline \hline
&& \multicolumn{3}{c|}{$n=200$} & \multicolumn{3}{c}{$n=400$} \\
&Level & $K=1$ & $K=2$ & $K=3$ &$K=1$ & $K=2$ & $K=3$  \\  
\hline 
Model 1&10\% &28.0&28.4&28.9&80.2&80.2&80.3 \\
&5\%  &16.3&16.8&17.0&69.1&69.3&69.5   \\
&1\%  &4.9&5.1&5.4&41.0&41.2&41.3  \\ \hline
Model 2&10\% &58.9&59.3&59.3&94.2&93.8&94.2  \\
 &5\%  &45.1&46.0&46.4&90.5&90.6&90.7 \\
&1 \% &19.8&20.7&21.6&72.0&71.9&72.0 \\
\hline
\end{tabular}
}\\ \smallskip
Based on 1000 replications with 199 bootstrapped samples. Model 1 and 2 are given in Table \ref{table5}. We set $a_n=1$.\\ 
\end{table}

\begin{figure}[h] 
        \centering
        \caption{ Scatter plot of two physical measurements of   flea beetles} \label{figure1}\vspace{-0.4cm}
        \includegraphics[width=0.8\textwidth]{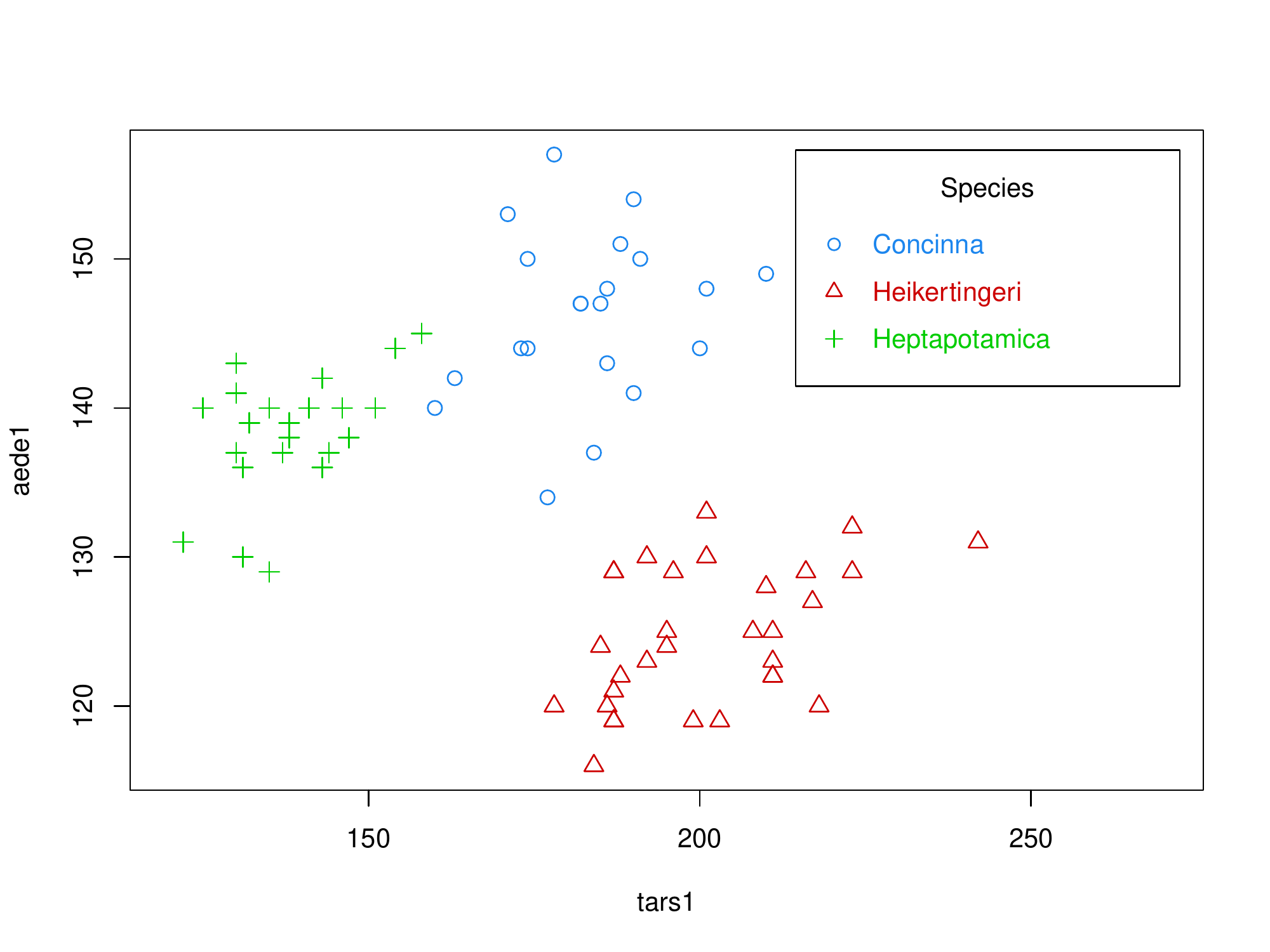}  
\end{figure}

\begin{table}[h]\caption{$P$-values of the EM test  in the flea beetles data} \label{table7}
\centering
\small{
\begin{tabular}{c|ccc|cc}
\hline \hline
&  \multicolumn{3}{c|}{p-values} & & \\  
  & $K=1$ & $K=2$ & $K=3$ & AIC & BIC\\ \hline
$H_0: M=1$ & 0.000 & 0.000 &0.000&\textbf{1129.5}&\textbf{1141.0}\\
$H_0:M=2$ & 0.010& 0.005& 0.005&1202.8&1228.1\\
$H_0:M=3$ &\textbf{0.337}&\textbf{0.347}&\textbf{0.362}&1191.8&1230.9\\
$H_0:M=4$ &0.588&0.633&0.618&1197.1&1250.1\\ \hline 
\hline
\end{tabular}
}\\ \smallskip
Notes: Based on  199 bootstrapped samples.  \\ 
\end{table}

\begin{table}[h]\caption{Parameter estimates from the three-component model for the flea beetles data} \label{table8}
\centering
\small{
\begin{tabular}{cccccccc}
\hline \hline
& $\widehat{\bs{\alpha}}$ & $\widehat{\bs{\mu}}_1$ & $\widehat{\bs{\mu}}_2$ & $\widehat{\bs{\mu}}_3$ & $\widehat{\bs{\Sigma}}_1$ & $\widehat{\bs{\Sigma}}_2$ & $\widehat{\bs{\Sigma}}_3$ \\ \hline
Mixture & $\begin{pmatrix} 0.312\\0.270\\0.418	\end{pmatrix}$ & $\begin{pmatrix}   139.4\\
138.3 \end{pmatrix}$ & $\begin{pmatrix}  184.3\\
146.5 \end{pmatrix}$ &  $\begin{pmatrix}  201.0\\
124.6\end{pmatrix}$ &  $\begin{pmatrix} 114.0& 18.5\\ 18.5& 16.5  \end{pmatrix}$ & $\begin{pmatrix}  134.6&
4.9\\ 4.9& 31.5  \end{pmatrix}$ & $\begin{pmatrix}  221.2& 28.0\\28.0&
21.4\end{pmatrix}$
\\ \hline\hline
& $ {\bs{\alpha}}$ & $\widehat{\bs{\mu}}_{\text{Hep}}$ & $\widehat{\bs{\mu}}_{\text{Con}}$ & $\widehat{\bs{\mu}}_{\text{Hei}}$ & $\widehat{\bs{\Sigma}}_{\text{Hep}}$ & $\widehat{\bs{\Sigma}}_{\text{Con}}$ & $\widehat{\bs{\Sigma}}_{\text{Hei}}$ \\ \hline
By species & $\begin{pmatrix}  0.297\\ 0.284\\	0.419 \end{pmatrix}$ & $\begin{pmatrix} 138.2\\ 138.3  \end{pmatrix}$ & $\begin{pmatrix}  183.1\\ 146.2 \end{pmatrix}$ &  $\begin{pmatrix}  201.0\\ 124.6 \end{pmatrix}$ &  $\begin{pmatrix}  83.4&
18.3\\ 18.3& 16.4
\end{pmatrix}$ & $\begin{pmatrix}  140.5& 14.4\\ 14.4& 30.2
\end{pmatrix}$ & $\begin{pmatrix} 215.0& 29.4\\29.4& 20.7 \end{pmatrix}$ \\ \hline \hline
\end{tabular}
}
Notes: $\widehat{\bs{\mu}}_{j}$  and $\widehat{\bs{\Sigma}}_{j}$ for $j\in \{\text{Hep}, \text{Con}, \text{Hei}\}$ reports the mean and the variance estimated from a subsample of observations that belong  to   ``Heptapotamica,'' ``Concinna," and ``Heikertingeri," respectively.
\end{table}

\begin{table}[h]\caption{$P$-values of the EM test for the rat data} \label{table9}
\centering
\small{
\begin{tabular}{c|ccc|cc}
\hline \hline
&  \multicolumn{3}{c|}{$p$-values} & & \\ 
  & $K=1$ & $K=2$ & $K=3$ & AIC & BIC\\ \hline
$H_0: M=1$ &0.000&0.000&0.000&-4935.9&-4910.5\\
$H_0: M=2$ &0.000&0.000&0.000&-6413.9&-6358.2\\
$H_0: M=3$ &0.000&0.000&0.000&-6812.3&-6726.1\\
$H_0: M=4$ &0.000&0.000&0.000&-6932.9&-6816.3\\
$H_0: M=5$ &0.000&0.000&0.000&-6963.7&\tbf{-6816.6}\\
$H_0: M=6$ &\tbf{0.101}&\tbf{0.101}&\tbf{0.101}&\tbf{-6978.8}&-6801.3\\ \hline  
\hline
\end{tabular}
}\\ \smallskip
Notes: Based on 199 bootstrapped samples.  \\ 
\end{table}

\begin{table}[h]\caption{Parameter estimates from the six-component model for the rat data} \label{table10}
\centering
\footnotesize{
\begin{tabular}{cccccc}
\hline \hline
$\widehat{\alpha}_1$ & $\widehat{\alpha}_2$ & $\widehat{\alpha}_3$\\
0.530 & 0.228 & 0.132  \\  \hline
$\widehat{\bs{\mu}}_1$ & $\widehat{\bs{\mu}}_2$ & $\widehat{\bs{\mu}}_3$\\
$\begin{pmatrix}
-0.014 & -0.029
 \end{pmatrix}$ & $\begin{pmatrix}
-0.004 & 0.019
\end{pmatrix}$ & $\begin{pmatrix}
 0.039 & 0.095
\end{pmatrix}$   \\ \hline 
$\widehat{\bs{\Sigma}}_1$ & $\widehat{\bs{\Sigma}}_2$ & $\widehat{\bs{\Sigma}}_3$ \\
$\begin{pmatrix}
0.0010 & 0.0000 \\ 0.0000 & 0.0005 
\end{pmatrix}$ & $\begin{pmatrix}
0.0011 & 0.0005 \\  0.0005 & 0.0009 
\end{pmatrix}$ & $\begin{pmatrix}
0.0026 & 0.0005 \\  0.0005 & 0.0029  
\end{pmatrix}$ \\ \hline
$\widehat{\alpha}_4$ & $\widehat{\alpha}_5$ & $\widehat{\alpha}_6$\\
0.068 & 0.028 & 0.013 \\  \hline
$\widehat{\bs{\mu}}_4$ & $\widehat{\bs{\mu}}_5$& $\widehat{\bs{\mu}}_6$ \\  
$\begin{pmatrix}
 0.138 & 0.261
\end{pmatrix}$ &  $\begin{pmatrix}
 0.404 & 0.636
\end{pmatrix}$  &  $\begin{pmatrix}
 1.215 & 1.420
\end{pmatrix}$   \\ \hline 
$\widehat{\bs{\Sigma}}_4$ & $\widehat{\bs{\Sigma}}_5$& $\widehat{\bs{\Sigma}}_6$\\  
$\begin{pmatrix}
0.0100 & 0.0010 \\  0.0010 & 0.0142  
\end{pmatrix}$ & $\begin{pmatrix}
0.0274 & -0.0100 \\ -0.0100 & 0.0588
\end{pmatrix}$ & $\begin{pmatrix}
0.2289 & 0.1671 \\ 0.1671 & 0.2423  
\end{pmatrix}$ \\\hline\hline \\
\end{tabular}
} 
\end{table}
\bigskip

\begin{figure}[h] 
        \centering
        \caption{ Scatter plot of gene expression levels for the rats with and without middle-ear infection}\label{figure2}\vspace{-0.4cm}
        \includegraphics[width=0.8\textwidth]{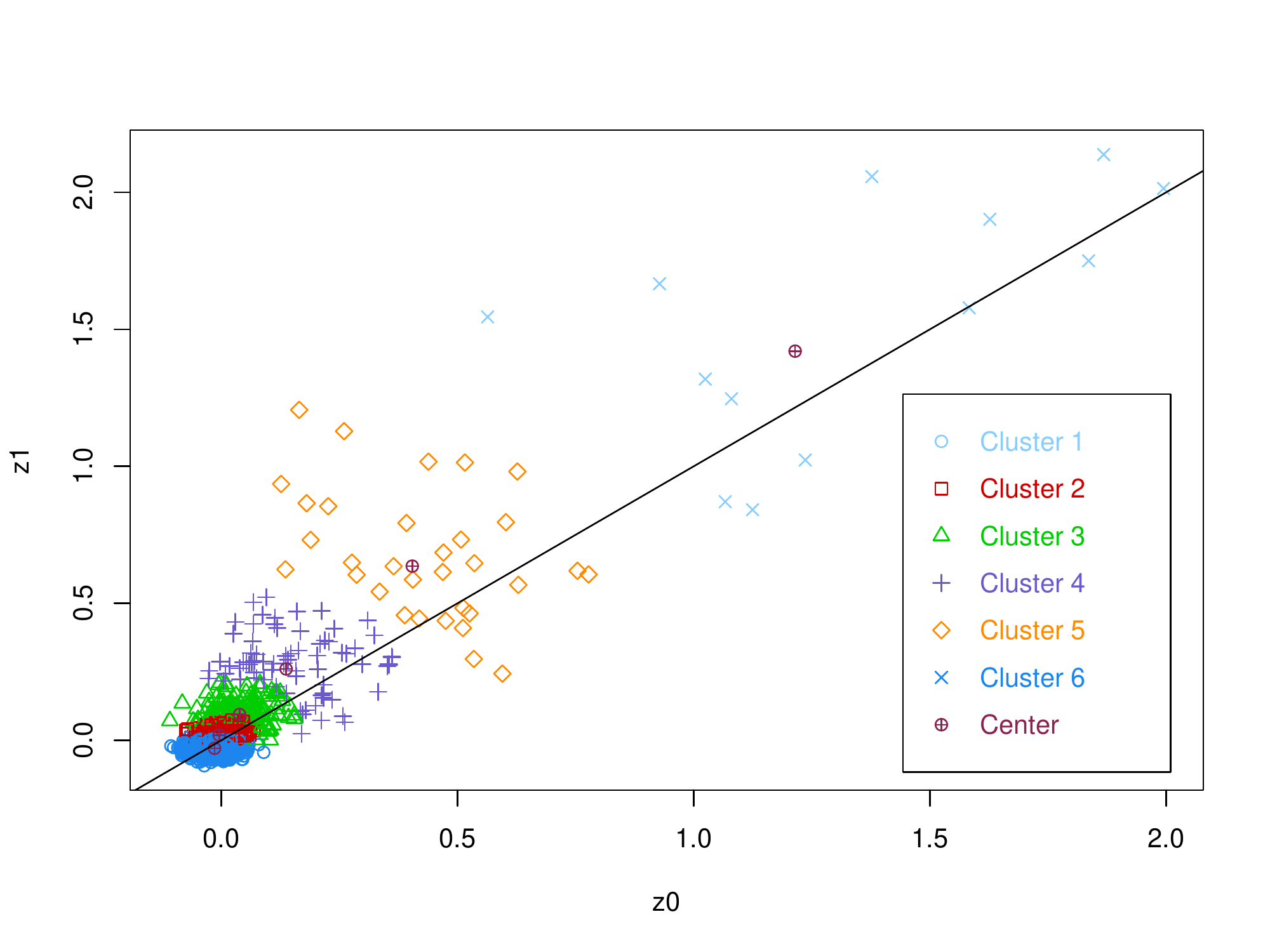}  
\end{figure}

\end{document}